\documentclass[11pt, reqno, a4paper]{amsart}

\usepackage{fullpage}

\usepackage{amssymb, amsmath, amscd}
\usepackage{amsthm}

\usepackage{verbatim}

\usepackage{color}

\usepackage{etex}
\usepackage{dsfont}
\usepackage[matrix, arrow, curve, frame, arc]{xy}
\usepackage{pgf}
\usepackage{tikz-cd}
\usetikzlibrary{arrows,shapes,automata,backgrounds,petri,calc}
\usepackage[square,sort,comma,numbers]{natbib}
 \usepackage{url}
 \usepackage{hyperref}
 \usepackage[shortlabels]{enumitem} 

\newcommand{\tikzAngleOfLine}{\tikz@AngleOfLine}
\def\tikz@AngleOfLine(#1)(#2)#3{%
\pgfmathanglebetweenpoints{%
\pgfpointanchor{#1}{center}}{%
\pgfpointanchor{#2}{center}}
\pgfmathsetmacro{#3}{\pgfmathresult}%
}

\newcommand{\cO}{\mathcal{O}}
\newcommand{\bN}{\mathbb{N}}

\newcommand{\cB}{\mathcal{B}}
\newcommand{\wt}{\widetilde}
\newcommand{\ba}{\bar{\alpha}}
\newcommand{\La}{\Lambda}
\newcommand{\vf}{\varphi}
\newcommand{\ve}{\varepsilon}

\newcommand{\cP}{\mathcal{P}}

\newcommand{\bD}{\mathbb{D}}
\newcommand{\bE}{\mathbb{E}}

\newcommand{\rad}{\operatorname{rad}}
\newcommand{\soc}{\operatorname{soc}}
\renewcommand{\mod}{\operatorname{mod}}
\newcommand{\Hom}{\operatorname{Hom}} 
\newcommand{\ima}{\operatorname{im}}
\newcommand{\alg}{\operatorname{alg}} 
\newcommand{\opo}{\operatorname{op}}

\def\vec#1{\left[\begin{smallmatrix}#1\end{smallmatrix}\right]}

\begin{document}

\newtheorem{defi}{Definition}[section]
\newtheorem{rem}[defi]{Remark}
\newtheorem{prop}[defi]{Proposition}
\newtheorem{ques}[defi]{Question}
\newtheorem{lemma}[defi]{Lemma}
\newtheorem{cor}[defi]{Corollary}
\newtheorem{thm}[defi]{Theorem}
\newtheorem{expl}[defi]{Example} 
\newtheorem*{mthm}{Main Theorem}

\parindent0pt

\title[Period four]{Algebras of generalized quaternion type: biregular case$^*$ \footnote{\tiny $^*$ This research has been 
supported from the grant no. 2023/51/D/ST1/01214 of the Polish National Science Center}} 

\author[K. Erdmann]{Karin Erdmann}
\address[Karin Erdmann]{Mathematical Institute, University of Oxford, ROQ, Oxford OX2 6GG, United Kingdom} 
\email{erdmann@maths.ox.ac.uk}

\author[A. Hajduk]{Adam Hajduk}
\address[Adam Hajduk]{Faculty of Mathematics and Computer Science, 
Nicolaus Copernicus University, Chopina 12/18, 87-100 Toru\'n, Poland}
\email{ahajduk@mat.umk.pl} 

\author[A. Skowyrski]{Adam Skowyrski}
\address[Adam Skowyrski]{Faculty of Mathematics and Computer Science, 
Nicolaus Copernicus University, Chopina 12/18, 87-100 Toru\'n, Poland}
\email{skowyr@mat.umk.pl}


\begin{abstract} This paper provides the next step towards classification of all algebras of generalized 
quaternion type. We extend the results of \cite{AGQT}, where such classification has been obtained in the case, 
when the Gabriel quiver of an algebra is $2$-regular (meaning two arrows start and end at each vertex). 
A natural generalization motivated by existing examples of weighted surface algebras leads to the study of 
algebras of generalized quaternion type with Gabriel quivers being biregular, i.e. consisting of $1$-regular 
and $2$-regular vertices. In the main result, we prove that any such an algebra is (up to socle equivalence) 
either a weighted surface algebra \cite{WSA} or the so called higher spherical algebra \cite{HSS}. \end{abstract}

\maketitle

\section{Introduction}\label{sec:1} 
In this paper, an algebra is a finite-dimensional, basic, indecomposable $K$-algebra over a fixed 
algebraically closed field $K$. For an algebra $\La$, we denote by $\mod\La$ its module category, consisting of 
finitely generated right $\La$-modules, and by $D$ the standard duality $\Hom_K(-,K)$. Any algebra admits 
a presentation by quiver and relations, that is, we have an isomorphism $\La=KQ/I$, where $KQ$ is the 
path algebra of a quiver $Q$, and $I$ is an admissible ideal of $KQ$ (for details see Section \ref{sec:2}). 
The quiver $Q$ is uniquely determined (up to permutation of vertices), it is called the Gabriel quiver 
of $\La$, and denoted by $Q_\La$. 

Recall that an algebra 
$\La$ is called self-injective, provided that $\La$ is injective module in $\mod\La$. We will focus on the 
symmetric algebras, that is self-injective algebras $\La$, for which there exists an associative, non-degenerate 
symmetric $K$-bilinear form $(-,-):\La\times\La\to K$, or equivalently, $\La$ and $D(\La)$ are isomorphic as 
$\La$-bimodules (see \cite[Theorem IV.2.2]{SkY}). By the remarkable Tame and Wild theorem \cite{Dro} (see also \cite{CB}), 
every algebra is either tame or wild, where by a tame algebra we mean an algebra $\La$, whose indecomposable modules 
in $\mod\La$ occur (in each dimension) in a finite number of discrete and a finite number of one-parameter families. 
Note that any representation-finite algebra, i.e. algebra $\La$ with finitely many isoclasses of indecomposable 
modules in $\mod\La$ is tame. \smallskip 

We are especially interested in the class of periodic algebras. We recall that a module $M$ in $\mod\La$ is 
periodic, when it is periodic with respect to syzygy, i.e. $\Omega_\La^d(M)\simeq M$, for some $d\geqslant 1$, 
where $\Omega_\La$ is the syzygy operator, associating to a module $X$ the kernel of its projective cover in 
$\mod\La$ (see Section \ref{sec:2}). By a periodic algebra we mean an algebra $\La$, such that $\La$ is periodic 
as a $\La$-bimodule, equivalently, as a module over the enveloping algebra $\La^e=\La^{\opo}\otimes_K\La$. \medskip 

Our main concern is the classification of all tame symmetric periodic algebras of period four (TSP4 algebras), 
or slightly more generally, all the so called algebras of generalized quaternion type (GQT algebras, for short). 
Following \cite{AGQT}, an algebra $\La$ is of generalized quaternion type if and only if $\La$ is tame, symmetric 
of infinite representation type, and every simple module in $\mod\La$ is periodic of period $4$. We note that any TSP4 
algebra of infinite representation type is automatically a GQT algebra, but the converse is an open question (see 
\cite{perconj}; see also \cite{note}). For a relevant background and motivations related to TSP4 
(or GQT) algebras we refer to the introductions of papers \cite{WSA,AGQT}. \smallskip 

The main class of algebras involved in the classification are the so called weighted 
surface algebras (WSA's), introduced and investigated in \cite{WSA,WSA-GV,WSA-corr}. Note that any weighted 
surface algebra $\La=KQ/I$ is determined by a triangulation quiver $Q$ (coming from a surface) and a set of 
weights and parameters defining the ideal $I$. We skip the details of the construction (more in Section 
\ref{sec:HSA}), we only mention that most WSA's are TSP4 (and GQT), except four peculiar families, consisted of the 
so called exceptional algebras. \smallskip 

We recall that the weighted surface algebras not only provide a variety of examples of TSP4, but 
also exhaust all TSP4 (or GQT) algebras in some cases. It is conjectured in general, and confirmed already 
for $2$-regular algebras in \cite[see Main Theorem]{AGQT}, where by a $2$-regular algebra we mean an algebra 
$\La$ with $2$-regular Gabriel quiver $Q_\La$, that is, every vertex of $Q_\La$ is $2$-regular, i.e. two arrows 
start and two arrows end at the vertex. Namely, it has been shown in \cite{AGQT} that for any $2$-regular algebra $\La$ 
(with $Q_\La$ having at least three vertices), $\La$ is TSP4 if and only if $\La$ is GQT if and only if $\La$ is isomorphic 
to a weighted surface algebra (up to socle equivalence) different from the exceptional algebras, or to the so called 
higher tetrahedral algebra (see \cite{HTA}). The higher tetrahedral algebras form an exotic family given by a triangulation 
quiver $Q_\La$, but not isomorphic to WSA's. The main result of this paper is an extension of this classification to 
the biregular case, and we will see that there is an analogous division of the algebras into the weigthed surface ones 
and remaining `higher' versions, which are closely related. \medskip 

The main result of \cite{AGQT} deals with $2$-regular algebras, for which $Q=Q_\La$ is a triangulation quiver. 
This gives a part of WSA's, for which the quiver $Q$ does not contain the so called virtual arrows (see Section 
\ref{sec:HSA}). General WSA \cite{WSA-GV} is of the form $\La=KQ/I$, where possibly $Q\neq Q_\La$ contains virtual arrows, not 
contained in the Gabriel quiver of $\La$. In this case, the Gabriel quiver $Q_\La$ consists of $2$-regular vertices 
and $1$-regular vertices (one arrow starts and ends at the vertex), which means that $Q_\La$ is biregular. Moreover, 
the $1$-regular vertices are contained in two types of blocks of the following form 
$$
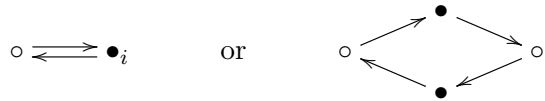
\xymatrix@R=.7ex{&&& &\bullet\ar[rd]& \\ 
\circ\ar@<+.4ex>[r]&\bullet_i\ar@<+.4ex>[l]&\mbox{ or }& \circ\ar[ru] && \circ\ar[ld]  \\ 
&&& &\ar[lu]\bullet&}$$
denoted by $V_1$ and $V_2$, respectively (see also the introduction of \cite{EHS2}); note that the vertices 
marked by $\bullet$ are $1$-regular, while $\circ$ are $2$-regular vertices of $Q_\La$. The property was confirmed 
in general, i.e. for arbitrary GQT algebras with biregular Gabriel quiver (see Theorem \ref{prop:4.3}), which was 
the main result of a preparatory paper \cite{EHS2}. A natural step forward was to generalize the classification 
for biregular algebras. It led to the following theorem, which is the main result of this article.

\begin{mthm} Let $\La$ be an algebra whose Gabriel quiver is biregular and has at least three vertices. 
Then the following conditions are equivalent. 
\begin{enumerate}
\item[(i)] $\Lambda$ is a TSP4 algebra. 
\item[(ii)] $\La$ is a GQT algebra.   
\item[(iii)] $\La$ is isomorphic to a weighted surface algebra different from an exceptional algebra, or 
it is isomorphic to a higher spherical algebra or a higher tetrahedral algebra. \end{enumerate} 
\end{mthm} \medskip 

The paper is organized as follows. We start with recalling some basic notions in Section \ref{sec:2} and then 
we discuss the central class of algebras we are dealing with, that is, the class of weighted surface algebras 
(Section \ref{sec:HSA}). Further in Section \ref{sec:3}, we recall known facts concerning arbitrary GQT algebras, 
and investigate in greater details the shape of relations around $1$-vertices in biregular case. Section \ref{bases} 
is the most technical part of the article, where we give a detailed description of all other relations, i.e. 
relations near the `$2$-regular part' of the Gabriel quiver, which is then used (together with partial results 
from Section \ref{sec:3}) to compute bases of indecomposable projective $\La$-modules over arbitrary GQT 
algebra $\La$ with biregular Gabriel quiver. After this preparation, we finish the paper with Section \ref{sec:6}, 
containing the proof of the Main Theorem. This section is also relatively technical, 
and it is splitted into subsections, most of the content covers the proof in case, when the Gabriel quiver 
is different from the spherical or the triangle quivers (see Section \ref{sec:HSA}), and the final part 
deals with the remaining exceptional quivers. \medskip 

For the neccessary background in the representation theory of algebras we refer the reader to books \cite{ASS,SkY}.  

\section{Basic notions}\label{sec:2} 

By a quiver we mean a quadruple $Q=(Q_0,Q_1,s,t)$, where $Q_0$ is a finite set of vertices, $Q_1$ a finite set of 
arrows and $s,t:Q_1\to Q_0$ are functions assigning to every arrow $\alpha$ its source $s(\alpha)$ and its target 
$t(\alpha)$. For a quiver $Q$, we denote by $KQ$ the {\it path algebra} of $Q$, whose $K$-basis is given by all 
paths of length $\geqslant 0$ in $Q$. Recall that the Jacobson radical of $KQ$ is the ideal $R_Q$ of $Q$ generated 
by all paths of length $\geqslant 1$, and ideal $I$ of $KQ$ is called {\it admissible}, provided that 
$R_Q^m\subseteq I \subseteq R_Q^2$, for some $m\geqslant 2$. Note that the paths $\ve_i$ of length $0$ at 
vertex $i\in Q_0$ form a complete set of primitive idempotents of $KQ$ (whose sum is the identity).  \smallskip  

If $Q$ is a quiver and $I$ is an admissible ideal $I$ of $KQ$, then $(Q,I)$ is said to be a {\it bound quiver}, 
and the associated algebra $KQ/I$ is called a {\it bound quiver algebra}. The Jacobson radical of an algebra 
$\La=KQ/I$ is $J_\La=R_Q+I$, and it is denoted by $J$. It is well-known that any algebra over an algebraically 
closed field is a bound quiver algebra, and by a presentation of an algebra $\Lambda$ we mean particular 
isomorphism $\Lambda\cong KQ/I$, for some bound quiver $(Q,I)$. In the case, the cosets $e_i=\varepsilon_i+I\in \La$ 
form a complete set of primitive idempotents of $\La$ and $\sum_{i\in Q_0}e_i$ is the identity of $\Lambda$. 
Note that the presentation may not be unique, but the quiver $Q$ is (up to permutation). It is called the {\it Gabriel 
quiver} of $\La$ and it is denoted by $Q_\La$; moreover, the arrows $i\to j$ in $Q_\La$ are in one-to-one correspondence 
with the elements of a basis of $e_i J e_j/e_iJ^2e_j$. \smallskip 

A {\it relation} in the path algebra $KQ$ is any $K$-linear combination of the form 
$$\sum_{i=1}^r \lambda_i w_i,\leqno{(1)}$$ 
where all $\lambda_i\in K$ are non-zero and $w_i$ are pairwise different paths of of length $\geqslant 2$ with 
common source and target. It is known that an ideal $I$ of $KQ$ is admissible if and only if $I$ is generated 
by a finite number of relations $\rho_1,\dots,\rho_m$. Moreover, we may choose such relations $\rho_1,\dots,\rho_m$ 
to be minimal (i.e. each $\rho_i$ is not a linear combination of relations from $I$). For a bound quiver 
algebra $A=KQ/I$, given the set of (minimal) relations $\rho_1,\dots,\rho_m$ generating $I$, we have the (minimal) 
equalities $\rho_1=0,\dots,\rho_m=0$ in $\La$, called {\it minimal relations}. \smallskip 

For a relation of the form $(1)$ and a path $w$ in $Q$, we write $w\prec \rho$, if $w$ is one of the summands of 
$\rho$, i.e. $w=w_i$, for some $i\in\{1,\dots,r\}$. Moreover, if $w$ is a path in $Q$, we will use notation $w\prec I$, 
if $w\prec \rho_i$, for $i\in\{1,\dots,m\}$ and some minimal relations $\rho_1,\dots,\rho_m$ generating $I$. \medskip  

Let $\Lambda$ be an algebra with given presentation $\Lambda=KQ/I$. Then modules $P_i=e_i\La$, for $i\in Q_0$, 
form a complete set of all pairwise non-isomorphic indecomposable projective modules in $\mod\La$, and modules 
$I_i=D(\La e_i)$, for $i\in Q_0$, form a complete set of all pairwise non-isomorphic indecomposable injective 
modules in $\mod\La$. We denote by $S_i$, for $i\in Q_0$, the associated simple module $S_i=P_i/\rad P_i\cong \soc I_i$. 
\smallskip  

We will always assume that algebras $\La$ are symmetric, i.e. there is a non-degenerate (associative) symmetric 
$K$-bilinear form $\La\times\La\to K$. For a symmetric algebra $\La$, we have $P_i\simeq I_i$, for any $i\in Q_0$, 
hence in particular, then $\Lambda$ is self-injective. We also assume $Q$ is connected, or equivalently, $\La$ is 
indecomposable as an algebra. \smallskip 

For $i\in Q_0$, we denote by $i^-$ be the set of arrows ending at $i$, and by $i^+$ the set of arrows starting at 
$i$. In this paper, the sizes $|i^-|$ and $|i^+|$ are at most $2$. A quiver $Q$ is said to be {\it 2-regular} if 
$|i^-|=|i^+|=2$, and {\it biregular} if $|i^-|=|i^+|\in\{1,2\}$, for all $i\in Q_0$. Recall that $Q$ is called 
biserial, if $|i^\pm|\leqslant 2$, for $i\in Q_0$. \smallskip  

We will use the following notation and convention for arrows: we write $\alpha, \ba$ for the arrows starting at 
vertex $i$, with the convention that $\ba$ does not exist in case $|i^+|=1$. Similarly we write $\gamma, \gamma^*$ 
for the arrows ending at some vertex $i$, where again $\gamma^*$ does not exist if $|i^-|=1$. Then $Q$ has a subquiver 
$$
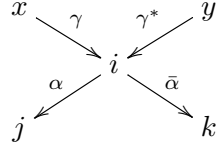
\xymatrix@R=0.3cm{x\ar[rd]^{\gamma}&&y\ar[ld]_{\gamma^*}\\&i\ar[ld]_{\alpha}\ar[rd]^{\ba}&\\j&&k}$$ 

We note that for a module $M$ in $\mod\La$, its {\it syzygy} is a kernel $\Omega(M)=Ker(\pi)$ of a projective 
cover $\pi:P(M)\to M$ of $M$ in $\mod\La$ (it is defined up to isomorphism). A module $M$ in $\mod\La$ is called 
{\it periodic} if $\Omega^d(M)\simeq M$, for some $n\geqslant 1$, and the smallest such a number is the {\it period} 
of $M$. The notion of {\it inverse syzygy} $\Omega^{-1}(M)$, for a module $M$ in $\mod \La$, is defined dually using 
injective envelopes. \smallskip  

Let $i\in Q_0$. Recall that there are natural isomorphisms $\Omega(S_i)=\rad P_i=\alpha\La+\ba\La$ and 
$\Omega^{-1}(S_i)\cong (\gamma,\gamma^*)\Lambda\subset P_x\oplus P_y$. In particular, it follows that the epimorphism 
$[\alpha \ \ba]:P_i^+\to \Omega(S_i)$, where $P_i^+=P_j\oplus P_k$, is a projective cover of $\Omega(S_i)$ and 
the monomorphism $\vec{\gamma \\ \gamma^*}:\Omega^{-1}(S_i)\to P_i^-$, where $P_i^-=P_x\oplus P_y$, is an injective 
envelope of $\Omega^{-1}(S_i)$ ($\La$ is symmetric, see \cite[Lemma 4.1]{AGQT}). As a result, if $S_i$ is a periodic 
module of period $4$, then $\Omega^2(S_i)\simeq\Omega^{-2}(S_i)$, and hence there is an exact sequence in $\mod\La$ of the form 
$$0\to S_i\to P_i \stackrel{d_3}\to  P_i^- \stackrel{d_2}\to  P_i^+ \stackrel{d_1}\to P_i \to S_i\to 0 \leqno{(*)}$$ 
with $\ima(d_k)\cong\Omega^k(S_i)$, for $k\in\{1,2,3\}$. By our convention, $P_y$ or $P_k$ may not exist. We may assume 
that $d_1(x, y) : = \alpha x + \ba y$, since the induced epimorphism $[\alpha \ \ba]:P_i^+ \to \Omega(S_i)$ is a 
projective cover of $\Omega(S_i)$ in $\mod\La$. Adjusting arrows $\gamma$ or $\gamma^*$ (including impact on 
presentation, i.e. on generators of $I$), we can fix $d_3(e_i) = (\gamma, \gamma^*)$ for some choice of the arrows 
$\gamma, \gamma^*$ ending at $i$ (see \cite[Proposition 4.3]{AGQT}). \smallskip 

The kernel of $d_1$ is then $\Omega^2(S_i)=\ima(d_2)$, and it has at most two minimal generators. They are images 
of idempotents $e_x\in P_x=e_x\La$ and $e_y\in P_y$ via $d_2:P_i^-\to P_i^+$. We may write them as $\vf$ and $\psi$, 
respectively, and they are contained in $P_j\oplus P_k$, so we can also write 
$$\vf = d_2(e_x,0) = (\vf_{jx}, \ \vf_{kx}) \ \ \mbox{and} \ \ \psi = d_2(0,e_y) = (\psi_{jy}, \ \psi_{ky}),$$
where $\vf_{jx}$ belongs to $e_j\La e_x$ (similarly for the other components of $\vf, \psi$). \smallskip 

Recall that any homomorphism $d:P_x\oplus P_y\to P_j\oplus P_k$ in $\mod\Lambda$ can be represented in the matrix 
form $$M={m_{jx} \ m_{jy}\choose m_{kx} \ m_{ky}},$$ 
where $m_{ab}$ is a homomorphism $P_b\to P_a$ in $\mod\La$, identified with an element $m_{ab}\in e_a\La e_b$, for 
any $a\in\{j,k\}$, $b\in\{x,y\}$. In this way, $d$ becomes multiplication by $M$, i.e. $d(u)=M\cdot u$, for $u\in P_i^-$ 
(using column notation for vectors in $P_i^-$ and $P_i^+$). \smallskip

Continuing with the generators of $\Omega^2(S_i)$, let $M_i$ be the matrix with column the components of $\vf$ and 
$\psi$, that is $d_2$ is given by a matrix 
$$M_i={\vf_{jx} \ \psi_{jy} \choose \vf_{kx} \ \psi_{ky}}.$$ 
Rewriting compositions $d_1d_2=0$ and $d_2d_3=0$ in the matrix form, we get the following identities 
$$(\alpha \ \ba)\cdot M_i = 0\mbox{ and }M_i\cdot {\gamma \choose \gamma^*} =0\leqno{(1)} $$
for some choice of arrows $\gamma, \gamma^*$ ending at $i$. This gives at most two minimal relations 
starting, and respectively, and ending at $i$ (cf. \cite[Proposition 4.3]{AGQT}). 

\section{Weighted surface algebras}\label{sec:HSA} 

In this section, we recall the definition of a weighted surface algebra and discuss some other 
related algebras, including the exotic family of higher spherical algebras. These algebras 
are main examples of TSP4 (or GQT) algebras, and their properties lead to general facts we are about to 
discuss in the next section. We also recap the classification of TSP4 algebras with spherical quiver; 
see Theorem \ref{prop:spherical}. We will see in Section \ref{sec:3} that this quiver 
naturally appears in the description of paths involved in minimal relations in blocks containing $1$-regular 
vertices. This is analogous situation as for the tetrahedral quiver arising in the classification of $2$-regular 
GQT algebras \cite{AGQT}. Another special quiver mentioned in this section is the triangle quiver, but its role is a 
classification is slightly different (see Section \ref{subsec:6.5}). \medskip 

First, we will explain what is a weighted surface (triangulation) algebra. By a triangulation 
quiver, we mean a ($2$-regular) quiver $Q$ which is a glueing of a finite number of the following three 
types of blocks 
$$
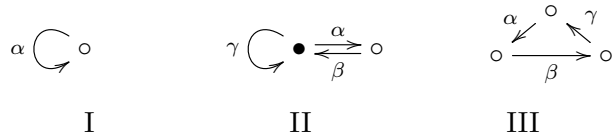
\xymatrix@C=0.3cm@R=0.2cm{\\ \ar@(lu, ld)[]_{\alpha}\circ} \qquad \qquad
\xymatrix@C=0.6cm@R=0.2cm{&\\ \ar@(lu, ld)[]_{\gamma}\bullet\ar@<.35ex>[r]^{\alpha}&\ar@<.35ex>[l]^{\beta}\circ} 
\quad \qquad 
\xymatrix@C=0.3cm@R=0.2cm{&\circ\ar[ld]_{\alpha}&\\ \circ\ar[rr]_{\beta}&&\circ\ar[lu]_{\gamma}}$$ 
$$\mbox{ I }\qquad\qquad\quad\quad\mbox{ II }\qquad\qquad\qquad\mbox{ III }$$
where by a glueing we mean that each vertex $\circ$ in any of the blocks is glued with exactly one vertex 
$\circ$ in a different block (see \cite[Section 2]{SS} for a precise definition of the glueing). 
Then the set of arrows of $Q$ admits a permutation $f:Q_1\to Q_1$ which fixes a loop in each block of type I, 
and otherwise, $f$ has an orbit of the form $(\alpha \ \beta \ \gamma)$. Moreover, $Q$ is $2$-regular, hence 
we have an involution $\overline{(-)}:Q_1\to Q_1$, which sends any arrow $\alpha$ to the arrow 
$\ba\neq \alpha$ starting at the same vertex as $\alpha$. In particular, one can consider the second 
permutation $g:Q_1\to Q_1$ given as $g(\alpha)=\overline{f(\alpha)}$. Now, take any collection of integers 
$m_\alpha$ and parameters $c_\alpha\in K\setminus\{0\}$, for $\alpha\in Q_1$, which are constant on $g$-orbits, 
and define the paths 
$$A_{\alpha}:=\alpha g(\alpha)\cdots g^{m_\alpha n_\alpha-2}(\alpha)\mbox{ and }B_\alpha=A_\alpha g^{-1}(\alpha),$$  
where $n_\alpha$ is the length of the $g$-orbit of $\alpha$. We assume that $m_\alpha n_\alpha\geqslant 2$, 
for all arrows, and an arrow $\alpha$ is said to be virtual, if $m_\alpha n_\alpha=2$. In the most general 
version, we also consider a border function, which is an arbitrary function $b_\bullet:\partial Q_0\to K$, 
assigning a coefficient $b_i$ to any vertex $i\in\partial Q_0$, where $\partial Q_0$ consists of vertices that 
admit a loop contained in a block of type I (sometimes called the border vertices). With this setup, the 
{\it weighted surface (triangulation) algebra} is a quotient $\La(Q,f,m_\bullet,c_\bullet,b_\bullet)=\La:=KQ/I$, 
where $I$ is generated by the following relations. 
\begin{enumerate}
\item[(1)] $\alpha f(\alpha)-c_{\ba}A_{\ba}$, if $\alpha\in Q_1$ is not a loop in block of type I.  
\item[(1')] $\alpha^2-c_{\ba}A_{\ba} - b_{s(\alpha)} B_{\alpha}$, if $\alpha$ is a loop in a block of type I. 
\item[(2)] $\alpha f(\alpha) g(f(\alpha))$, for all arrows $\alpha\in Q_1$, except $f^2(\alpha)$ is virtual, or 
$f(\ba)$ is virtual with $m_{\ba}n_{\ba}=3$.  
\item[(3)] $\alpha g(\alpha) f(g(\alpha))$, for all arrows $\alpha\in Q_1$, except $f(\alpha)$ is virtual, or 
$f^2(\alpha)$ is virtual with $m_{f(\alpha)}n_{f(\alpha)}=3$. 
\end{enumerate} \medskip 

{\it Remarks}. 

(i) Originally, the weighted surface algebras $\La(Q,f,m_\bullet,c_\bullet)$ were defined without the border 
function, or equivalently, with zero border function. In this paper, we use the above most general version, 
which covers all socle deformations of the algebras $\La(Q,f,m_\bullet,c_\bullet)$. Then one can keep isomorphism 
in the Main Theorem (instead of socle equivalence, as in \cite{AGQT}). Actually, algebras socle equivalent to 
algebras of the form $\La(Q,f,m_\bullet,c_\bullet)$ are exactly the algebras $\La(Q,f,m_\bullet,c_\bullet,b_\bullet)$, 
for some border function $b_\bullet:\partial Q_0\to K$ \cite[Theorem 1.2]{WSA-SD}. \smallskip 

(ii) Every virtual arrow $\alpha\in Q_1$ satisfies $A_\alpha=\alpha$, and then $\ba f(\ba)-c_\alpha\alpha\in I$, 
so $\alpha\in J^2$ is not an arrow of the Gabriel quiver $Q_\La\subset Q$. It is easy to check that virtual 
arrows are either loops in blocks of type II, or arrows of the $2$-cycle created by glueing of two triangles 
sharing two vertices. Hence, if $\La$ is a WSA, its Gabriel quiver $Q_\La$ is a glueing of finite number of 
blocks of types I-III and $V_1$, $V_2$. In particular, the quiver $Q_\La$ is biregular, and $1$-vertices 
are contained in blocks of types $V_1$, $V_2$. \smallskip 

(iii) We note that there are examples of TSP4 algebras, whose Gabriel quivers are not biregular or 
not biserial. We omit the details, just to mention, there are algebras called virtual mutations of WSA's 
\cite{HSS}, given by quivers, which are glueings of four types of blocks, the blocks I-III from WSA's, and 
a new block IV, containing two $1$-vertices $\bullet$ and two non-regular vertices $\circ$ (see 
\cite[Introduction]{HSS}). In most cases, the Gabriel quiver 
of a virtual mutation contains $(2,3)$- and $(3,2)$-vertices, so it is not biserial. In one special case, 
when we glue two blocks of type IV, one can get a $2$-cycle of virtual arrows, and the Gabriel quiver 
of the virtual mutation is the spherical quiver (disucussed below). \smallskip

(iv) We mention that the WSA's are TSP4 (GQT) algebras, except few cases, which we call exceptional 
algebras. These are the so called singular disc, triangle, spherical and tetrahedral algebras. Each of 
the algebras is a WSA with particular weights and parameters such that it fails to be symmetric or periodic. 
 Two of the cases (spherical and triangle) will be quickly recalled further, since these are 
related to the classification in biregular case. For more details we refer to 
\cite[Examples 3.1-3.3 and 3.6]{WSA-GV}. \smallskip 

(vi) There are also higher versions of two exceptional algebras. Indeed, the first algebra, called the 
higher tetrahedral algebra $\La=\La(m,\lambda)$ (for short, HTA), is a TSP4 algebra given by a triangulation 
quiver $Q=Q_\La$ (the same as the singular tetrahedral algebra), but not isomorphic to any WSA given by this 
quiver. Relations resembling similar pattern can be viewed as higher perturbations of relations defining WSA's. 
The second family consists of the higher spherical algebras $S(m,\lambda)$ (for short, HSA), which admit analogous 
description (see also \cite{E25}). The higher spherical algebras are given by the same Gabriel quiver as the singular 
spherical algebras, but they are not isomorphic to a WSA; see \cite[Lemma 5.2]{HSS}. \medskip 

Finally, recall the following theorem \cite[Main Theorem]{AGQT}, which shows that the WSA's exhaust almost all 
TSP4 (or GQT) algebras in the case, when the Gabriel quiver is $2$-regular. 

\begin{thm}\label{MTGQT} Let $A$ be an algebra with $2$-regular Gabriel quiver having at least three 
vertices. Then the following statements are equivalent. 
\begin{enumerate}
\item[(i)] $A$ is a TSP4 algebra. 
\item[(ii)] $A$ is a GQT algebra. 
\item[(iii)] $A$ is isomorphic to a weighted surface algebra different from the exceptional algebras, or 
is isomorphic to the higher tetrahedral algebra. 
\end{enumerate} 
\end{thm} \medskip

The second part of this section is devoted to discuss some special weighted surface algebras, i.e. spherical 
and triangle algebras, and present some clasification results. \smallskip 

Consider first the following triangulation quiver $Q$
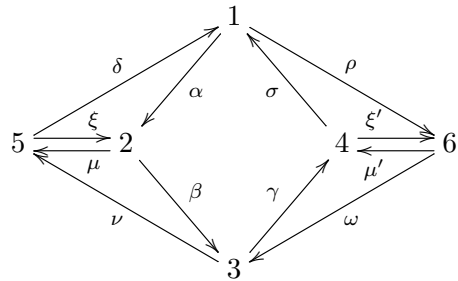
$$\xymatrix@C=1cm@R=1.2cm{&&1 \ar[ld]^{\alpha} \ar[rrd]^{\rho} && \\ 
5 \ar[rru]^{\delta}\ar@<+0.5ex>[r]^(.7){\xi} & 2 \ar@<+0.5ex>[l]^(.3){\mu}\ar[rd]^{\beta} & & 4 \ar[lu]^{\sigma} \ar@<+0.5ex>[r]^(.3){\xi'} & 6\ar[lld]^{\omega} \ar@<+0.5ex>[l]^(.7){\mu'} \\ 
&& 3 \ar[ru]^{\gamma} \ar[llu]^{\nu} &&}$$ 
which is a glueing of four triangles (note only that $Q$ is obtained from a certain triangulation of a sphere 
containing four triangles with coherent orientation \cite[see Example 3.2]{WSA-GV}). This has clearly four $f$-orbits, 
and four $g$-orbits: 
$$(\xi \ \mu), \ (\xi' \ \mu'), \ (\alpha \ \beta \ \gamma \ \sigma), \mbox{ and } (\rho \ \omega \ \nu \ \delta)$$ 
of lengths $n_\xi=n_{\xi'}=2$ and $n_\alpha=n_\rho=4$. If weights $n=m_\xi,n'=m_{\xi'}$ satisfy $n,n'\geqslant 2$, 
then $Q_\La=Q$ is $2$-regular and all paths of the form $\eta f(\eta)$ are involved in minimal relations (1) 
defining the weighted surface algebra $S=\La(Q,f,m_\bullet,c_\bullet)$ - note that then relations (2) and (3) 
hold for all arrows. In this case, $I$ is admissible, and using Theorem \ref{MTGQT}, we obtain that the unique TSP4 
(or GQT) algebra $\La$ with $Q_\La=Q$ is the weighted surface algebra $\La(Q,f,m_\bullet,c_\bullet)$, for some 
$m_\bullet,c_\bullet$ satisfying $n,n'\geqslant 2$ (see also \cite[Example 3.6]{WSA-GV}). Note that the singular 
shperical algebra appears only in case $n=n'=1$. \medskip 

If one of $n,n'$ is $1$, then we get at least one pair of virtual arrows ($\xi,\mu$ or $\xi',\mu'$). In 
this case, the Gabriel quiver of the associated weighted surface algebra is one of the following two 
quivers  
$$
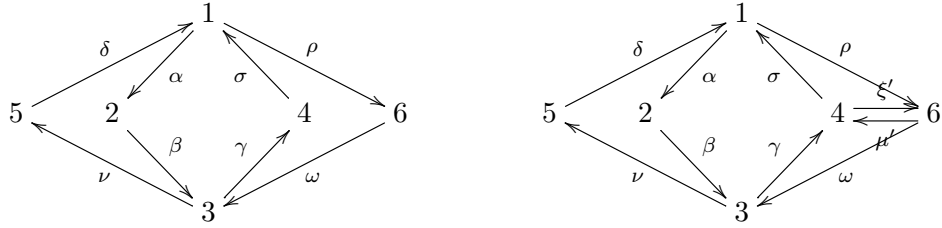
\xymatrix{&&1 \ar[ld]^{\alpha} \ar[rrd]^{\rho} && \\ 
5 \ar[rru]^{\delta} & 2 \ar[rd]^{\beta} & & 4 \ar[lu]^{\sigma} & 6\ar[lld]^{\omega} \\ 
&& 3 \ar[ru]^{\gamma} \ar[llu]^{\nu} &&} \qquad \qquad 
\xymatrix{&&1 \ar[ld]^{\alpha} \ar[rrd]^{\rho} && \\ 
5 \ar[rru]^{\delta} & 2 \ar[rd]^{\beta} & & 4 \ar[lu]^{\sigma} \ar@<+0.5ex>[r]^{\xi'} & 6\ar[lld]^{\omega} \ar@<+0.5ex>[l]^{\mu'} \\ 
&& 3 \ar[ru]^{\gamma} \ar[llu]^{\nu} &&}$$ 
denoted respectively, by $Q^S$ and $Q^{S'}$. The quiver $Q^S$ is called the spherical quiver, while $Q'$ the 
almost spherical quiver. \smallskip 

Recall that the higher spherical algebras introduced in \cite{HSA}, were proved to be derived equivalent with  
the higher tetrahedral algebras, but there was no classification of all TSP4 algebras given by the spherical quiver. 
The following recent result \cite[see Theorem 1.1]{E25} is filling this gap.  

\begin{thm}\label{prop:spherical} Assume $\La$ is a GQT algebra with Gabriel quiver $Q^S$. Then one of the following holds:\\
(a) \  $\La$ is isomorphic to a weighted surface algebra $KQ/I$ (different from the exceptional algebras), where 
$Q$ is the spherical quiver with four virtual arrrow $\xi, \mu, \xi', \mu'$.\\
(b) \ $\La$ is isomorphic to a Higher Spherical Algebra $S(m, \lambda)$ with $m>1$.
\end{thm} \medskip 

For a classification of GQT algebras with Gabriel quiver $Q^{S'}$, we refer to \cite[Theorem 1.2]{E25}. 
We only note that the HSA's can be also realized as iterated virtual mutations of the spherical algebras; 
see also {\it Remark}(iii) after the definition of WSA. \medskip 

Finally, we shall briefly discuss the triangle quiver and associated algebras. Namely, consider the following 
triangulation quiver $Q$
$$\xymatrix{\ar@(lu, ld)[]_{\rho}1\ar@<.35ex>[r]^{\delta}&\ar@<.35ex>[l]^{\nu} 2 \ar@<.35ex>[r]^{\alpha} & 
\ar@<.35ex>[l]^{\beta}3 \ar@(ru,rd)[]^{\sigma}}$$ 
being glueing of two blocks of type II. The permutation $f$ has two orbits $(\alpha \ \sigma \ \beta)$ and 
$(\nu \ \rho \ \delta)$, which induces permutation $g$ with one orbit $(\alpha \ \beta \ \nu \ \delta)$ of 
length $4$ and two orbits $(\rho)$, $(\sigma)$ of length $1$. As in the previous case, the weighted 
surface algebra $\La=\La(Q,f,m_\bullet,c_\bullet)$ depends on the weights $n=m_\rho,n'=m_\sigma\geqslant 2$. 
Indeed, if $n,n'\geqslant 3$, then $Q_\La=Q$ is $2$-regular, and all TSP4 algebras with $Q$ as a Gabriel 
quiver are classified, due to Theorem \ref{MTGQT}. If one of $n,n'$ is equal to $2$, then $Q_\La$ is a 
biregular quiver obtained from $Q$ by deleting one loop, and it is (up to permutation) the quiver of the form 
$$\xymatrix{\ar@(lu, ld)[]_{\rho}1\ar@<.35ex>[r]^{\delta}&\ar@<.35ex>[l]^{\nu} 2 \ar@<.35ex>[r]^{\alpha} & 
\ar@<.35ex>[l]^{\beta}3}$$ 
called the almost triangle quiver, and denoted by $Q^{T'}$. In case, when $n=n'=2$, both loops $\rho,\sigma$ are 
virtual, so the Gabriel quiver of $\La$ is the following quiver 
$$\xymatrix{1\ar@<.35ex>[r]^{\delta}&\ar@<.35ex>[l]^{\nu} 2 \ar@<.35ex>[r]^{\alpha} & 
\ar@<.35ex>[l]^{\beta}3}$$ 
called the triangle quiver, and denoted by $Q^T$. \medskip 

{\it Remark.} The classification of GQT algebras given by the triangle or almost quiver is covered in the final part 
of the proof (see Section \ref{subsec:6.5}). The GQT algebras $\La$ given by the almost triangle quiver 
$Q_\La=Q^{T'}$ admit the unique realization as a weighted surface algebra given by the triangulation quiver with 
one virtual loop (a glueing of two blocks of type II). The triangle quiver $Q^T$ requires 
a separate treatement (Section \ref{subsec:6.5}), since it generates two possible cases of WSA's with this Gabriel quiver. 
Namely, we can get $Q_\La=Q^T$ from a WSA $\La(Q,f,m_\bullet,c_\bullet)$, where $(Q,f)$ is glueing of two blocks of type II 
(with two virtual loops), or a glueing of two blocks of type III (with a $2$-cycle of virtual arrows). 
In the first case, $Q_\La$ is obtained from $Q$ by deleting both virtual loops, wheras in the second, by deleting 
the $2$-cycle of virtual arrows. The corresponding weighted surface algebras form two families of algebras called 
the triangle and disc algebras, which are isomorphic, up to a scalar deformation. Note also that in 
both cases we have a different structure of permutation $f$ (and $g$); for details, see \cite[Examples 3.3 and 3.4]{WSA-GV}.

\section{Algebras of generalized quaternion type}\label{sec:3} 

In this section, we recall known facts on GQT algebras, which will be needed further. In what follows, we denote 
by $\La$ a fixed GQT algebra. \medskip 

\subsection{General properties} 

The following results can be found in \cite[see Proposition 4.1 and Lemmas 4.3-4.4]{EHS1}. 

\begin{lemma}\label{lem:3.3} Assume $\alpha: i\to j$ and $\beta: j\to k$ are arrows such that $\alpha\beta \prec I$. Then there
is an arrow in $Q$ from $k$ to $i$, so that  $\alpha$ and $\beta$ are part of a triangle in $Q$. \end{lemma} 

\begin{lemma}\label{lem:3.4}  Assume $Q$ contains  a triangle
        
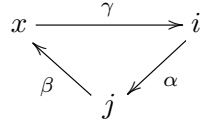
\[
 \xymatrix@R=1.3pc@C=1.8pc{
    x
    \ar[rr]^{\gamma}
    && i
    \ar@<.35ex>[ld]^{\alpha}
    \\
    & j
    \ar@<.35ex>[lu]^{\beta}
  }
\]
with $\alpha\beta\prec I$. If $\gamma$ is the unique arrow $x\to i$, then $\gamma\alpha\prec I$ and 
$\beta\gamma\prec I$. If we have double arrows $\gamma,\bar{\gamma}: x\to i$, then there is one  
$\delta\in\{\gamma,\bar{\gamma}\}$ such that $\delta\alpha\prec I$ and $\beta\delta\prec I$. \end{lemma} 

\begin{lemma}\label{lem:3.5} Assume $i$ is a 1-vertex which is part of a triangle 

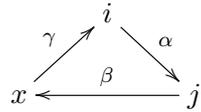
\[
  \xymatrix@R=1.3pc@C=1.8pc{
    & i
    \ar[rd]^{\alpha}
    \\
  x 
    \ar[ru]^{\gamma}
    && j
   \ar[ll]_{\beta}
  }
\]

Then both $x$ and $j$ must be at least $2$-vertices. \end{lemma} 

\medskip

Next two lemmas show similar properties of paths of length $3$ (see \cite[Proposition 4.5 and Lemma 4.6]{EHS1}).  

\begin{lemma}\label{lem:3.6} Suppose there is a path 
$$\xymatrix{i \ar[r]^{\alpha} & k \ar[r]^{\beta} & t \ar[r]^{\gamma} & j}$$ 
such that $\alpha\beta \not\prec I$ and $\alpha$ is the unique arrow $i\to k$. If $\alpha\beta\gamma \prec I$, then 
there is an arrow $j\to i$.  \end{lemma}

\begin{lemma}\label{lem:3.7} Assume $Q$ contains a square 
$$\xymatrix@R=1.3pc{j\ar[r]^{\delta} & i\ar[d]^{\alpha} \\ t\ar[u]^{\gamma} & k\ar[l]_{\beta}}$$ 
with $\alpha\beta\gamma\prec I$ but $\beta\gamma\nprec I$. If $\delta$ is the unique arrow $j\to i$, then 
$\beta\gamma\delta\prec I$.  \end{lemma} \medskip 

{\it Remark.} All the above lemmas hold for arbitrary algebra with all simples periodic of period $4$, 
so in particular, for any TSP4 algebra (with arbitrary Gabriel quiver). In case $\La$ is a GQT algebra with 
biregular Gabriel quiver, more is known. Namely, the following two results from \cite[see Proposition 5.1 and 
the Main Theorem]{EHS2} hold for arbitrary GQT algebra with biregular Gabriel quiver $Q=Q_\La$.  

\begin{prop}\label{prop:4.2} Let $\La$ be a GQT algebra with biregular Gabriel quiver $Q$ and $i\in Q_0$ a 
$1$-vertex in a triangle 
$$
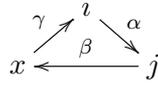
\xymatrix@R=0.6pc@C=0.5cm{&i\ar[rd]^{\alpha}&\\ x\ar[ru]^{\gamma}&&j\ar[ll]_{\beta}} $$  
Then $Q$ contains a block of the form    
$$
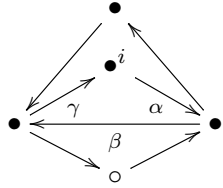
\xymatrix@R=0.7pc{ & \ar[ldd] \bullet & \\ & \bullet^i \ar[rd]_{\alpha} & \\ 
\bullet \ar[rd] \ar[ru]_{\gamma} && \ar[ll]^{\beta} \bullet \ar[luu] \\ 
& \circ \ar[ru] & }$$ 
where $\circ$ is a $2$-vertex and $\bullet$ are $1$-vertices. \end{prop} 

\begin{thm}\label{prop:4.3} 
If the Gabriel quiver $Q=Q_\La$ is biregular, then every $1$-vertex $i\in Q_0$ is contained in a block of the form 
$$
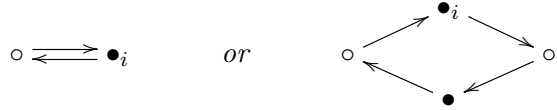
\xymatrix@R=0.4pc{&&& &\bullet_i\ar[rd]& \\ 
\circ\ar@<+.4ex>[r]&\bullet_i\ar@<+.4ex>[l]&\mbox{ or }& \circ\ar[ru] && \circ\ar[ld]  \\ 
&&& &\ar[lu]\bullet&}$$ 
where $\bullet$ and $\circ$ denote, respectively, the $1$-vertices and $2$-vertices. \end{thm} 

The above theorem confirms that in biregular case, the $1$-regular vertices of $Q$ belong to the desired blocks, 
exactly as for general weighted surface algebras (see Section \ref{sec:HSA}). \smallskip  

\subsection{Minimal relations around $1$-vertices}

Now, we assume $\La$ has biregular Gabriel quiver $Q_\La=Q$. According to Theorem \ref{prop:4.3} (see also \ref{prop:4.2}), 
every $1$-vertex of $Q$ is contained in one of two types of blocks $V_1$ or $V_2$. We will now discuss minimal relations 
of $I$ around such blocks. From now on, we assume that $Q=Q_\La$ is different from the spherical quiver $Q^S$ and the 
triangle or the almost triangle quivers $Q^T$ and $Q^{T'}$; see Section \ref{sec:HSA}. Moreover, we will also assume that 
$Q$ has at least three vertices. \bigskip 

Clearly, $Q$ consists of $2$-regular and $1$-regular vertices, and by Theorem \ref{prop:4.3}, one can group 
all $1$-regular vertices in $Q_0$ into two sets of vertices, namely, vertices $c_i,d_i$, $i\in\{1,\dots,p\}$, 
and vertices $y_i$, $i\in\{1,\dots,q\}$, where for any $i\leqslant p$, vertices $c_i,d_i$ are $1$-veritces 
in a block $B_i$ of type $V_2$  
$$
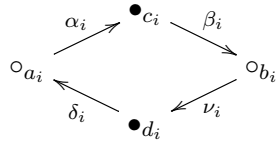
\xymatrix@R=0.6pc{&\bullet_{c_i}\ar[rd]^{\beta_i}& \\ 
\circ_{a_i}\ar[ru]^{\alpha_i} && \circ_{b_i}\ar[ld]^{\nu_i}\\ & \bullet_{d_i} \ar[lu]^{\delta_i}& }$$ 
whereas $y_i$ lies in a block $B_i'$ of type V$_1$: $\xymatrix{\circ_{x_i}\ar@<+.4ex>[r]^{\varepsilon_i} & 
\bullet_{y_i}\ar@<+.4ex>[l]^{\eta_i}}$, 
if $i\leqslant q$. In particular, all vertices $x_i,a_i,b_i$ are $2$-regular in $Q$, and $y_i,c_i,d_i$ exhaust  
$1$-regular vertices in $Q$. Note also that $a_i\neq b_i$, for all $i$, since otherwise $Q=Q^T$ is the triangle 
quiver, excluded for now. \smallskip 

Observe further that for any $i\in\{1,\dots,p\}$, we have no arrows between $c_i$ and $d_i$ in $Q$ (and no 
loops at $y_i$'s), so we conclude from Lemma \ref{lem:3.3} that $\delta_i\alpha_i\nprec I$, $\beta_i\nu_i\nprec I$ 
(and $\eta_i\ve_i\nprec I$). Moreover, we can assume that at least one $\alpha_i\beta_i\nprec I$ or $\nu_i\delta_i\nprec I$. 
Indeed, if one of the paths, say $\nu_i\delta_i$, is involved in a minimal relation, then it is a part of a 
triangle (with $1$-vertex), by Lemma \ref{lem:3.3}, and hence $Q$ has a local shape as described in Proposition 
\ref{prop:4.2}. But then there is no arrow $b_i\to a_i$ in $Q$, and again by Lemma \ref{lem:3.3}, we conclude that 
$\alpha_i\beta_i\nprec I$. Hence up to relabelling of arrows, one can assume that $\alpha_i\beta_i\nprec I$, for all 
$i\in\{1,\dots,p\}$. In particular, we have also $\beta_i\bar{\nu}_i\nprec I$ and $\delta_i^*\alpha_i$, due 
to Lemmas \ref{lem:3.3} and \ref{lem:3.4}. 

\subsection*{Critical block.} A block $B_i$ with $\nu_i\delta_i\prec I$ will be called {\it critical}. 
We note that a non-critical block $B_i$ may be a part of a block depicted in Proposition \ref{prop:4.2}, 
but then we have $\nu_i\delta_i\nprec I$. The block from Proposition \ref{prop:4.2} will be sometimes called 
{\it sup-critical}, as any critical block in $Q$ is contained in such a block. \medskip 

Note finally that for a block $B'_i$, $i\in\{1,\dots,q\}$, we have also $\ve_i\eta_i\nprec I$. 
Indeed, if $\ve_i\eta_i\prec I$, we get a loop at $x$, by Lemma \ref{lem:3.3}, and then $Q$ consists of two vertices 
(and three arrows), which contradicts the assumptions on $Q$. Moreover, there are no arrows $t(\bar{\ve})\to y_i$ 
and $y_i\to s(\eta^*)$, so Lemma \ref{lem:3.3} gives $\eta_i\bar{\ve}_i\nprec I$ and $\eta_i^*\ve_i\nprec I$. \medskip 

\begin{lemma}\label{lemA} If $B_i$ is not critical, then $\beta_i^*\nu_i,\delta_i\bar{\alpha}_i\nprec I$ and all 
the paths $$\alpha_i\beta_i\nu_i, \ \beta_i\nu_i\delta_i, \ \nu_i\delta_i\alpha_i,\delta_i\alpha_i\beta_i$$ 
are involved in minimal relations of $I$. Moreover, we have $\beta_i^*\overline{\nu_i}\prec I$ (respectively, 
$\delta_i^*\overline{\alpha_i}\prec I$), if $t(\overline{\nu_i})$ (respectively, $t(\overline{\alpha_i})$) 
is a $2$-vertex. \end{lemma} 

\begin{proof} We write $\alpha,\beta,\nu,\delta$ for the arrows $\alpha_i,\beta_i,\nu_i,\delta_i$, 
skipping the indices (the same convention for vertices $a:=a_i,\dots$). It follows easily from Lemmas \ref{lem:3.3} 
and \ref{lem:3.4} that $\beta^*\nu\nprec I$ and $\delta\ba\nprec I$, since in a non-critical block we have 
$\nu_i\delta_i\nprec I$, by definition. Hence all paths of length $2$ starting or ending with an arrow from 
$B_i$ are not involved in minimal relations of $I$. In particular, using Lemma \ref{lem:3.7}, we deduce that 
either none or all of the paths of length $3$ in $B_i$ are involved in relations. Suppose to the 
contrary that all these paths are not involved in minimal relations of $I$. Then by \cite[Lemma 4.7]{EHS1}, 
we have $\delta\alpha\beta\nu\nprec I$. \medskip 

Note that both $a$ and $b$ are $2$-regular, due to Lemma \ref{lem:3.5}. \smallskip 

Now, if $b^*:=s(\beta^*)$ is a $2$-vertex, then we obtain the following wild subcategory in 
covering. 
$$\xymatrix@R=0.7pc{ & d \ar[d]_{\delta} && d && \\ 
\bar{a} & \ar[l]_{\ba} a \ar[r]^{\alpha} & c \ar[r]^{\beta} & b \ar[u]^{\nu} & \ar[l]_{\beta^*} b^* \ar[r]^{\sigma} & b' }$$ 

Hence, we can assume that $b^*$ is a $1$-vertex, so $b^*=c_k,d_k$ or $y_k$, for some $k$. If $\sigma$ denotes the 
unique arrow ending at $b^*$, then $\sigma\beta^*\nprec I$, because otherwise, we have an arrow $b\to s(\sigma)$, 
by Lemma \ref{lem:3.3}, which is impossible, since $s(\sigma)$ is a $2$-vertex, whereas arrows $\nu,\bar{\nu}$ 
starting at $b$ have only $1$-regular targets. Observe also that $\sigma\beta^*\nu\nprec I$. Indeed, if this is 
not the case, then due to Lemma \ref{lem:3.6}, there is an arrow $d_i\to s(\sigma)$, hence $s(\sigma)=a_i$, 
so that $Q=Q^S$, a contradiction. Consequently, we obtain a wild subcategory of the same type as above, but 
with reversed orientation of $\sigma$. \medskip 

This shows the first part of the claim. Finally, it follows that $\beta^*\bar{\nu}\prec I$ in case $t(\bar{\nu})$ 
is $2$-regular (the inverse implication is trivial). Indeed, if this was not true, we would get all the four paths 
of length $2$ passing through $b$ not involved in minimal relations of $I$. But then $\bar{b}:=t(\bar{\nu})$ is 
$2$-regular and we get the following wild subcategory in covering. 
$$\xymatrix@R=0.7pc{&c \ar[d]^{\beta} && \\ 
b' \ar[r]^{\beta^*} & b \ar[r]^{\bar{\nu}} \ar[d]^{\nu} & \bar{b} & \ar[l]_{\sigma} a' \\ &d&& }$$ 
Dually, we prove that $\delta\ba\prec I$, for $t(\ba)$ being $2$-regular. \end{proof} \medskip

Now, assume that $B_i$ is a critical block. Then $\overline{\alpha}=\beta^*$ and $B_i$ is contained in 
a block of the form  
$$
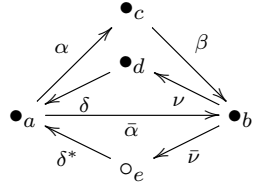
\xymatrix@R=0.6pc{&\bullet_{c}\ar[rdd]^{\beta}& \\ & \bullet_{d}\ar[ld]^{\delta}& \\ 
\bullet_{a}\ar[rr]_{\ba}\ar[ruu]^{\alpha} && \bullet_{b}\ar[lu]^{\nu} \ar[ld]^{\bar{\nu}} \\ 
& \circ_{e} \ar[lu]^{\delta^*} &  }$$ 
with $e=t(\overline{\nu})=s(\delta^*)$ being $2$-regular. Additionally, we denote 
$\tau=\delta^*$, $\sigma=\bar{\alpha}$ and $\gamma=\bar{\nu}$. In this case, we have the following.  

\begin{lemma}\label{lemAcritical} If $B_i$ is critical, then all the paths $\alpha\beta\nu$, $\beta\nu\delta$, 
$\nu\delta\alpha$ and $\delta\alpha\beta$ are involved in minimal relations of $I$, whereas $\tau\alpha\beta\nprec I$ 
and $\alpha\beta\gamma\nprec I$. Moreover, we have $\tau\alpha\nprec I$, $\beta\gamma\nprec I$ and 
$\tau\sigma,\sigma\gamma,\gamma\tau\prec I$. \end{lemma} 

\begin{proof} First, it is clear from Lemma \ref{lem:3.3} that $\tau\alpha\nprec I$ and $\beta\gamma\nprec I$, 
because we have no arrows from (or to) $1$-regular vertices ending (or starting) at vertex $e$. \medskip 

Next, observe that $\sigma\nu\prec I$, by Lemma \ref{lem:3.4}, since we have a triangle $(\nu \ \delta \ \sigma)$ 
and $\nu\delta\prec I$. Moreover, we conclude that $\sigma\gamma\prec I$, because otherwise, we would get a 
wild subcategory of the form 
$$\xymatrix@R=0.6pc{ && c \ar[d]_{\beta} & a \ar[ld]_{\ve} \ar[r]^{\alpha} & c \ar[r]^{\beta} & b & a \ar[l]_{\sigma}  
\ar[r]^{\alpha} & c \\ e' \ar[r] & e & \ar[l]_{\gamma} b \ar[r]^{\nu} & d &&&& }$$ 
which is isomorphic to a wild one-relation algebra \cite[see XII in Theorem 1.5.2]{Rin}. By Lemma \ref{lem:3.4}, 
we have all $\sigma\gamma,\gamma\tau,\tau\sigma\prec I$. \medskip

Finally, due to Lemma \ref{lem:3.5} (and its dual), it is sufficient to show that one of the paths 
of length three in $B_i$ is involved in a minimal relation. First, observe that 
$$\tau\alpha\beta\prec I\mbox{ or }\delta\alpha\beta\prec I,$$ 
since if not, we would get the following wild subcategory of type $\wt{\wt{\bE}}_8$  
$$\xymatrix@R=0.6pc{&&&&&& d\ar[d]_{\delta} && \\  a \ar[r]^{\sigma}  & b & \ar[l]_{\beta} c & \ar[l]_{\alpha} a 
\ar[r]^{\sigma} & b & \ar[l]_{\beta} c & \ar[l]_{\alpha} a & \ar[l]_{\tau} e \ar[r] & \circ  }$$  

As a result, it remains to show that $\tau\alpha\beta\nprec I$. After possibly adjusting arrow $\sigma$, we have one minimal 
relation of the form $\tau\sigma=A_{\bar{\tau}}$, where $A_{\bar{\tau}}$ is a 
combination of paths starting with $\bar{\tau}$. Therefore, we can write $A_{\bar{\tau}}=\bar{\tau} ({'}A_{\bar{\tau}})$, 
and hence we obtain a generator of $\Omega^2(S_e)$ of the form 
$$\vf=(\sigma,-('A_{\bar{\tau}}))\in e_a\La e_b\oplus e_{t(\bar{\tau})}\La e_b.$$ 

Now, it follows that $\tau\alpha\beta\nprec I$, because otherwise, we would obtain a minimal relation of the form 
$\tau\alpha\beta+\tau z_1 + \bar{\tau}z_2=0$, and consequently, another generator of the form $\vf'=(\alpha\beta+z_1,z_2)$ 
with $z_1\in J$. But then, there is $x\in e_b\La e_b$ such that $\vf'=\vf x$, and comparing the first coordinates 
we get $\sigma x=\alpha\beta+z_1$. Suppose $x\notin J$. Then the left hand side is a scalar multiplication of the 
arrow $\sigma$, and hence, also $z_1\notin J^2$ has the form $z_1=\lambda \sigma$, for $\lambda\in K^*$. Then either 
$\lambda\neq x$, and we obtain $(\lambda-x)\sigma=\alpha\beta\in J^2$, a contradiction, or $\lambda=x$, but then 
$\alpha\beta=0$ is involved in a minimal zero relation, which is not the case (always $\alpha\beta\nprec I$). Finally, 
if $x\in J$, then $\vf x\in J^2$, so also $z_1\in J^2$, and therefore, again $\alpha\beta$ is involved in a minimal 
relation of the form $\alpha\beta+z_1-\sigma x$, a contradiction. \end{proof} 

In case of a critical block as above, we know a bit more about relations involving $\gamma\tau$ and $\nu\delta$. 
Namely, the following lemma holds. 

\begin{lemma}\label{lemAcrit1} We may assume that both $\nu\delta,\gamma\tau$ are involved in (one) commutativity 
relation $\nu\delta-\gamma\tau$. \end{lemma} 

\begin{proof} To see this, consider the exact sequence for $S_{b}$  
$$0\to (\beta, \sigma)\La \to P_{c}\oplus P_{a} \to P_{d}\oplus P_{e} \stackrel{\pi}\to \nu\La +  \gamma\La \to 0.$$
We can write down two minimal generators of ${\rm ker}(\pi)$ which are both images of $P_{a}$. 
Namely, we know that $\gamma\tau$ is part of a triangle, so we have a relation 
$$\gamma\tau = \nu ('A_{\nu})$$ 
where $'A_{\nu}$ is an element in $e_{d}\La e_{a}$. So we have a generator
$('A_{\nu},  -\tau)e_1$ of the kernel of $\pi$. \smallskip 

On the other hand, because $\nu\delta \prec I$, we may assume that 
$$\nu\delta = \nu z_1\tau + \gamma z_2\tau$$ 
(with $z_1\in J^2$). Indeed, we can write $\nu\delta = \nu z_1\tau + \gamma z_2\tau + \nu z_3\delta+\gamma z_4\delta$, 
for some $z_1,\dots,z_4\in \La$. But  $z_3,z_4\in J$, so 
$\nu z_3\delta + \gamma z_4\delta = \nu\delta z_3'\nu\delta+ \gamma z_4'\nu\delta$, hence we obtain 
$$\nu\delta - \nu\delta z_3'\nu\delta - \gamma z_4'\nu\delta = 
(1-\nu\delta z_3'-\gamma z_4')\nu\delta = u \nu \delta,$$
where $u$ is a unit. After rescaling by $u^{-1}$, we get the required relation. \smallskip 

Consequently, we have a generator $(\delta - z_1\tau, z_2\tau)e_1$ of ${\rm ker}(\pi)$. But there is 
only one such generator.
It follows that
$$(\delta - z_1\tau, z_2\tau) = ('A_{\nu}, -\tau)w$$
where $w\in e_1\La e_1$ is a unit. We have $z_1\in J^2$, so we may replace $\delta$ by 
$\delta'= \delta-z_1\tau$ and then $\delta' = 'A_{\nu}$, which gives 
$\nu\delta' = \nu('A_{\nu}) = \gamma\tau$. \end{proof} 

Eventually, we consider relations involving arrows from blocks of type V$_1$. 

\begin{lemma} \label{lemB} For any $i\in\{1,\dots,p\}$, we have 
$$\ve_i\eta_i\ve_i, \eta_i\ve_i\eta_i \prec I.$$ 
Moreover $\eta_j^*\overline{\ve_j}\prec I$ if and only if $t(\overline{\ve_i})$ is a $2$-vertex. \end{lemma} 

\begin{proof} We abbreviate as in previous cases $\eta:=\eta_i$ and $\ve:=\ve_i$. Suppose to the contrary that 
one of $\eta\ve\eta,\ve\eta\ve$ is not involved in a minimal relation of $I$. Actually, we can assume that both 
$\eta\ve\eta,\ve\eta\ve\nprec I$, due to Lemma \ref{lem:3.7}, since we have the square $(\eta \ \ve \ \eta \ \ve)$. 
As a result, we conclude from \cite[Lemma 4.7]{EHS1} that then also $\eta\ve\eta\ve\nprec I$. \smallskip 

Now consider the vertex $z=t(\bar{\ve})$. If $z$ is a $2$ vertex, let $\sigma:z'\to z$ denote the second arrow 
ending at $z$, different from $\bar{\ve}$. For $z$ being a $1$-vertex, let $\sigma:z\to z'$ denote the unique 
arrow starting at $z$. We claim that then $\bar{\ve}\sigma\nprec I$. Indeed, because $z$ is a $1$-vertex, 
we conclude that $\bar{\ve}\sigma=\alpha_k\beta_k$, $\nu_k\delta_k$ or $\ve_k\eta_k$, for some $k$. In the 
first and the last case the claim is clear, for the second, observe that $B_k$ cannot be a part of a critical 
block. \medskip 

Finally, it remains to see that $\eta\bar{\ve}\sigma\nprec I$, in case when $z$ is a $1$-vertex. Otherwise, by Lemma 
\ref{lem:3.6}, we would obtain an arrow $z'\to s(\eta)=y$, and then $z'=x$, so $Q$ is the triangle quiver $Q=Q^T$, 
excluded from our considerations. In a consequence, we conclude that $\La$ admits the following wild hereditary 
subcategory of type $\wt{\wt{\bD}}_6$. 
$$\xymatrix@R=0.6cm{ &z \ar@{-}[r]^{\sigma} & z' & y' \ar[d]^{\eta^*} & \\ 
y \ar[r]^{\eta} & x \ar[u]^{\bar{\ve}} \ar[r]^{\ve} & y \ar[r]^{\eta} & x \ar[r]^{\ve} & y}$$ \medskip 

This proves that first part of the claim. For the second, it is sufficient to see that $\eta^*\bar{\ve}$ is 
the unique path of length $2$ passing through $x$ which can be involved in a minimal relation. Consequently, 
if $z$ is a $2$-vertex, then $\eta^*\bar{\ve}\prec I$, because otherwise we get the following wild hereditary 
subcategory of type $\wt{\wt{\bD}}_4$. 
$$\xymatrix@R=0.6cm{&y \ar[d]^{\eta} && \\ 
x' \ar[r]^{\eta^*} & x \ar[r]^{\bar{\ve}} \ar[d]^{\ve} & z & \ar[l]_{\sigma} z' \\ &y&& }$$ 
It is easy to see that $z$ must be a $2$-vertex if $\eta^*\bar{\ve}\prec I$, and the proof is now finished. \end{proof} \medskip 

\subsection{Summary}\label{summary}

Summing up, it follows from Lemmas \ref{lemA}-\ref{lemAcritical} that for any $i\in\{1,\dots,p\}$, we can fix 
the following minimal relations of $I$ involving paths of length $3$ in $B_i$.  
$$\rho_1^{(i)}=\alpha_i\beta_i\nu_i-{\bf A}_{\overline{\alpha_i}}, \  
\rho_2^{(i)}=\beta_i\nu_i\delta_i-{\bf A}_{\beta_i}, \ 
\rho_3^{(i)}=\nu_i\delta_i\alpha_i-{\bf A}_{\overline{\nu_i}}, \ \mbox{and} \ 
\rho_4^{(i)}=\delta_i\alpha_i\beta_i-{\bf A}_{\delta_i}.$$ 
Moreover, each of ${\bf A}_{\alpha}$ is a combination of paths in $Q$ starting with the arrow $\alpha$. Indeed, every 
path from $a_i$ to $d_i$ starting with $\alpha_i$ has the form $\alpha_i\beta_i\dots\nu_i$, so we can write 
${\bf A}_{\overline{\alpha_i}}$ as $\alpha_i\beta_iz\nu_i+\hat{A}_{\overline{\alpha_i}}$, where $z\in J$ and $\hat{A}_{\overline{\alpha_i}}$ has all terms 
starting with $\overline{\alpha_i}$. In this way, the relation $\rho_1^{(i)}$ becomes 
$$\alpha_i\beta_i(1-z)\nu_i-\hat{A}_{\overline{\alpha_i}},$$ 
so after rescaling by the inverse of the unit $1-z$, we may take the new relation $\rho_1^{(i)}$ with 
${\bf A}_{\overline{\alpha_i}}:=(1-z)^{-1}\hat{A}_{\overline{\alpha_i}}$ having all summands starting with  
$\overline{\alpha_i}$ (and ending with $\nu_i$). Similarly, one can show that (up to a scalar multiplication) 
${\bf A}_{\overline{\nu_i}}$ has all summands starting with $\overline{\nu_i}$ (and ending with $\alpha_i$). 
Actually, using the same argument we can show that ${\bf A}_{\overline{\alpha_i}}$ ends with $\beta_i^*\nu_i$, while 
${\bf A}_{\overline{\nu_i}}$ with $\delta_i^*\alpha_i$. \smallskip 

Further, every path from $c_i$ starts with $\beta_i$ (the unique arrow from $c_i$), so ${\bf A}_{\beta_i}$ has all 
summands starting with $\beta_i$. Thus we can write $\rho_2^{(i)}=\beta_i(\nu_i\delta_i-{'A}_{\beta_i})$. As above, 
after rescaling by a unit, we can assume that all summands of ${'A}_{\beta_i}$ (so of ${\bf A}_{\beta_i}$) end with 
$\delta_i^*$. In the same way, one can show that all summands of ${\bf A}_{\delta_i}$ start with $\delta_i$ and end with 
$\beta_i^*$. In fact, we will soon see in Section \ref{bases} that every ${\bf A}_\alpha$ is a scalar multiplication 
of one path (being an initial submonomial of some socle element). \smallskip

Finally, it follows from Lemma \ref{lemB} that we have minimal relations of the form  
$$\rho^{(j)}_5=\ve_j\eta_j\ve_j-{\bf A}_{\overline{\ve_j}} \ \mbox{and} \ \rho^{(j)}_6=\eta_j\ve_j\eta_j-{\bf A}_{\eta_j},$$ 
$j\in\{1,\dots,q\}$. Applying analogous arguments as in case of block of type V$_2$, we may take ${\bf A}_{\overline{\ve_j}}$ having 
all summands starting with $\overline{\ve_j}$ and ending with $\eta^*\ve_j$ and ${\bf A}_{\eta_j}$ having all summands starting with 
$\eta_j\bar{\ve}_j$ and ending with $\eta_j^*$. \medskip 

As a result, we can define elements ${'A}_\alpha$ and $A'_\alpha$, for arrows $\alpha=\beta_i,\delta_i,\eta_i$ or 
$\bar{\alpha}_i,\bar{\nu}_i$ such that 
$$\overline{\alpha_i}('{A}_{\overline{\alpha_i}})={\bf A}_{\overline{\alpha_i}}=A_{\overline{\alpha_i}}'\nu_i, \ 
\overline{\nu_i}({'A}_{\overline{\nu_i}})={\bf A}_{\overline{\nu_i}}=A_{\overline{\nu_i}}'\alpha_i, \ 
\beta_i({'A}_{\beta_i})={\bf A}_{\beta_i}=A'_{\beta_i}\delta_i^*,$$ 
$$\delta_i({'A}_{\delta_i})={\bf A}_{\delta_i}=A'_{\delta_i}\beta_i^*, \ 
\overline{\ve_j}('{A}_{\overline{\ve_j}})={\bf A}_{\overline{\ve_j}}=A_{\overline{\ve_j}}'\ve_j, \mbox{ and } 
\eta_j({'A}_{\eta_j})={\bf A}_{\eta_j}=A'_{\eta_j}\eta_j^*.$$ 

Consider now the exact sequence for $S_{c_i}$: 
$$\xymatrix{P_{c_i} \ar[r]^{\alpha_i} & P_{a_i} \ar[r] & P_{b_i} \ar[r]^{\beta_i} & P_{c_i} }.$$ 
It follows from the relation $\rho_2^{(i)}$ that $\beta_i(\nu_i\delta_i-{'A}_{\beta_i})=0$, so the unique generator 
of $\Omega^2(S_{c_i})$ is given by the element $\vf=\nu_i\delta_i-{'A}_{\beta_i}$. Using the exact sequence, we 
deduce that $\vf\alpha_i=0$, which gives $\rho_3^{(i)}$, hence ${\bf A}_{\overline{\nu_i}}={'A}_{\beta_i}\alpha_i$, 
and consequently, we get $A_{\overline{\nu_i}}'={'A}_{\beta_i}$. \smallskip 

Similarly, using the exact sequence for $S_{d_i}$ and relation $\rho_4^{(i)}$, we conclude that 
${'A}_{\delta_i}\nu_i={\bf A}_{\overline{\alpha_i}}$, that is $A_{\overline{\alpha_i}}'={'A}_{\delta_i}$. The same argument 
works for $\rho_6^{(j)}$, and this gives $A_{\overline{\ve_j}}'={'A}_{\eta_j}$.

\section{Bases of projectives}\label{bases} 

This section is devoted to present the description of bases of projective $\Lambda$-modules, for 
any algebra $\La$ of generalized quaternion type with biregular Gabriel quiver. We assume throughout that 
$Q_\La$ is different from the spherical quiver $Q^S$ and the triangle and almost triangle quivers $Q^T$ and 
$Q^{T'}$. In particular, then we can use results describing minimal relations from the previous section. \medskip 

Let $\La$ be a fixed GQT algebra with biregular Gabriel quiver. We use notation for vertices and arrows of $Q=Q_\La$ 
as introduced in Section \ref{sec:3} (after Theorem \ref{prop:4.3}). In the previous section, we mainly 
described relations around $1$-vertices, here we need to investigate what are other minimal relations. \smallskip 

The major part of this section is devoted to describe a basis of projective $P_i$, for a $2$-regular vertex 
$i\in Q_0$. The case when $i$ is $1$-regular is considered at the end (see Lemma \ref{basis1reg}). \smallskip 

Fix a $2$-vertex $i\in Q_0$ with arrows $\alpha,\ba$ starting from $i$ and arrows $\gamma,\gamma^*$ ending at $i$; let 
$j=t(\alpha)$, $k=t(\ba)$, $x=s(\gamma)$ and $y=s(\gamma^*)$.  Note that, thanks to Proposition \ref{prop:4.2}, we 
may relabel vertices such that $j$ is $1$-regular if and only if $x$ is, and if so, then they belong to one block of 
type $V_1$ or $V_2$. Similarly, vertices $k,y$ belong to one block of type $V_1$ or $V_2$ if one of them is a $1$-vertex. \smallskip 

We have to consider three cases, depending on the neighbours of $i$. If $i$ belongs to the `$2$-regular part' of 
$Q$, that is, both $j,k$ are $2$-regular, we call $i$ a vertex {\it of type A}. In the remaining two cases, $i$ has 
$1$-regular predecessors (equivalently, successors); we say that $i$ is {\it of type B or C}, provided that $i$ 
admits exactly one or exactly two $1$-regular predecessors, respectively. In each of these cases, there is a 
significantly different behaviour of ralations starting from $i$, as well as the radical quotients of projective 
$P_i$, as we will see below. \medskip 

The rest part of this section is organized as follows. In the first part \ref{subsec:5.1}, we will study 
the remaining relations generating $I$, showing that all minimal relations starting from a $2$-vertex have a common 
form analogous to the relations $\rho^{(i)}_1,\dots,\rho^{(i)}_6$ around $1$-vertices (see Section \ref{summary}). 
Next, in Section \ref{subsec:5.2}, we analize the radical quotients of projectives corresponding to $2$-regular 
vertices, obtaining generating sets: first for $2$-vetices of type A (Lemma \ref{lem:gen}) and then in full 
generality, for all $2$-vertices (Lemma \ref{lem:2vertex}). The last part \ref{subsec:5.3} contains two results 
describing bases of projective module $P_i$, the first for a $2$-vertex $i$, the second, for a $1$-vertex. 
\medskip 

\subsection{Minimal relations starting from a $2$-vertex.}\label{subsec:5.1} We begin with the following proposition, 
which is a partial extension of \cite[Proposition 4.2]{AGQT}.  

\begin{prop}\label{GQT4.2} Let $i$ be a $2$-vertex of $Q$, and $\alpha$ an arrow starting at $i$. If $j=t(\alpha)$ 
is a $2$-vertex and  $\beta,\bar{\beta}$ arrows starting at $j$, then one of the paths $\alpha\beta,\alpha\bar{\beta}$ 
is involved in a minimal relation of $I$. \end{prop} 

\begin{proof} Suppose to the contrary that both $\alpha\beta\nprec I$ and $\alpha\bar{\beta}\nprec I$. We consider 
the second arrow $\ba:i\to k$ starting at $i$. If $k$ is a $2$-vertex then we have an arrow $\sigma:l\to k$, $\sigma\neq\ba$, 
and otherwise, for $k^+=\{\sigma\}$ we get $\ba\sigma\nprec I$, because then $\ba=\alpha_s,\nu_s$ or $\eta_s$ and we 
can have $\ba\sigma\prec I$ only when $\ba=\nu_s$ is in a critical block $B_s$, which is impossible, since then  
$\alpha=\bar{\nu_s}$ satisfies $\alpha\beta\prec I$ or $\alpha\bar{\beta}\prec I$ ($\beta$ or $\bar{\beta}$ is $\tau$). \smallskip 

As a result, we obtain the following wild subcategory in covering 
$$\xymatrix@R=0.6cm{&& \circ\ar[d] &&& \\ && \circ &&& \\ 
\circ \ar[r] & \circ & \ar[l]_{\beta} j \ar[u]^{\bar{\beta}} & \ar[l]_{\alpha} i \ar[r]^{\ba} & k \ar@{-}[r]^{\sigma} & \circ  }$$ 
if both $t(\beta)$ and $t(\bar{\beta})$ are $2$-regular (note that $\sigma$ is either an arrow $l\to k$ different from 
$\ba$ or this is the unique arrow starting from $k$ and $\ba\sigma\nprec I$). Consequently, we can assume that one of 
$t(\beta),t(\bar{\beta})$ is a $1$-vertex, say $t(\beta)$. Then one of the predecessors of $j$ must be a $1$-vertex 
(by Proposition \ref{prop:4.3}), so we have an arrow $\alpha^*:u\to j$, where $u$ is a $1$-vertex. It follows also 
that $t(\bar{\beta})$ must be $2$-regular, since the second predecessor of $j$ is a $2$-vertex $i$. In particular, 
we have both $\alpha^*\beta\nprec I$ and $\alpha^*\bar{\beta}\nprec I$, since otherwise $t(\beta)$ lies in a critical block 
and then $\alpha=\delta_s^*$, so $\alpha\bar{\beta}$ is involved in a minimal relation. But then we get a wild (hereditary) 
subcategory in covering of the form 
$$\xymatrix@R=0.6cm{&& \circ & \\ k & \ar[l]_{\ba} i \ar[r]^{\alpha} & j \ar[r]^{\beta} \ar[u]^{\bar{\beta}} & \circ \\ 
&& u \ar[u]_{\alpha^*} & }$$ \end{proof} 

Hence applying the above proposition for a $2$-vertex $i$ of type A, we conclude that there are arrows $\beta_1\in j^+$ 
and $\gamma_1\in k^+$, with the paths $\alpha\beta_1,\ba\gamma_1$ involved in minimal relations. The following 
result shows that relations starting from a $i$ behave as described in \cite[Proposition 4.2(ii)]{AGQT} for 
$2$-regular Gabriel quivers, or in other words, in biregular case relations in the '$2$-regular part' of $Q$ 
behave exactly as in $2$-regular case. 

\begin{lemma}\label{lem:relations} We may assume that the paths $\alpha\beta_1$ and $\ba\gamma_1$ are involved 
in two independent minimal relations of $I$. \end{lemma}

\begin{proof} 
Suppose to the contrary that $\alpha\beta_1,\ba\gamma_1$ are involved in one minimal relation of type C:  
$$a\alpha\beta_1+b\ba\gamma_1\in J^3,$$ 
where $a,b\neq 0$. Let $u=t(\beta)$ and $v=t(\gamma)$, where $\beta=\bar{\beta_1}$ and $\gamma=\bar{\gamma_1}$. 
In particular, the paths $\alpha\beta_1,\ba\gamma_1$ end at common 
vertex $t(\beta_1)=t(\gamma_1)$, which is different from $u$ and $v$, because otherwise, we have $j=k$ and $Q$ is a 
Markov quiver (see \cite[Lemma 5.2]{AGQT}). \smallskip 

We can further assume that both $\alpha\beta\nprec I$ and $\ba\gamma\nprec I$. Indeed, if this is not the case, then 
we may exchange $\beta$ and $\beta_1$ (or $\gamma,\gamma_1$) to get paths involved in two independent relations 
(starting from $i$ and ending at $t(\beta_1)$ and $u$ or $v$, which are different from $t(\beta_1)$). Moreover, there 
are two arrows $\delta:x\to i$ and $\delta^*:y\to i$ in $Q$, and due to the given relation and Lemma \ref{lem:3.3}, 
we get an arrow $t(\beta_1)\to i$, so $t(\beta_1)=x$ or $y$. We may assume that the relation of type C starts from $i$ 
and ends at $x=t(\beta_1)$. \medskip  

We claim that then $y=x$. Indeed, if $y\neq x$ then $\delta^*\alpha\nprec I$, 
because otherwise, by Lemma \ref{lem:3.3}, we have an arrow $j=t(\alpha)\to y$, so $y=x$ or $u$, thus $y=u$, 
by the assumption, and hence, we get a triangle $(\delta^* \ \alpha \ \beta)$ (without double arrows), where 
$\delta^*\alpha\prec I$, but $\alpha\beta\nprec I$, a contradiction with Lemma \ref{lem:3.4}. In the same way 
(using $\ba\gamma\nprec I$ and Lemmas \ref{lem:3.3}-\ref{lem:3.4}), one can show that $\delta^*\ba\nprec I$. Now, 
applying Proposition \ref{prop:4.3}, we deduce that both predecessors $x,y$ of $i$ are $2$-regular (since 
by the general assumption, here successors $k,j$ are $2$-regular). Consequently, we have an arrow 
$\delta^*:y\to i$ between $2$-vertices $y$ and $i$ such that both $\delta^*\alpha\nprec I$ and $\delta^*\ba\nprec I$, 
which contradicts Proposition \ref{GQT4.2}. \medskip 

Therefore, we have $y=x$ and $\delta,\delta^*$ are double arrows $x\to i$. Now, recall that we have an 
exact sequence for $S_i$ of the form 
$$\xymatrix{0 \ar[r] & S_i \ar[r] & P_i\ar[r]^{d_1} & P_x\oplus P_x \ar[r]^{d_2} & P_j\oplus P_k\ar[r]^{d_3} & P_i 
\ar[r] & S_i \ar[r] & 0}$$ 
where the rows of $d_2=[\vf \ \psi]$, $\vf=\vec{\vf_1 \\ \vf_2}$, $\psi=\vec{\psi_1 \\ \psi_2}$, induce two 
(independent) minimal relations ending at $i$. By Lemma \ref{lem:3.4}, we can adjust arrows $\delta,\delta^*$ 
to get both $\delta\alpha,\beta_1\delta\prec I$ and both $\delta^*\alpha,
\beta_1\delta^*\nprec I$. \smallskip 

If we had also $\gamma_1\delta^*\nprec I$, we would get an arrow $\delta^*$ 
(between $2$-regular vertices) with both paths ending at $\delta^*$ satisfying $\nprec I$, which gives a 
contradiction with the dual of Proposition \ref{GQT4.2}. As a result, we have also $\gamma_1\delta^*\prec I$. 
\smallskip 

After possibly adjusting arrows $\beta_1$ or $\gamma_1$, we may assume that $\beta_1\delta$ and $\gamma_1\delta^*$ 
are involved in two (independent) minimal relations of the form 
$$\beta_1\delta+\psi_1\delta^*=0 \qquad\mbox{ and } \qquad \vf_2\delta + \gamma_1\delta^*=0$$ 
In other words, we get $\vf_1=\beta_1$ and $\psi_2=\gamma_1$. Now using that $[\alpha \ \ba] d_2=0$, 
we obtain two independent minimal relations involving $\alpha\beta_1$ and $\ba\gamma_1$ (induced 
by columns of $d_2$). This finishes the proof. \end{proof} \medskip 

In particalar, we get the following corollary. 

\begin{cor} There exists a permutation $f$ defined locally as $f(\alpha)=\beta_1$ and $f(\ba)=\gamma_1$, 
for arrows between $2$-vertices. The induced permutation $g$ is given as follows $g(\alpha)=\beta(=\overline{\beta_1})$ and $g(\ba)=\gamma(=\overline{\gamma_1}$). \end{cor} \medskip 

Consequently, for all arrows $\alpha\in Q_1$ between $2$-vertices we can define an arrow $f(\alpha)$ (starting 
at vertex $t(\alpha)$) such that $\alpha f(\alpha)\prec I$ and arrow $g(\alpha)=\overline{f(\alpha)}$. If 
$\alpha\in Q_1$ starts at a $2$-vertex, but ends at a $1$-vertex, then $\alpha=\alpha_i,\nu_i$ or $\ve_i$, for 
some $i$, and $f(\alpha)$ is not defined. However, we can define $g(\alpha)=\beta_i$, $\delta_i$ or $\eta_i$, 
respectively. Otherwise, $\alpha$ starts from a $1$-vertex, but ends at a $2$-vertex, and then $\alpha=\beta_i$, 
$\delta_i$ or $\eta_i$, for some $i$, so we define $f(\alpha)=\nu_i$, $\alpha_i$ or $\ve_i$, respectively. 
We can also define $g(\alpha)=\overline{f(\alpha)}$, but in this case $\alpha f(\alpha)\nprec I$. \medskip 

Summing up, we have a permutation $g:Q_1\to Q_1$ and partially defined permutation $f$ (for all arrows ending 
at $2$-vertex). Moreover, we have $g=\bar{f}$, whenever $f$ is defined, and $\alpha f(\alpha)\prec I$ if and 
only if $\alpha$ is different from $\beta_i,\delta_i,\eta_i$, i.e. $s(\alpha)$ is a $2$-vertex. \medskip 

We may artificially 'extend' the permutation $f$, defining $f(\alpha)$ for arrows $\alpha=\alpha_i,\nu_i,\ve_i$, 
so that $f(\alpha)$ is defined for all arrows $\alpha$, but $f(\alpha)$ may be a path. Indeed, we can set 
$$f(\alpha_i)=\beta_i\nu_i, \ f(\nu_i)=\delta_i\alpha_i, \mbox{ and } f(\ve_i)=\eta_i\ve_i.$$ 
In this way, for any arrow $\alpha\in Q_1$ we have a path $\alpha f(\alpha)$ (of length $2$ or $3$), such 
that $\alpha f(\alpha)\prec I$ if and only if $s(\alpha)$ is a $2$-vertex. Moreover, if this is the case, 
then there is a minimal relation of the form $\alpha f(\alpha)-{\bf A}_{\ba}$, with ${\bf A}_{\ba}\in J^2$.  
For arrows $\alpha$ in blocks of type $V_1$ or $V_2$, we know more about the summands of ${\bf A}_{\ba}$, 
namely, it has been shown in Section \ref{summary} that 
$${\bf A}_{\ba_i}\in \ba_i\La\beta_i^*\nu_i, \ {\bf A}_{\beta_i}\in\beta_i\La\delta_i^*, \ 
{\bf A}_{\bar{\nu}_i}\in\bar{\nu}_i\La \delta_i^*\alpha_i, \mbox{ and } {\bf A}_{\delta_i}\in\delta_i\La\beta_i^*,$$ 
and moreover: ${\bf A}_{\bar{\ve}_i}\in\bar{\ve}_i\La\ve_i$ and ${\bf A}_{\eta_i}\in\eta_i\La\eta_i^*$. We will 
see below that analogous properties hold for all elements ${\bf A}_{\ba}$, $\alpha\in Q_1$. \medskip 

Assume now that $i$ is a vertex of type B or C. Then $\alpha f(\alpha)\prec I$ and $\ba f(\ba)\prec I$, since 
$s(\alpha)=s(\ba)=i$ is a $2$-vertex, and these paths (of length $2$ or $3$) are involved in minimal relations 
of the form 
$$\alpha f(\alpha)-{\bf A}_{\ba} \qquad\mbox{ and }\qquad \ba f(\ba)-{\bf A}_{\alpha},\leqno{(*)}$$ 
where at least one is of the form $\rho^{(k)}_1$ or $\rho^{(k)}_3$. Then one of the relations is ending at 
$1$-vertex, and the second at a $2$-vertex, or a different $1$-vertex, so these are always two independent 
relations (that end at different vertices).  Therefore, we get the following corollary. 

\begin{cor} If $i$ is a $2$-vertex of type B or C, then the paths $\alpha f(\alpha)$ and $\ba f(\ba)$ are 
defined and involved in two independent minimal relations of $I$. \end{cor} \smallskip 

Note also that, if $f(\alpha)$ is not an arrow, then $\alpha f(\alpha)=\alpha g(\alpha) f(g(\alpha))$ is 
one of the paths $\alpha_i\beta_i\nu_i$, $\nu_i\delta_i\alpha_i$, or $\ve_i\eta_i\ve_i$. \medskip 

Summarizing the above observations, for any $2$-vertex $i$ there are two independent minimal relations starting from 
$i$ of the form $(*)$. If $i$ is of type C, then ${\bf A}_{\ba}\in\ba\La g(\alpha)^* f(g(\alpha))$ 
and ${\bf A}_\alpha\in\alpha\La g(\ba)^* f(g(\ba))$, by the above considerations (see Section \ref{summary}). 
Otherwise, one of $\alpha,\ba$ is an arrow between $2$-vertices, say $\ba$, and then we may rewrite 
relation involving $\ba f(\ba)$ in the form $\ba (f(\ba)+z)-{\bf A}_{\alpha}$, so that the new ${\bf A}_\alpha$ 
belongs to $\alpha\La$, after possibly adjusting arrow $f(\ba)$. We may assume that any such an adjustement 
has $z\in\La f(\ba)^*$ because summands of $z$ in $\La f(\ba)$ can be skipped, up to scalar multiplication 
of $f(\ba)$. Analogous adjustement of $\ba$, gives the 
minimal relation $(*)$ involving $\ba f(\ba)$ with ${\bf A}_\alpha\in \La f(\ba)^*$. Similarly, one can assume 
that ${\bf A}_{\ba}\in\ba\La f(\alpha)^*$, if $t(\alpha)$ is a $2$-vertex (possibly adjusting arrows $f(\alpha),\alpha$), 
and otherwise, ${\bf A}_{\ba}\in\ba\La g(\alpha)^*f(g(\alpha))$. Then we may write 
${\bf A}_\alpha=\alpha('A)_\alpha$ and ${\bf A}_{\ba}=\ba('A_{\ba})$, and hence, we obtain 
$[\alpha \ \ba]\cdot M=0$, for a matrix $M$ of the form 
$$M=\vec{f(\alpha) & 'A_{\alpha} \\ 'A_{\ba} & f(\ba)},$$ 
identified with a homomorphism $m:P_x\oplus P_y \to P_j\oplus P_k$. \medskip 

Now, the exact sequence associated 
to $S_i$ has the following form 
$$\xymatrix{P_i\ar[r]^{\vec{\gamma \\ \gamma ^*}} & P_x\oplus P_y \ar[r]^{d_2} & 
P_j\oplus P_k\ar[r]^{\vec{\alpha & \ba}} & P_i},$$  
and due to \cite[Lemma 3.2(i)]{EHS1}, we can take $d_2=\vec{\vf & \psi}=\vec{\vf_1 & \psi_1 \\ \vf_2 & \psi_2}$ with 
the second column $\psi$ equal to the second column of $M$ (possibly adjusting arrow $\gamma,\gamma^*$). Similarly, if 
$f(\alpha)$ is an arrow, then we may take $\vf$ equal to the first column $\theta$ of $M$. If $f(\alpha)$ is not an arrow, 
but $'A_{\ba}\notin J^2$, then we can repeat the same argument. So we assume that $f(\alpha)$ and $'A_{\ba}$ 
are not arrows, or equivalently, $\theta\in J^2$; since $f(\alpha)\nprec I$, we conclude that $\theta$ is not in $J^3$. 
Clearly $[\alpha \ \ba]\cdot\theta=0$, so $\theta\in\ker(d_1)=\ima(d_2)$, and 
therefore, we can write $\theta=\vec{\theta_1 \\ \theta_2}$ as $\theta=d_2(X,Y)=\vf X+\psi Y$, for some 
$X\in e_x\La e_x$ and $Y\in e_y\La e_x$. As a result, we have the following two relations in $I$ 
$$\theta_1-\vf_1X+('A_\alpha)Y \quad\mbox{ and }\quad \theta_2-\vf_2X-f(\ba)Y.$$ 
It yields identity $M=d_2\cdot N$, for a matrix $N=\vec{X & 0 \\ Y & e_y}$. We claim that 
$Y\in J^2$ and $X\notin J$. Indeed, otherwise $Y$ is an arrow $y\to x$ (and $\alpha=\nu_i$ is contained in a 
block presented in Proposition \ref{prop:4.3}), but then $f(\ba)Y\prec I$, which is impossible, since we have 
no arrows from a $1$-vertex $x$ to $s(f(\ba))=k$. Therefore indeed, we have $Y\in J^2$, hence $\psi Y\in J^3$, 
and consequently, we have $\vf X\notin J^3$, because $\theta\notin J^3$. Finally, since $\vf\in J$, we 
cannot have $X\in J^2$, so either $X$ is an arrow or $X\notin J$. But we have no loops at $x$, hence 
$X\notin J$ is a unit of the local algebra $e_x\La e_x$. It follows that the map 
$v:P_x\oplus P_y\to P_x\oplus P_y$, defined by the matrix $N$ is an isomorphism (acting as adjustement 
of arrow $\gamma:x\to i$), and we have the following commutative diagram in $\mod\La$ 
$$\xymatrix{P_i\ar[r]^{d_3} & P_x\oplus P_y \ar[r]^{d_2} & P_j\oplus P_k \ar[r]^{d_1} & P_i \\ 
P_i\ar[r]^{d_3'} \ar[u]^{1} & P_x\oplus P_y \ar[r]^{M} \ar[u]^{v} & P_j\oplus P_k \ar[r]^{d_1} \ar[u]^{1} & P_i \ar[u]^{1}}$$ 
where $d_3'=v^{-1}d_3$ and all vertical maps are isomorphisms. This means that, after adjusting arrows 
$\gamma,\gamma^*$, we can take an exact sequence for module $S_i$ with $d_2=M$ determined by the two minimal 
relations $(*)$ starting from $i$, and additionally, satisfying ${\bf A}_{\alpha}\in \alpha\La f(\ba)^*$ and 
${\bf A}_{\ba}\in \ba\La f(\alpha)^*$, if $t(\alpha)$ is a $2$-vertex, or 
${\bf A}_{\ba}\in\ba\La g(\alpha)^*f(g(\alpha))$, otherwise. \medskip 

In particular, every two relations $(*)$ starting from $i$ give two relations ending at $i$. Namely, 
rewriting the identity $M\cdot d_3=0$, we obtain two minimal relations of the form 
$$f(\alpha)\gamma - 'A_\alpha\gamma^* \quad \mbox{ and }\quad f(\ba)\gamma^*-'A_{\ba}\gamma.$$ 
As before, adjusting $\gamma,\gamma^*$ again, we may assume that $'A_\alpha\in g(\alpha)\La$ and 
$'A_{\ba}\in g(\ba)\La$, and we get the two relations of the form $\omega f(\omega)-{\bf A}_{\bar{\omega}}$ 
for $\omega=f(\alpha)$ or $f(\ba)$, being rotations of the relations starting from $i$, and 
additionally, we have ${\bf A}_{\bar{\omega}}\in \bar{\omega}\La f(\omega)^*$ for both arrows. \medskip 

In this way, we continue `rotations', to get desired relations for all arrows in triangles (the relations 
involving path in blocks with $1$-vertices are known). We mention that, if an arrow $\alpha$ was 
adjusted once to get a relation $\alpha f(\alpha)-{\bf A}_{\ba}$ with ${\bf A}_{\ba}\in\ba \La f(\alpha)^*$, 
then after any further adjustement (of $\alpha$ or $f(\alpha)$), these properties are preserved, i.e. we have 
still $\alpha f(\alpha)\prec I$ (for a new set of minimal relations), and possibly changed ${\bf A}_{\ba}$ is 
also in $\ba\La f(\alpha)^*$. Consequently, for any $2$-regular vertex $i\in Q_0$, we can take the exact sequence 
for $S_i$ with the middle map $d_2$ given by a matrix of the form $M$, determined by the two relations $(*)$ 
starting from $I$. Similarly for $1$-vertices, thanks to relations in Section \ref{summary}. Because, all relations 
in $I$ are generated by the minimal relations induced from the exact sequences of simples, we conclude that 
there is a set of minimal relations generating $I$, consisting of all relations $\rho^{(i)}_1,\dots,\rho^{(i)}_6$, 
and all relations $\alpha f(\alpha)-{\bf A}_{\ba}$, for arrows $\alpha$ between $2$-vertices, where 
${\bf A}_{\ba}\in\ba\La f(\alpha)^*$. Moreover, we have $'A_{\alpha}=A_{g(\alpha)}'$, for all arrows $\alpha$ 
(see also Section \ref{summary}).

\subsection{Radical quotients and generators.}\label{subsec:5.2} Now, we will investigate the radical quotients 
of projectives associated with $2$-regular vertices. The following corollary describes the second quotient in case 
of vertices of type A. 

\begin{cor}\label{cor:J2/J3} If $i$ is a $2$-vertex of type A, then the the cosets 
$\alpha g(\alpha)+J^3$ and $\ba g(\ba)+J^3$ form a basis of $e_iJ^2/e_iJ^3$. \end{cor} \smallskip 

{\it Remark.} Note that for a $2$-vertex $i$ with $1$-regular successor, say $j=t(\alpha)$, we may have 
$e_iJ^2/e_iJ^3$ one-dimensional, but $\dim_K e_iJ^k/e_iJ^{k+1}=2$, for higher $k$'s. Namely, it happens when 
$\alpha=\alpha_i$ lies in a critical block $B_i$ such that $\alpha_i\beta_i\nu_i\notin J^3$. Then, we must 
have $\alpha_i\beta_i\nu_i={\bf A}_{\ba}\equiv \ba g(\ba)$, modulo $J^3$. In this case, we have arrows 
$\alpha,\ba$ generating the first radical quotient, but for the second, we have only one generator 
$\alpha g(\alpha)=\alpha_i\beta_i$, because all other paths of length $2$ are $\ba g(\ba)\equiv \alpha_i\beta_i\nu_i$, 
which belongs to $J^3$, and $\ba f(\ba)={\bf A}_\alpha$, which is also in $J^3$, since every path in $Q$ 
starting with $\alpha$ and ending at $t(f(\ba))$ has length $\geqslant 3$ ($B_i$ is contained in a block 
presented in the statement of Proposition \ref{prop:4.2}). So $\dim_K e_iJ^2/e_iJ^3=1$ and the coset of 
$\alpha g(\alpha)$ is a generator. But further step shows that the dimension increases, i.e. the cosets 
of $\alpha_i\beta_i\nu_i\equiv \ba g(\ba)\in J^3$ and $\alpha_i\beta_i\bar{\nu}_i=\alpha g(\alpha) g^2(\alpha)$ 
generate $e_i J^3/e_iJ^4$. \medskip 

Note also that for a $1$-vertex $i$, $i^+=\{\alpha\}$, the quotient $e_iJ^2/e_iJ^3$ is generated by the cosets of paths 
$\alpha f(\alpha)$ and $\alpha g(\alpha)$. \medskip 

Now, we will present some preparatory lemmas. In what follows 
we use the notation $\Theta_k(\eta)$ for the path $\eta g(\eta) \cdots g^{k-1}(\eta)$ of length $k$ along the 
$g$-cycle of an arrow $\eta\in Q_1$. 

\begin{lemma}\label{gfV2} Let $\alpha$ be an arrow $\alpha=\alpha_i$ or $\nu_i$. Then $\alpha g(\alpha) f(g(\alpha))={\bf A}_{\ba}\in J^4$ or ${\bf A}_{\ba}\equiv \Theta_k(\ba)$, for 
$k=2$ or $3$ (modulo $J^4$). In the second case, we have $g^{-1}(\alpha){\bf A}_{\ba}\in J^5$. 
\end{lemma} 

\begin{proof} We denote by $\beta,\nu$ the arrows $g(\alpha)$ and $f(g(\alpha))=f(\beta)$ and $\delta=g(\nu)$. 
Let $j=t(\alpha)$, $k=t(\ba)$, be the targets of arrows starting from $a:=a_i$, and $x=s(\delta)=t(\nu)$, 
$y=s(\delta^*)$, the sources of arrows ending at $a$ (this is a $2$-vertex). Then $\beta f(\beta)\nprec I$, 
since the path starts at $1$-vertex, and $\alpha\beta\nu={\bf A}_{\ba}$, due to relations \ref{summary}. Observe that 
every path $p\prec {\bf A}_{\ba}$ is a path starting from $\ba$ and ending with $\beta^*\nu$. \smallskip 

Of course, if every $p\prec {\bf A}_{\ba}$ has length $\geqslant 4$, then ${\bf A}_{\ba}$ is in $J^4$, and we are 
done. Otherwise, we have a summand $p$ of length $2$ or $3$. In the first case, we must have $\ba=\beta^*$, and 
then $p=\beta^*\nu=\ba g(\ba)$ is the unique summand of ${\bf A}_{\ba}$ length $2$ (and others are in $J^4$), hence 
$\alpha\beta\nu={\bf A}_{\ba}\equiv \ba g(\ba)=\Theta_2(\ba)$ modulo $J^4$, as required. If the smallest length 
summand $p\prec {\bf A}_{\ba}$ has length $3$, then $\ba\neq \beta^*$ and $p=\ba\beta^*\nu$ is the unique 
summand of ${\bf A}_{\ba}$ of length $3$, so we obtain ${\bf A}_{\ba}\equiv p$ modulo $J^4$. In this case, vertex 
$t(\ba)=k$ cannot be a $1$-vertex, since it would be contained in a block of type V$_2$, and then $Q=Q^S$, a 
contradiction. Therefore, $k,y$ are $2$-regular, and we have arrows $f(\ba):k\to y$ and $f^2(\ba)=\delta^*:y\to a$. 
We can assume that $f(\ba)\neq\beta^*$, because otherwise, we have a minimal relation $\ba f(\ba)={\bf A}_\alpha$, 
where ${\bf A}_\alpha=z\alpha\beta$, thus $p={\bf A}_\alpha\nu$ , and consequently, we have a relation 
$\alpha\beta\nu=c p + {\bf \tilde{A}}_{\ba} ={\bf A}_{\ba}=c z\alpha\beta\nu+{\bf \tilde{A}}_{\ba}$, for some 
$c\in K^*$. Then after rescaling by the unit $1-cz$, we can remove the summand $cp$ from ${\bf A}_{\ba}$, obtaining 
a relation $\alpha\beta\nu={\bf A}_{\ba}$, with new ${\bf A}_{\ba}:=(1-cz)^{-1}{\bf \tilde{A}}_{\ba}\in J^4$ 
having no summands of the form $\ba\beta^*\nu=\ba f(\ba) g(f(\ba))$. As a result, we proved that $p$ must be 
of the form $p=\ba\beta^*\nu=\ba g(\ba) g^2(\ba)=\Theta_3(\ba)$, and the claim follows. \medskip 

It remained to show that $Z=g^{-1}(\alpha){\bf A}_{\ba}\in J^5$, in the case ${\bf A}_{\ba}\equiv p$ for a path  
$p=\Theta_k(\ba)$ of length $k=2$ or $3$. 

If $p=\Theta_2(\ba)=\ba\nu$, then the block containing $\alpha$ occurs in  
in the block shown in Proposition \ref{prop:4.2}, and $\ba=f(g^{-1}(\alpha))$, so we have a minimal relation 
of the form 
$$g^{-1}(\alpha)\ba={\bf A}_{\bar{\sigma}},$$ 
where $\sigma=g^{-1}(\alpha)$ and all paths $p\prec {\bf A}_{\bar{\sigma}}$ are starting with $\bar{\sigma}$ 
and ending with $\alpha\beta$. Since $t(\bar{\sigma})\neq a=s(\alpha)$, we conclude that 
$\sigma f(\sigma)={\bf A}_{\bar{\sigma}}\in J^4$, and consequently, $Z=\sigma f(\sigma)\nu\in J^5$. \smallskip 

Finally, let $p=\Theta_3(\ba)=\ba\beta^*\nu$. In this case, $Z\equiv \sigma p =\sigma f(\sigma)\beta^*\nu$. 
As above, we get a minimal relation $\sigma f(\sigma)={\bf A}_{\bar{\sigma}}$, where summands of ${\bf A}_{\bar{\sigma}}$ 
are starting with $\bar{\sigma}$ and ending with $(\ba)^*=f^{-1}(\beta^*)$. If $\sigma f(\sigma)={\bf A}_{\bar{\sigma}}\in J^3$, 
then clearly $Z\in J^5$, and there is nothing to prove. If ${\bf A}_{\bar{\sigma}}$ admits a summand of $cq$, 
$c\in K^*$, of length $2$, then $q=\bar{\sigma}(\ba)^*=\bar{\sigma} g(\bar{\sigma})$ is the unique path 
of length $2$ involved in ${\bf A}_{\bar{\sigma}}\equiv q$ modulo $J^3$, and then $\sigma p\equiv 
q\beta^*\nu=\bar{\sigma} f^{-1}(\beta^*)\beta^* \nu$. But $f^{-1}(\beta^*)$ is an arrow between $2$-vertices 
of $Q$, so we get a minimal relation of the form $f^{-1}(\beta^*)\beta^*={\bf A}_{\zeta}$, where 
all paths $r\prec A_{\zeta}$ are starting with $\zeta=\overline{f^{-1}(\beta^*)}$ and ending with $\alpha\beta$. 
But then ${\bf A}_\zeta\in J^3$, and hence $Z\equiv \sigma p \equiv \bar{\sigma} {\bf A}_\zeta \nu\in J^5$. \end{proof} 

For blocks of type $V_1$ one can prove the following. 

\begin{lemma}\label{gfV1} For any arrow $\alpha=\ve_i$ we have $\alpha g(\alpha) f(g(\alpha))={\bf A}_{\ba}\in J^4$. \end{lemma}

\begin{proof} We abbreivate, $\eta=\eta_i$, $\ve=\ve_i$. It is straightforward to see that all paths $p\prec {\bf A}_{\ba}$ 
are of length $\geqslant 4$, except $p=\ba \eta^*\ve=\bar{\ve} \eta^*\ve$, which is then ${\bf A}_{\ba}\equiv\Theta_3(\ba)$. 
In the second case, $Q$ must be a glueing of a block $(\alpha \ \eta)$ of type $V_1$ and a block of type $V_1$ or II, 
and then $Q$ is either the triangle quiver $Q^T$ or the almost triangle quiver $Q^{T'}$, which are both 
excluded from our considerations at the moment. \end{proof} \medskip   

Now, the following lemma describes the radical quotient for vertices of type A. 

\begin{lemma}\label{lem:gen} If $i$ is a $2$-vertex of type A, $\alpha,\ba$ arrows starting at $i$, then 
the quotients $e_iJ^k/e_iJ^{k+1}$, $k\geqslant 1$, are spanned by the cosets of monomials along the 
$g$-cycles of $\alpha$ and $\ba$. In particular, we have $\dim_K e_iJ^k/e_iJ^{k+1}\leq 2$.  
\end{lemma}

\begin{proof} By the assumption, vertices $j=t(\alpha),k=t(\ba)$ are $2$-regular, and hence, also $x, y$ are 
$2$-vertices (see Proposition \ref{prop:4.3}). We do not need to consider the case when there are two relations 
of type C starting from vertex $i$. Indeed, suppose there are two relations of type C starting from $i$ and 
ending at vertices $x, y$. In fact, $Q$ has the following subquiver  
$$
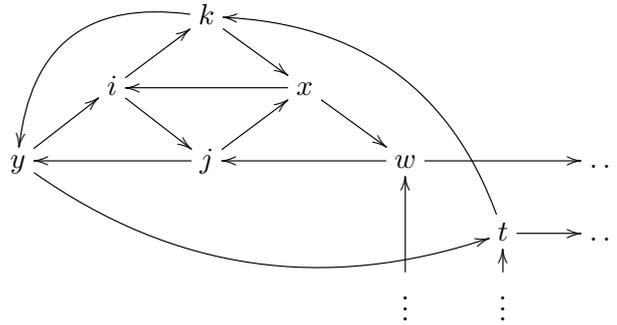
\xymatrix@R=0.45cm{&& k \ar[rd]\ar@/_25pt/[lldd] &&&& \\ & i \ar[ru] \ar[rd] && \ar[ll] x\ar[rd] &&& \\ 
y\ar[ru]\ar@/_25pt/[rrrrrd] && \ar[ll] j\ar[ru] && \ar[ll] w \ar[rr] && \dots \\ &&&&& t \ar[r] \ar@/_25pt/[llluuu]& \dots \\ 
&&&& \vdots\ar[uu] & \vdots \ar[u]&}$$
If $w=t$ then this subquiver is biregular (equal the tetrahedral quiver) and is equal to $Q$, so we do not need to 
consider the case. Otherwise, $w \neq t$, and then $w, t$ cannot be 1-vertices, this follows from the shape of 
blocks containing $1$-vertices (Proposition \ref{prop:4.3}). Therefore they are $2$-vertices, and we can then 
apply  \cite[Lemma 7.1]{AGQT}; in particular $\dim e_i\La = 6$ and the radical quotients are spanned by required 
cosets. \smallskip 

We will now show the claim using induction on $k\geqslant 2$ (the case $k=1$ is obvious) in two cases: 
\begin{enumerate}
\item[a)] there is no type C relation starting from $i$;  
\item[b)] there is (exactly) one type C relation $\alpha f(\alpha) + \ba g(\ba)\in J^3$ starting from $i$. 
\end{enumerate} \medskip 

Therefore, we may choose one arrow, say $\ba$, such that $\ba f(\ba)\in J^3$ and the second arrow satisfies 
$\alpha f(\alpha)\in J^3$ or $\alpha f(\alpha)\equiv \ba  g(\ba)$ modulo $J^3$, in case a) or b), respectively.
We use notation $\Theta_k,\bar{\Theta}_k$ for the paths $\Theta_k(\alpha),\Theta_k(\ba)$ (of 
length $k$) along the $g$-orbits of $\alpha,\ba$. Moreover, let $\sigma_k=g^{k-1}(\alpha)$ denote the last 
arrow on $\Theta_k$ and $\rho_k=g^{k-1}(\ba)$ the last arrow on $\bar{\Theta}_k$. \medskip

Now, we claim that for all $k\geqslant 2$ the following condition holds 
$$\bar{\Theta}_{k-1}f(\rho_{k-1})\in J^{k+1}\mbox{ and }\Theta_{k-1}f(\sigma_{k-1})\in J^{k+1} 
\mbox{ or } \equiv \bar{\Theta}_{k} 
\mbox{ (modulo $J^{k+1}$) } ,\leqno{(*)}$$ 
if the paths are defined. It will also imply that $e_iJ^k=\langle \Theta_k,\bar{\Theta}_k \rangle + e_iJ^{k+1}$, 
for any $k\geqslant 2$, so then the claim follows. \smallskip 

If $k=2$, we have $\Theta_{k-1}f(\sigma_{k-1})=\alpha f(\alpha)$ and $\bar{\Theta}_{k-1} f(\rho_{k-1})=\ba f(\ba)$, 
so the condition $(*)$ holds, by the assumption. Note only that all paths in $e_iJ^k$ are 
$\Theta_{k},\bar{\Theta}_k$ and possibly $\Theta_{k-1}f(\sigma_{k-1})$ or $\bar{\Theta}_{k-1}f(\rho_{k-1})$. 
By $(*)$, the last two paths modulo $J^{k+1}$ are $0$ or $\bar{\Theta}_k$, so $e_iJ^k/e_iJ^{k+1}$ is spanned 
by cosets of $\Theta_k,\bar{\Theta}_k$. In case $k=2$, the arrows $f(\alpha)$ and $f(\ba)$ are defined, by 
the assumption. 
\medskip 

Suppose $(*)$ holds for some $k\geqslant 2$ and denote by $\rho,\sigma$ the arrows $\rho_k,\sigma_k$. If 
$f(\rho)$ is defined, then we have two cases. \smallskip 

First, suppose $\rho f(\rho)\nprec I$, or equivalently, 
$s(\rho)$ is a $1$-vertex. Then the arrow $\theta=\rho_{k-1}$ is $\alpha_i,\nu_i$ or $\ve_i$, and hence 
$$\bar{\Theta}_k f(\rho)=\bar{\Theta}_{k-2}\theta g(\theta) f(g(\theta))=\Theta_{k-2}{\bf A}_{\bar{\theta}}.$$ 
By the assumptions, we must have $k\geqslant 3$, since otherwise, $\ba=\rho_1$ is an arrow with a $1$-regular 
target $t(\rho_1)=s(\rho_2)$, which contradicts the assumption on $i$. Therefore, we conclude that 
$\bar{\Theta}_{k-2}$ has at least one arrow, so the path $\bar{\Theta}_{k-2}{\bf A}_{\bar{\theta}}$ contains 
a subpath $g^{-1}(\rho){\bf A}_{\bar{\theta}}$, which is in $J^5$, by Lemmas \ref{gfV2} and 
\ref{gfV1}. As a result, we get $\bar{\Theta}_k f(\rho_k)\in J^{k+2}$. \smallskip 

Next, let $\rho f(\rho)\prec I$, so that $\rho f(\rho)\in J^3$ or we have a type C relation 
$\rho f(\rho)+\bar{\rho} g(\bar{\rho})\in J^3$. In the first case, it is clear that 
$\bar{\Theta}_{k}f(\rho)=\bar{\Theta}_{k-1}\rho f(\rho)\in J^{k+2}$. In the second, we obtain that 
$$\bar{\Theta}_{k}f(\rho)\equiv \bar{\Theta}_{k-1}\bar{\rho} g(\bar{\rho})=
\bar{\Theta}_{k-1}f(\rho_{k-1})g(\bar{\rho})\in J^{k+2},$$ 
because $\bar{\Theta}_{k-1}f(\rho_{k-1})$ is defined and belongs to $J^{k+1}$, by the inductive assumption. \medskip 

Now, assume that $f(\sigma)$ is defined. As above, in case $\sigma f(\sigma)\nprec I$, we have 
$\Theta_k f(\sigma)=\Theta_{k-2}{\bf A}_{\bar{\theta}}$, for $\theta=\sigma_{k-1}$ being $\alpha_i,\nu_i$ 
or $\ve_i$. In this case, we have $k\geqslant 3$, by the assumption on $i$, so $\Theta_k f(\sigma)\in J^{k+2}$, 
due to Lemmas \ref{gfV2}, \ref{gfV1}. \smallskip 

If the path $\sigma f(\sigma)$ is involved in a minimal relation, then it is either in $J^3$, or there is a 
type C relation $\sigma f(\sigma)+\bar{\sigma} g(\bar{\sigma})\in J^3$. Similarly as for $\bar{\Theta}_{k}f(\rho)$, 
we obtain $\Theta_k f(\sigma)=\Theta_{k-1} \sigma f(\sigma)\in J^{k+2}$, in the first case, and  
$$\Theta_k f(\sigma_k)=\Theta_{k-1}\equiv \Theta_{k-1}\bar{\sigma} g(\bar{\sigma})=
\Theta_{k-1}f(\sigma_{k-1})g(\bar{\sigma})$$ 
modulo $J^{k+2}$, in the second. For the latter, we have $\Theta_{k-1}f(\sigma_{k-1})\in J^{k+1}$ or it is 
$\equiv\bar{\Theta}_k$ modulo $J^{k+1}$, by the inductive assumption. Consequently, again $\Theta_k f(\sigma)\in J^{k+2}$, 
in the first case, but in the second, we get $\Theta_k f(\sigma)\equiv \bar{\Theta}_k g(\bar{\sigma})$. Finally, 
the last path belongs to $J^{k+2}$, if $g(\bar{\sigma})=f(\rho_k)$ (is defined), by the first part, and otherwise, 
it is $\bar{\Theta}_{k+1}$. This shows that $(*)$ holds for $k+1$, and consequently, for all $k\geqslant 2$. \end{proof} \medskip

As a result, for any $2$-regular vertex $i\in Q_0$ with $i^+=\{\alpha,\ba\}$, we have two cycles around $g$-orbits 
$C_\alpha=\alpha g(\alpha) \cdots g^{n_\alpha-1}(\alpha)$ and $C_{\ba}=\ba g(\ba) \cdots g^{n_{\ba}-1}(\ba)$, where 
$n_\alpha,n_{\ba}$ are the lengths of $g$-orbits of $\alpha$ and $\ba$, respectively. Because $\La$ is a symmetric algebra, 
there are integers $m_\alpha,m_{\ba}\geqslant 1$ such that $C_\alpha^{m_\alpha+1}=0$ and $C_{\ba}^{m_{\ba}+1}=0$, 
and then the cycles $B_\alpha:=C_{\alpha}^{m_\alpha}$ and $B_{\ba}:=C_{\ba}^{m_{\ba}}$ generate the socle $\soc(P_i)$ 
of $P_i$. 

\begin{cor} If $i$ is a $2$-vertex of type A with $i^+=\{\alpha,\ba\}$, then $e_i\La$ is generated by the 
set $$\cB_i:=\{\Theta_k(\alpha); \ 0\leq k \leq m_\alpha n_\alpha\}\cup\{\Theta_k(\ba); \ 0< k < m_{\ba} n_{\ba}\}.$$ 
In particular, we have $\dim_K e_i\La\leqslant n_\alpha m_\alpha + n_{\ba}m_{\ba}$. 
\end{cor} \medskip 

The last step is to extend the above to all $2$-vertices. 

\begin{lemma}\label{lem:2vertex} If $i$ is a $2$-vertex, then the set $\cB_i$ from the above corollary generates $e_i\La$. 
\end{lemma} 

\begin{proof} Clearly, it is sufficient to prove the claim for a $2$-vertex of type 
B or C. So assume that $i$ has at least one $1$-regular successor, say $j$, and $j,x$ are contained 
in a block of type V$_1$ or V$_2$. Then $\alpha g(\alpha)=\Theta_2$ is the unique path of length $2$ starting 
from $\alpha$, and we have paths starting from $\ba$: $\ba g(\ba)=\bar{\Theta}_2$ and possibly $\ba f(\ba)$, 
if $f(\ba)$ exists, equivalently, the second successor $k=t(\ba)$ is a $2$-vertex (we use the same 
notations as in the proof of Lemma \ref{lem:gen}). Note that $\ba f(\ba)\prec I$, if defined, because 
$i=s(\ba)$ is a $2$-vertex. In this case, $\ba f(\ba)={\bf A}_{\alpha}\in J^3$ or $\ba f(\ba)\equiv 
\alpha g(\alpha)=\Theta_2$. In both cases, the cosets of $\Theta_2,\bar{\Theta}_2$ span $e_i J^2/e_i J^3$, 
but we may have $\bar{\Theta}_2\in J^3$, so $\bar{\Theta}_2$ is zero in the quotient. This happens for 
example if ${\bf A}_{\ba}$ has a summand of length $2$, and then ${\bf A}_{\ba}\equiv \bar{\Theta}_2$, 
so $\bar{\Theta}_2\equiv \alpha g(\alpha) f(g(\alpha))\in J^3$. \medskip 

Suppose first that $t(\ba)$ is a $2$-vertex i.e. $i$ is of type B. In this case $\ba f(\ba)\prec I$, so 
$\ba f(\ba)={\bf A}_{\alpha}$ is either in $J^3$ or it is $\Theta_2$ (modulo $J^3$). 

{\it Case a)} Let $\ba f(\ba)\in J^3$. Then the cosets $\Theta_2$ and $\bar{\Theta}_2$ form a basis of 
the quotient $e_iJ^2/e_iJ^3$. Moreover, one can show using induction that 
$$\bar{\Theta}_{k-1}f(\rho_{k-1})\in J^{k+1},\leqno{(*)}$$ 
if the path is defined, for any $k\geqslant 2$. It is clear for $k=2$, and for $k=3$, we have 
$\bar{\Theta}_2f(\rho_2)=\ba g(\ba) f(g(\ba))\in J^4$, because $g(\ba)f(g(\ba))=\rho_2 f(\rho_2)\prec I$ 
($t(\ba)=s(g(\ba))$ is not a $1$-vertex), hence $\rho_2 f(\rho_2)\in J^3$ or $\rho_2 f(\rho_2)\equiv 
\bar{\rho}_2 g(\bar{\rho}_2)$ (modulo $J^3$), and then $\ba g(\ba) f(g(\ba))\equiv \ba f(\ba) g(\bar{\rho}_2)\in J^4$, 
because $\ba f(\ba)\in J^3$, by the assumption. Now, for $k\geqslant 3$, we may have $\rho_k f(\rho_k)\nprec I$, but  
then $\bar{\Theta}_kf(\rho_k)$ contains a subpath $g^{-1}(\eta)\eta g(\eta) f(g(\eta))$ of length $4$, for some 
$\eta=\alpha_i,\nu_i$ or $\ve_i$, which belongs to $J^5$, by Lemma \ref{gfV2} or \ref{gfV1}. Therefore, we 
get $\bar{\Theta}_k f(\rho_k)\in J^{k+2}$, in case $\rho_k f(\rho_k)\nprec I$. Otherwise, either $\rho_k f(\rho_k)\in J^3$, 
and then $\bar{\Theta}_kf(\rho_k)\in J^{k+2}$, or we have a type C relation $\rho_k f(\rho_k)\equiv \bar{\rho_k} g(\bar{\rho_k})$, 
modulo $J^3$ (with $\bar{\rho_k}=f(\rho_{k-1})$), which also forces $\bar{\Theta}_kf(\rho_k)\equiv 
\bar{\Theta}_{k-1} f(\rho_{k-1}) g(\bar{\rho_k})\in J^{k+2}$, since $\bar{\Theta}_{k-1} f(\rho_{k-1})\in J^{k+1}$, 
by the inductive assumption. \smallskip 

In the next step, we show that for any $k\geqslant 3$, we have 
$\Theta_{k-1} f(\sigma_{k-1})\in J^{k+1}$ or it is $\bar{\Theta}_k$ or $\bar{\Theta}_{k-1}$ 
(modulo $J^{k+1}$). Indeed, the case $k=3$ follows from Lemma \ref{gfV2} or \ref{gfV1}. If $k\geqslant 3$, 
then as above, we can assume $\sigma_k f(\sigma_k)\prec I$, and then either $\sigma_k f(\sigma_k)\in J^3$ 
or $\sigma_kf(\sigma_k)\equiv \bar{\sigma_k} g(\bar{\sigma_k})=f(\sigma_{k-1}) g(\bar{\sigma}_k)$ modulo $J^3$. In the 
first case, we get $\Theta_k f(\sigma_k)\in J^{k+2}$, whereas in the second, we obtain that  
$$\Theta_k f(\sigma_k)=\Theta_{k-1}f(\sigma_{k-1}) g(\bar{\sigma}_k).$$ 
By the inductive assumption, the path $\Theta_{k-1}f(\sigma_{k-1})$ is in $J^{k+1}$ or it is $\bar{\Theta}_{k}$ or 
$\bar{\Theta}_{k-1}$ modulo $J^{k+1}$. Obviously, $\Theta_k f(\sigma_k)\in J^{k+2}$, in the first case. In the second, 
we get $\Theta_k f(\sigma_k)\equiv \bar{\Theta}_k g(\bar{\sigma}_k)$, which is in $J^{k+2}$, by $(*)$, or 
it is $\bar{\Theta}_{k+1}$. In the third case, we have $\Theta_k f(\sigma_k)\equiv\bar{\Theta}_{k-1}g(\bar{\sigma}_k)$, 
and this may happen only if $\Theta_2 f(\sigma_2)\equiv\bar{\Theta}_2$ (then $\alpha$ lies in a critical block). 
But then $\ba$ is an arrow $i\to t( g(\alpha))$ 
lying in a triangle with $1$-vertex (forming a $g$-orbit of length $3$) and $\bar{\Theta}_2\in J^3$, so $\bar{\Theta}_{k-1} 
\in J^{k}$, and either $\Theta_k f(\sigma_k)$ must be $\bar{\Theta}_k$ or it is 
$\bar{\Theta}_{k-1} f(\rho_{k-1})$. It is sufficient to see that the last path is in $J^{k+2}$. To see this, 
note that $\bar{\Theta}_{k-1}$ is a path along the $g$-orbit of $\ba$, consisting of arrows $\ba,g(\ba)=\nu_l$ 
and $g^2(\ba)=\delta_l=\gamma$ (if $\alpha=\alpha_l$). The arrow $\rho_{k-1}$ cannot be $\nu_k$, since then $f(\rho_{k-1})$ 
is not defined, hence we have to consider two cases: $\rho_{k-1}=\ba$ or $\rho_{k-1}=\delta_k$. In the first case, 
we have $\bar{\Theta}_{k-1} f(\rho_{k-1})=\bar{\Theta}_{k-2}\ba f(\ba)\in J^{k+2}$, because $\ba f(\ba)\in J^4$. 
The latter comes from the fact that $\ba f(\ba)={\bf A}_\alpha$, and paths starting with $\alpha$ and ending with 
$f(\alpha)^*$ have length at least $4$. Now, if $\rho_{k-1}=\delta_l$, then $k\geqslant 4$
$$\bar{\Theta}_{k-1} f(\rho_{k-1})=\bar{\Theta}_{k-3} \nu_l\delta_l\alpha_l=\bar{\Theta}_{k-3}{\bf A}_{\bar{\nu}_l},$$ 
due to relations in $\La$ (see \ref{summary}), and $\Theta_{k-3}$ ends with $\ba$. Moreover, we get 
${\bf A}_{\bar{\nu}_l}\in J^3$, because paths starting with $\bar{\nu}_l$ and ending with $\alpha_l=\alpha$ are 
of length $\geqslant 3$. Consequently, $\bar{\Theta}_{k-1} f(\rho_{k-1})$ is a combination of paths of length 
$\geqslant k$ containing a subpath of the form $\rho_{k-3}\bar{\nu}_l=\ba f(\ba)$, which is in $J^4$, so  
the claim follows. \smallskip 

Summing up this case, it has been shown that either $e_iJ^k=\langle\Theta_k,\bar{\Theta}_{k}\rangle+e_iJ^{k+1}$, 
for any $k\geqslant 3$, or $e_iJ^k=\langle\Theta_k,\bar{\Theta}_{k-1}\rangle+e_iJ^{k+1}$. \medskip 

{\it Case b)} Suppose $\ba f(\ba)={\bf A}_\alpha\equiv\Theta_2$ (modulo $J^3$). Then $\alpha=\nu_i$ lies in a critical 
block. In this case, we prove by induction that  
$$\Theta_{k-1}f(\sigma_{k-1})\in J^{k+1} \mbox{ and } \bar{\Theta}_{k-1}f(\rho_k)\in J^{k+1}\mbox{ or }\equiv\Theta_k,$$ 
for all $k\geqslant 2$. For $k=2$, the first path is not defined, because $t(\alpha)=j$ is a $1$-vertex, and the 
second satisfies the condition, by the assumption. Now, for $k=3$, we have $\Theta_2 f(\sigma_2)=\alpha g(\alpha) f(g(\alpha))$ 
and this path belongs to $J^4$, by Lemma \ref{gfV2} (because the shortest path along $g$-cycle starting with $\ba$ 
and ending with $f(g(\alpha))$ has length $\geqslant 4$). Similarly, $\bar{\Theta}_2f(\rho_2)\in J^4$, if 
$\rho_2 f(\rho_2)\in J^3$, and otherwise, we have $\rho_2 f(\rho_2)\prec I$ (since $t(\ba)=s(\rho_2)$ is not $1$-regular), 
so $\rho_2 f(\rho_2)\equiv \bar{\rho_2} g(\bar{\rho_2})$, and hence 
$\bar{\Theta}_2 f(\rho_2)=\ba f(\ba) g(f(\ba))\equiv \Theta_2 g(f(\ba))$ which is in $J^4$ or it is $\Theta_3$, modulo 
$J^4$. We may repeat similar induction for $k\geqslant 3$ as in Case a) (but $\sigma$ and $\rho$ are interchanged).   

\bigskip  

Finally, it remains to consider the case, when $i$ is a vertex of type C, that is, $k=t(\ba)$ is also 
a $1$-vertex. Without loss of generality, we assume both $\alpha,\ba$ are in blocks of type $V_2$. Of course,  
$f(\alpha),f(\ba)$ are not defined, hence $\Theta_2=\alpha g(\alpha)$ and $\bar{\Theta}_2=\ba g(\ba)$ form a 
basis of $e_iJ^2$ modulo $J^3$. Since $Q$ is not the spherical quiver $Q^S$ (nor the triangle quiver), we infer that 
the shortest path in $Q$ starting with $\alpha$ and ending with $f(g(\ba))$ has length $\geqslant 4$. Dually, every 
path in $Q$ of the form $\ba\cdots f(g(\alpha))$ has length $\geqslant 4$, and therefore, we have 
${\bf A}_{\alpha},{\bf A}_{\ba}\in J^4$ 
(see also Section \ref{summary}). This shows that $\Theta_2 f(\sigma_2)=\alpha g(\alpha) f(g(\alpha))={\bf A}_{\ba}\in J^4$ 
and $\bar{\Theta}_2 f(\rho_2)={\bf A}_\alpha\in J^4$. Following similar induction as in previous cases (using Lemma 
\ref{gfV2}), we can show that $\Theta_{k-1}f(\sigma_{k-1})\in J^{k+1}$ and $\bar{\Theta}_{k-1}f(\rho_{k-1})\in J^{k+1}$, 
for all $k\geqslant 3$. The same can be proved in case one of the blocks containing 
$\alpha$ or $\ba$ is of type $V_1$, but we apply Lemma \ref{gfV1} and the assumption that $Q$ is different from 
the triangle quiver $Q^T$. It follows that every quotient $e_iJ^k/e_iJ^{k+1}$ is generated by the cosets of 
$\Theta_k$ and $\bar{\Theta}_k$, and hence, $\cB_i$ generates the projective $P_i=e_i\La$. The proof is now 
complete. \end{proof}

\subsection{Bases.}\label{subsec:5.3} Now, we are ready to give the bases of projective modules. First, we prove the 
following result providing the $2$-regular case.  

\begin{lemma}\label{basis2reg} If $i$ is a $2$-vertex with $i^+=\{\alpha,\ba\}$, then the set of all initial 
submonomials of $B_\alpha$ and $B_{\ba}$ is a basis of $e_i\La$. In particular, we have 
$\dim_K e_i\La=n_\alpha m_\alpha + n_{\ba}m_{\ba}$. \end{lemma} 

\begin{proof} We will show that $e_i\La$ has a required basis exploiting the radical quotients. We keep notation 
from the previous two lemmas. It follows from the proofs of Lemmas \ref{lem:gen} and \ref{lem:2vertex} that every 
quotient $e_i J^k/e_iJ^{k+1}$, for $k\geqslant 2$, is spanned by the (at most) two cosets of initial submonomials 
$\Theta_k$ of $B_\alpha$ and $\bar{\Theta}_k$ of $B_{\ba}$, except the case when $i$ is of type B and 
$\alpha g(\alpha)f(g(\alpha))\equiv\Theta_2$ (then $e_iJ^k/e_iJ^{k+1}$ is spanned by the cosets of $\Theta_{k}$ 
and $\bar{\Theta}_{k-1}$). For simplicity, we assume that quotients are generated by the paths of the same length, 
but the proof can be repeated in the remaining case. \smallskip 

After possibly relabelling the arrows, we may assume that $m_\alpha n_\alpha\leqslant m_{\ba}n_{\ba}$. 
We claim that for any $k<m_\alpha n_\alpha$, the cosets of paths $\Theta_{k},\bar{\Theta}_{k}$ are 
independent generators of $e_iJ^k/e_iJ^{k+1}$. Indeed, observe first that $\Theta_k$ is not in $J^{k+1}$, i.e. 
$\Theta_k$ is non-zero in the quotient. Suppose to the contrary that $\Theta_k\in J^{k+1}$. Then 
$\Theta_k\equiv a\Theta_{k+1}+b\bar{\Theta}_{k+1}$, modulo $J^{k+2}$, for some $a,b\in K$, and we can assume 
that $a=0$ (otherwise, we get $\Theta_k(1-a\sigma_{k+1})\equiv b\bar{\Theta}_k$, so after rescaling 
by a unit, we get the required relation). We can continue with longer summands $\Theta_{k+2},\bar{\Theta}_{k+2}\in J^{k+2}$ 
(generating the next radical quotient), and after grouping all summands starting from $\bar{\Theta}_k$ (or $\Theta_k$) 
in the form $\Theta_k(1+z)$, $z\in J$, we obtain that in fact $\Theta_k=\lambda\bar{\Theta}_{k+1}$ in $\La$, for 
some $\lambda\in K$, or equivalently, $\Theta_k-\lambda\bar{\Theta}_{k+1}\in I$. But we know the shape of 
minimal relations, which involve paths of length $2$ or $3$ and the paths ${\bf A}_\eta=c_\eta \Theta_m(\eta)$ 
along $g$-cycles. Therefore, it is not possible to get an equality of the form 
$$\Theta_k-\lambda\bar{\Theta}_{k+1}=\sum_{s\geqslant 1}\lambda_s u_s\rho_s v_s$$     
in $KQ$, where $\lambda_s\in K$, $u_s,v_s$ are paths and $\rho_s$ are minimal relations. This shows that 
indeed $\Theta_k\notin J^{k+1}$. Similarly, one can check that also $\bar{\Theta}_{k}\notin J^{k+1}$, so we 
have two (non-zero) generators of $e_iJ^k$ (modulo $J^{k+1}$). It remains to see that $\Theta_k$ and $\bar{\Theta}_{k}$ 
are independent in $e_iJ^k/e_iJ^{k+1}$. This is a consequence of the above non-vanishing property. Indeed, 
if $\Theta_k,\bar{\Theta}_k$ are dependent, then $\Theta_k\equiv\bar{\Theta}_k$ modulo $J^{k+1}$, which 
gives $\Theta_{k+1}=\Theta_k\sigma_{k+1}\equiv \bar{\Theta}_k f(\rho_k)$. But this is impossible, since the 
last path belongs to $J^{k+2}$, in all cases except $\ba f(\ba)\notin J^3$ (see proofs of Lemmas \ref{lem:gen} and 
\ref{lem:2vertex}), and then $\Theta_kf(\sigma_k)\in J^{k+2}$, which also gives a contradiction, because 
$\bar{\Theta}_{k+1}=\bar{\Theta}_k\rho_{k+1}\equiv \Theta_k f(\sigma_k)$. \medskip

Therefore, we have proved that $\dim_K e_iJ^k/J^{k+1}=2$, for all $k\in\{1,\dots,m_\alpha n_\alpha-1\}$. Now, let 
$k\geqslant m_\alpha n_\alpha$. Then the path $\Theta_{k}=B_\alpha \alpha \cdots \equiv B_{\ba} \alpha \cdots$ is either 
zero or $\equiv B_{\ba}$, hence we conclude that $\dim_Ke_iJ^k/J^{k+1}=1$ and the quotient is generated by the coset of 
$\bar{\Theta}_{k}$, for all $k\in\{m_{\alpha} n_{\alpha},\dots,m_{\ba} n_{\ba}\}$. Consquently, counting all the dimensions, 
we get 
$$\dim_K e_i\La=\sum_{k=0}^{m_{\ba}n_{\ba}}\dim_Ke_iJ^k/e_iJ^{k+1}=1+2(m_\alpha n_\alpha -1)+(m_{\ba}n_{\ba}-m_\alpha n_\alpha +1)=$$ 
$$=m_\alpha n_\alpha + m_{\ba} n_{\ba}.$$ 
This completes the proof. \end{proof} \medskip 

Now, it remains to consider the $1$-regular case. 

\begin{lemma}\label{basis1reg} Let $i$ be a $1$-vertex of $Q$ with $i^+=\{\alpha\}$. Then the set $\cB_i$ of all 
initial submonomials of $B_\alpha$ together with $\alpha f(\alpha)$ form a basis of $e_i\La$. In particular, 
we have $\dim_Ke_i\La=m_\alpha n_\alpha +2$. \end{lemma} 

\begin{proof} It is sufficient to prove that $\cB_i$ generates $e_i\La$, since initial submonomials of a socle 
cycle are always independent (and have different targets from $t(f(\alpha))$). We will use notation introduced 
in the proof of Lemma \ref{lem:gen} for arrows $\sigma_1,\sigma_2,\dots,\sigma_k=g^{k-1}(\alpha)$, and paths 
$\Theta_k=\sigma_1\cdots \sigma_k$. Of course, $e_i J/e_i J^2$ is spanned by the coset of arrow $\alpha=\Theta_1$ 
and $e_iJ^2/e_i J^3$ has basis given by the coset of $\alpha f(\alpha)$ and the coset of $\alpha g(\alpha)=\Theta_2$ 
if $\Theta_2\notin J^3$ (this may happen for $\alpha f(\alpha)g(f(\alpha))\equiv\Theta_2$, see below). \smallskip 

Since the successor $j=t(\alpha)$ of $i$ is $2$-regular (see Lemma \ref{lem:3.5}), we conclude from Lemma 
\ref{basis2reg} that $e_iJ$ is generated by paths of the form 
$\alpha\Theta_{k-1}(f(\alpha))$ and $\alpha\Theta_{k-1}(g(\alpha))=\Theta_{k}$, for $k\geqslant 1$. Now, observe 
that $\alpha\Theta_2(f(\alpha))=\alpha f(\alpha) g(f(\alpha))={\bf A}_\alpha\in J^3$, due to relations in $I$ (see 
Section \ref{summary}). We claim that all paths $p\prec {\bf A}_\alpha$ start with $\alpha g(\alpha)$, and hence 
they are of the form $\Theta_2$ or $\Theta_3$ (modulo $J^4$). Indeed, every summand starting with $\alpha f(\alpha)$ 
contains a subpath $\alpha f(\alpha) g(f(\alpha))$, hence we may rewrite the relation involving 
$\alpha f(\alpha)g(f(\alpha))$ as $\alpha f(\alpha) g(f(\alpha))(1-z)={\bf A}_\alpha$, with new ${\bf A}_\alpha$ 
having only summands starting with $\Theta_2$, and $1-z$ a unit in the local algebra $e_x\La e_x$, $x=s(\gamma)$, 
where $\gamma:x\to i$ is the unique arrow ending at $i$. Further, if $f(\sigma_2)$ is defined, then we 
have $\sigma_2 f(\sigma_2)\prec I$, since $s(\sigma_2)=t(\sigma_1)=j$ is $2$-regular, and then 
$\Theta_2 f(\sigma_2)\in J^4$, provided that $\sigma_2 f(\sigma_2)\in J^3$, and otherwise, we have 
$\sigma_2f(\sigma_2)\equiv\bar{\sigma}_2g(\bar{\sigma}_2)$, where $\bar{\sigma}_2=f(\sigma_1)$, so 
we obtain 
$$\Theta_2 f(\sigma_2)\equiv \alpha f(\alpha) g(f(\alpha))$$ 
which is $\Theta_2$ or $\Theta_3$, by the previous considerations. It follows that $e_iJ^3/e_iJ^4$ is generated 
by the coset of $\Theta_3$ (or by the coset of $\Theta_2$, and then $\alpha=\nu_i$, $\Theta_2=\alpha f(\alpha)g(\alpha)\in J^3$ 
and $\Theta_3\in J^4$). In the first case, one can check that $\Theta_2f(\sigma_2)\in J^4$. 
In fact, otherwise we have $\Theta_2f(\sigma_2)\equiv\Theta_3=\Theta_2 g(\sigma_2)$ modulo $J^4$, which is 
impossible, because $f(\sigma_2)$ and $g(\sigma_2)=g(f(\alpha))^*$ end at different vertices (we have an 
arrow from a $1$-vertex to $x$, so there are no double arrows ending at $x$). If $\Theta_2$ generates $e_iJ^3$ 
modulo $J^4$, then obviously, we have $\Theta_2f(\sigma_2)\in J^4$. \smallskip 

Finally, one can show using induction that for any $k\geqslant 3$, the space $e_iJ^k/e_iJ^{k+1}$ is spanned by the coset 
of $\Theta_k$ or coset of $\Theta_{k-1}$ (and $\Theta_{k-1}f(\sigma_{k-1})\in J^{k+1}$ in both cases). For $k=3$, it 
follows from the above arguments. If $\Theta_k$ generates $e_iJ^k$ (modulo $J^{k+1}$), $k\geqslant 3$, then we have two 
generators of $e_iJ^{k+1}$, which are $\Theta_{k+1}$, and $\Theta_kf(\sigma_k)$, when $f(\sigma_k)$ exists. For the latter, 
we have $\Theta_kf(\sigma_k)\in J^{k+2}$, by Lemmas \ref{gfV2} or \ref{gfV1}, in case $\sigma_k f(\sigma_k)\nprec I$. 
Otherwise, either $\sigma_k f(\sigma_k)\in J^3$, and then also $\Theta_k f(\sigma_k)\in J^{k+2}$, or else there is a 
type C relation $\sigma_kf(\sigma_k)\equiv\bar{\sigma}_k g(\bar{\sigma}_k)$ modulo $J^3$, $\bar{\sigma}_k=f(\sigma_{k-1})$. 
In the last case, we obtain $\Theta_kf(\sigma_k)\equiv\Theta_{k-1}f(\sigma_{k-1})g(\bar{\sigma}_k)\in J^{k+2}$, 
because $\Theta_{k-1}f(\sigma_{k-1})\in J^{k+1}$, by the inductive assumption. Therefore, we get $\Theta_k f(\sigma_k)\in J^{k+2}$ 
so the coset of $\Theta_{k+1}$ is a generator of $e_iJ^{k+1}/e_iJ^{k+2}$. In the case, when $\Theta_{k-1}\in J^k$ 
generates $e_iJ^k$ modulo $J^{k+1}$, one can easily check that $\Theta_k\in J^{k+1}$ generates $e_i J^{k+1}$ 
modulo $J^{k+2}$. In both cases, we get required generating sets for the radical quotients, and the claim follows. \end{proof}
\medskip 

Summing up, we have a description of bases of projective $\La$ modules. In particular, we can now say more 
about the elements ${\bf A}_{\ba}$. Indeed, if $i$ is a $2$-vertex, $\alpha,\ba$ arrows starting from $i$, then 
${\bf A}_{\alpha}\in \alpha\La$ forces that ${\bf A}_{\alpha}$ is a combination of paths $\Theta_k=\Theta_k(\alpha)$, 
along $g$-cycle of $\alpha$ (no summands starting with the other arrow). One can order summands of 
${\bf A}_{\alpha}$ with respect to the length, so that ${\bf A}_{\ba}=A_{\ba}(1+z)$, where $z\in e_x \La e_x$, 
$x=t(f(\alpha))$, and $A_{\ba}$ is the shortest path among the summands of ${\bf A}_{\ba}$. If ${\bf A}_\alpha$ has 
more than one summand, then $u=1+z$ is a unit of the local algebra $e_x\La e_x$, hence there is $c_{\alpha}\in K^*$ 
such that ${\bf A}_\alpha= c_{\alpha} A_{\alpha}$, where $A_{\ba}$ is a path along $g$-orbit of $\ba$. We may do the same 
for the second arrow $\ba$ obtaining a relation $\alpha f(\alpha)-c_{\ba} A_{\ba}$. It is easy to check using exact 
sequence for $S_i$, that the scalars $c_\bullet$ coincide along $g$-orbits, i.e. $c_{g(\alpha)}=c_\alpha$ and 
$c_{\ba}=c_{g(\ba)}$. This shows that there is a parameter function $c_\bullet:Q_1\to K^*$, constant on $g$-orbits. \smallskip 

The same can be done with relations $\rho^{(i)}_1,\dots,\rho^{(i)}_6$, where each ${\bf A}_\alpha$ has the form 
${\bf A}_\alpha=c_\alpha A_\alpha$ with $c_\alpha\in K$ non-zero and $A_\alpha$ a path along the $g$-cycle of $\alpha$ 
(see also Section \ref{summary}). Since $A_\alpha\in \La f(\ba)^*$ and $f(\ba)^*=g^{-1}(\gamma^*)=g^{-2}(\alpha)$, we 
conclude that $A_\alpha=(\alpha g(\alpha)\cdots g^{n_\alpha-1}(\alpha))^{m-1}\alpha\cdots g^{n_\alpha-2}(\alpha)$, 
for some $m\leqslant m_\alpha$. We will see in Section \ref{hat:construction} that in fact $m=m_\alpha$, for all arrows, 
which means that paths $A_\alpha$ are in the second socle.

\section{Proof of Main Theorem}\label{sec:6} 

In this section, we present the proof of the Main Theorem. We will first provide the proof in case when 
Gabriel quiver of the algebra is different from the spherical quiver $Q^S$ and the triangle (and almost triangle) 
quivers $Q^T$ and $Q^{T'}$; see Section \ref{sec:HSA}. This is done in Sections \ref{hat:construction}-\ref{hat:tame}. 
In the last Section \ref{subsec:6.5}, we complete the proof in the remaining cases. \medskip 

Let $\Lambda$ be an algebra with biregular Gabriel quiver $Q=Q_\Lambda$, i.e. each vertex of $Q$ is 
either $1$-regular or $2$-regular. We always assume that $Q$ has at least $3$ vertices. \smallskip 
  
Implication $(i)\Rightarrow (ii)$ follows from \cite[IV. Theorem 11.19]{SkY}. Note that $\La$ is always 
of infinite representation type, by the assumptions on the number of vertices and arrows; see \cite[Lemma 3.1]{note}. The 
implication $(iii)\Rightarrow (i)$ is a direct consequence of the results \cite[Theorems 1.1-1.3]{WSA-GV}. 
Hence it remains to prove that implication $(ii)\Rightarrow (iii)$ holds. \smallskip 

So assume $\Lambda$ is an algebra of generalized quaternion type, that is, $\Lambda$ is tame symmetric of 
infinite representation type, and each simple module in $\mod \Lambda$ is periodic of period $4$. We may assume that 
there exists at least one $1$-regular vertex in $Q$, since otherwise the claim follows from the Theorem \ref{MTGQT}. 
The idea of the proof is to construct directly (by quivers and relations) the algebra 
$\hat{\Lambda}(\hat{m})=K\hat{Q}/\hat{I}$ depending on a chosen collection of weights $\hat{m}$ such that 
\begin{itemize}
\item $\hat{Q}$ is a $2$-regular quiver with $\hat{Q}_0=Q_0$, obtained from $Q\subset\hat{Q}$ by adding loops and $2$-cycles, 
\item $\hat{Q}=Q_{\hat{\Lambda}(\hat{m})}$, if all weights are $\geqslant 2$, 
\item $\hat{\Lambda}(\hat{m})$ is a an algebra of generalized quaternion type with $\hat{\Lambda}({\bf 1})\cong\Lambda$, 
where ${\bf 1}$ means that all weights are equal to $1$. 
\end{itemize} 

\medskip 

The proof is divided into several steps, each in a separate subsection. 

\subsection{Construction of $\hat{\Lambda}$}\label{hat:construction}

We use notation introduced in Section \ref{bases}. \smallskip 

First, we define the quiver $\hat{Q}$. The set of vertices is the same as for $Q$: $\hat{Q}_0=Q_0$, and the 
set of arrows $\hat{Q}_1$ of $\hat{Q}$ is obtained from $Q_1$ by adding a loop $\rho_i$ at each vertex $y_i$ in a block 
$B_i'$, for $i\in\{1,\dots,q\}$, and a $2$-cycle $\xymatrix{c_i\ar@<+.4ex>[r]^{\xi_i} & d_i\ar@<+.4ex>[l]^{\mu_i}}$ in 
each block $B_i$, $i\in\{1,\dots,p\}$. In other words, every block $B_i'$ in $Q$ becomes a block 
$$\xymatrix@C=0.6cm@R=0.2cm{&\\ \ar@(lu, ld)[]_{\rho_i} y_i\ar@<.35ex>[r]^{\eta_i}&\ar@<.35ex>[l]^{\ve_i}x_i} $$ 
of type II in $\hat{Q}$, and each block $B_i$ in $Q$ becomes a glueing 
$$\xymatrix@R=0.3cm{&c_i\ar[rd]^{\beta_i}\ar@<-.37ex>[dd]_{\xi_i}& \\ 
a_i\ar[ru]^{\alpha_i} && b_i\ar[ld]^{\nu_i}\\ &d_i\ar[lu]^{\delta_i} \ar@<-.37ex>[uu]_{\mu_i} & }$$  
of two triangles in $\hat{Q}$. By the construction, the quiver $\hat{Q}$ is $2$-regular. It will follow later that 
this is actually a triangulation quiver. \medskip

Now, we will define the ideal $\hat{I}$ of $K\hat{Q}$. First, we need to fix natural numbers $m_i\geqslant 1$, 
for $i\in\{1,\dots,p\}$, and $m'_i\geqslant 1$, for $i\in\{1,\dots,q\}$. With this, we define $\hat{I}$ by generators 
as follows. \medskip 

We know from Section \ref{summary} that $I$ is generated by minimal relations of the form $\rho^{(i)}_1,\dots,\rho^{(i)}_4$, 
$i\in\{1,\dots,p\}$, relations $\rho^{(i)}_5,\rho^{(i)}_6$, for $i\in\{1,\dots,q\}$, and relations of the form 
$\alpha f(\alpha)-c_{\ba}A_{\ba}$, for arrows $\alpha\in Q_1$ between $2$-regular vertices. Let $R_0$ be the set of 
relations of the form $\alpha f(\alpha)-c_{\ba}A_{\ba}$, where $s(\alpha),t(\alpha)$ are $2$-regular. In this way, 
$I$ is generated by $R_0$ and relations $\rho_1^{(i)},\dots,\rho_6^{(i)}$. \smallskip 

Further, let $R_1$ be the set of relations in $K\hat{Q}$ of the form 
$$\alpha_i\xi_i-{\bf A}_{\overline{\alpha_i}}, \ \xi_i\delta_i-{\bf A}_{\beta_i}, \ \delta_i\alpha_i-(\mu_i\xi_i)^{m_i-1}\mu_i, \ 
\nu_i\mu_i-{\bf }{\bf A}_{\overline{\nu_i}}, \ \mu_i\beta_i-{\bf A}_{\delta_i}, \ \beta_i\nu_i-(\xi_i\mu_i)^{m_i-1}\xi_i, $$ 
$$\xi_i\delta_i\overline{\alpha_i}, \ \mu_i\beta_i\overline{\nu_i}, \ \alpha_i\xi_i\mu_i, \ \nu_i\mu_i\xi_i, \ 
\xi_i\mu_i\beta_i, \ \mu_i\xi_i\delta_i, \ \delta_i^*\alpha_i\xi_i, \  \beta_i^*\nu_i\mu_i $$  
for all $i\in\{1,\dots,p\}$. \smallskip 

Finally, let $R_2$ be the set of relations 
$$\ve_i\rho_i-A_{\overline{\ve_i}}, \ \rho_i\eta_i-A_{\eta_i}, \ \eta_i\ve_i-\rho_i^{m_i'-1}, \ \rho_i\eta_i\overline{\ve_i}, \ 
\ve_i\rho_i^2, \ \rho_i^2\eta_i, \ \eta_i^*\ve_i\rho_i,$$ 
for all $i\in\{1,\dots,q\}$. \smallskip 

Let $\hat{I}$ be an ideal of $K\hat{Q}$ generated by the relations in $R=R_0\cup R_1\cup R_2$. It is easy to 
see that $\hat{I}$ is admissible if and only if $m_i\geqslant 2,m_i'\geqslant 3$, for all possible $i$. \bigskip 

It follows directly from the definition of $\hat{I}$, that for any $i$ with $m_i\geqslant 2$ (respectively, with 
$m'_i\geqslant 3$), we have the following zero relations in $\hat{\La}$: 
$$\beta_i\nu_i\delta_i, \ \delta_i\alpha_i\beta_i, \ \alpha_i\beta_i\nu_i \ \mbox{and} \ \nu_i\delta_i\alpha_i 
\ (\mbox{respectively, } \eta_j\ve_j\eta_j \mbox{ and } \ve_j\eta_j\ve_j).$$ \medskip 

Now, we will show that $\hat{\La}\cong \La$ for trivial weights. Indeed, suppose all $m_i=1$ and $m_i'=2$. Then, relations 
in $R_1\cup R_2$ force $\xi_i=\beta_i\nu_i$, $\mu_i=\delta_i\alpha_i$ and $\rho_j=\eta_j\ve_j$, and hence the Gabriel 
quiver of $\hat{\La}$ is equal to $Q$. In particular, we have $\hat{\La}\cong KQ/J$, where $J$ is an ideal of $KQ$ 
generated by relations obtained from relations in $R$ via replacing arrows $\xi_i,\mu_i$ and $\rho_i$ by paths 
$\beta_i\nu_i$, $\delta_i\alpha_i$ and $\eta_i\ve_i$, respectively. Obviously, the ideal $J$ contains all relations 
in $R_0$ and all relations from $R_1\cup R_2$ after substitutions. Therefore, we have $I\subseteq J$. \smallskip 

Finally, we will prove that $J\subseteq I$. It is clear that the relations in $R_0$ remain relations in $I$. Further, the 
generators of $J$ obtained from the commutativity relations in $R_1\cup R_2$ after substitutions give exactly the 
relations $\rho^{(i)}_1,\dots,\rho^{(i)}_4,\rho^{(j)}_5,\rho^{(j)}_6$ in $I$. Consequently, we must show that the same 
holds for zero relations, i.e. that the paths 
$$\beta_i\nu_i\delta_i\overline{\alpha_i}, \ \delta_i\alpha_i\beta_i\overline{\nu_i}, \ \alpha_i\beta_i\nu_i\delta_i\alpha_i, 
\ \nu_i\delta_i\alpha_i\beta_i\nu_i, \ 
\beta_i\nu_i\delta_i\alpha_i\beta_i, \ \delta_i\alpha_i\beta_i\nu_i\delta_i, \ \delta_i^*\alpha_i\beta_i\nu_i, 
\  \beta_i^*\nu_i\delta_i\alpha_i, $$
and the paths 
$$ \eta_j\ve_j\eta_j\overline{\ve_j}, \ 
\ve_j\eta_j\ve_j\eta_j\ve_j, \ \eta_j\ve_j\eta_j\ve_j\eta_j, \ \eta_j^*\ve_j\eta_j\ve_j$$ 
are zero in $\La$. \smallskip 

To do this, it is sufficient to show that the following paths of length three: 
$$\alpha_i\beta_i\nu_i, \ \beta_i\nu_i\delta_i, \ \nu_i\delta_i\alpha_i, \ \delta_i\alpha_i\beta_i$$ 
and 
$$\eta_j\ve_j\eta_j, \ \ve_j\eta_j\ve_j$$ 
belong to the second socle of $\La$. Then all the paths of length five are clearly zero, and the paths of length four belong to the 
socle, so they are also zero, because all of them are not cycles. \medskip 

The rest part of this subsection is devoted to show that the above paths are in the second socle. We 
fix a block $(\alpha \ \beta \ \nu \ \delta)$ of type V$_2$ (with vertices $a,b,c,d$). We start with the following 
lemma.

\begin{lemma} 
Assume  that  the paths of length three around the block are $\prec I$. 
Assume also that there is no other path of length $3$ from vertex $a$ or vertex $c$ which is the lowest term of a minimal relation. Then there are
minimal relations satisfying
	$$\beta\nu\delta + \beta X=0 \ \ \mbox{ and } \
	\nu\delta\alpha + X\alpha = 0, \ \ (X\in J^3)$$
	$$\delta\alpha\beta + \delta Y = 0, \ \ \mbox{ and} \ 
	\alpha\beta\nu + Y\nu = 0, \ \ (Y\in J^3).
	$$
\end{lemma}

\begin{proof}

We can write the minimal relation involving $\beta\nu\delta$ in the stated form,
noting that $\beta$ is the only arrow starting at vertex $b$, then $X\in e_c\La e_d$. 
We claim that $X$ must be in $J^3$, otherwise there would be another path of length three in a minimal relation ending at vertex $a$. 

Consider now the exact sequence for the simple module $S_b$, this shows that
$\Omega^2(S_b) = \psi \La$ for some $\psi \in e_c\La e_a$. The minimal relation
gives $\psi = \beta\delta + X$. 
Then $\Omega^2(S_b) = \alpha\La$ and we get $\nu\delta\alpha + X\alpha=0$.

The second part is similar. \end{proof}

\medskip

\begin{rem}\normalfont When $Q$ is the spherical quiver, we have two mainimal relations from 1-vertices ending at the same 
$2$-vertex, and this is causing 	a special case, and in fact this is responsible for the existence of the 
Higher Spherical Algebras, which are GQT but not WSA's. \end{rem} \medskip

\begin{lemma} With the assumptions as in the
	previous Lemma, the paths of length three around
	the block of type $V_2$ belong to the second socle of the algebra.
\end{lemma}

\medskip

\begin{proof}

 (1) \ We start with $\beta\nu\delta$. 
The socle of $e_c\La$ is spanned by $B_{\beta} = (\beta \bar{\nu}\ldots \alpha)^m$. 
Then (by Lemma \ref{basis1reg}) the module $e_c\La$ has a basis given by
$$\{ [B_{\beta}]_t \mid 0\leq t \leq \ell(B_{\beta}) , \beta\nu\},$$ 
where $[B_{\beta}]_t:=\Theta_t(\beta)$ denotes the initial submonomial of $B_\beta$ of length $t$.

Now consider the element $\delta\alpha \in KQ$, we know $\delta\alpha\nprec I$. 
This implies that $e_d\La$ has a factor module which is uniserial of length three which is $U(S_d, S_a, S_c)$. 
Such a module must be isomorphic to a submodule of $e_c\La$.

This means that we have an element $\psi = e_c\psi e_d$ generating such
a submodule. Then 
$\psi\delta$ belongs to the second socle, that is, 
$$\psi\delta \equiv  B_{\beta}'$$
We can write  $\psi = \sum_{t} a_t[B_{\beta}]_t + \lambda \beta\nu$ where
the terms appearing the in the sum  end at vertex $d$. So if the
cycle of $\beta$ does not pass through $d$ then there is no such 
non-zero sum and then it
follows that $\beta\nu\delta$ is in the second socle. 

\medskip

Otherwise, there can be a monomial $[B_{\beta}]_t= [B_{\beta}]_te_d$ 
ending  in the arrow
$\nu$. If so then $[B_{\beta}]_t\delta = [B_{\beta}]_{t+1}$ but a linear
combination of such elements is not in the second socle. 

Then $\sum_t a_t[B_{\beta}]_t\delta + \lambda \beta\nu\delta$ is in the
second 
socle only if the sum is zero, then $\beta\nu\delta$ is in the second socle.

\bigskip
Similarly one shows that $\delta\alpha\beta$ is in the second socle of
the algebra.

(2) \ To show that the other two terms also are in the second socle,
we apply the previous Lemma.

Since $\beta\nu\delta$ is in 
the second socle, the element $\beta X$ also is in the second socle. This starts
at a 1-vertex and therefore $\beta X$ is a multiple of $B_{\beta}'$. This gives
that $\beta X\alpha$ is in the socle. Now $\alpha$ is the only way to get
a socle element from $B_{\beta}'$ and therefore 
$B_{\beta}'\alpha$ is non-zero in the socle, i.e. is a non-zero
multiple of $B_{\beta}$. By rotation, $X\alpha\beta$ is a non-zero socle element. Since $\beta$ is the only 
arrow starting at $c$ it follows that $X\alpha$ is in the second socle. This shows that $\nu\delta\alpha$ 
is in the second socle.

Similar arguments show that $\alpha\beta\nu$ is in the second socle. \end{proof}

Now consider a block of type $V_1$ in $Q$ with vertices $x,y$ (and arrows $\eta,\ve$). We assume $|Q_0|>2$ and then 
there are no loops at $x,y$, hence the paths of length $2$ are $\nprec I$ (see also Lemma \ref{lemB}).  

\begin{lemma} Assume the paths $\eta\ve\eta$ and $\ve\eta\ve$ are $\prec I$. Then there are minimal relations satisfying
	$$\eta\ve\eta + \eta X=0, \ \mbox{and} \ \ \ve\eta\ve + X\ve =0 \ (X\in J^3)$$
\end{lemma}

\begin{proof} We can write the minimal relation for $\eta\ve\eta$ in the form 
as stated. Then consider the exact sequence for $S_y$, the module
$\Omega^2(S_y)$ is cyclic and as a generator we can take $\ve\eta + X$. 
Since $\Omega^3(S_y)\cong \ve \La$ there is an arrow $\ve' = \ve u$ 
(for $u$ a unit) and $(\ve\eta + X)\ve'=0$, and then $\ve\eta\ve + X\ve =0$. \end{proof}

\medskip

\begin{lemma} The elements $\eta\ve\eta$ and $\ve\eta\ve$ belong to the second socle of the algebra.
\end{lemma}

\begin{proof}

 The socle of $e_y\La$ is spanned by an element $B_{\eta} = (\eta\bar{\ve}\ldots \eta^*\ve)^m$. By Lemma 5.8, the module $e_y\La$ has basis 
$$\{ [B_{\eta}]_t \mid 0\leq t\leq \ell(B_{\eta})\} \cup \{ \eta\ve\}$$
We know $\eta\ve\nprec I$ and then $e_y\La$ has uniserial factor module
$U(S_y, S_x, S_y)$ which then also is isomorphic to a 
submodule of $e_y\La$. 

Let $\psi = e_y\psi e_y$ be a generator for such a submodule, then 
$\psi \eta$ generates the second socle and hence 
$$\psi\eta \equiv B_{\eta}'$$
Write
$$\psi = \sum_t [B_{\eta}]_t + \lambda \eta\ve
$$
An element $[B_{\eta}]_t\eta$ is either zero, or is equal to $[B_{\eta}]_{t+1}$.
But if it is non-zero it does not generate the second socle
(the generator for the second socle must end in $\eta^*$. 
We deduce that the sum multiplied with $\eta$ must be zero, and then 
$\eta\nu\eta$ is in the second socle, as required. \end{proof} \medskip

 Moreover, we know that each path $\alpha f(\alpha)=c_{\ba}A_{\ba}$ belongs to the 
second socle \cite[see Proposition 9.1]{AGQT}. Therefore, $A_{\ba}$ must be a path of length $m_{\ba} n_{\ba} -1$, because otherwise, it is of length $k<m_{\ba} n_{\ba} -1$, and then $A_{\ba} g^k(\ba)$ is an element of the socle 
of length $k+1<m_{\ba} n_{\ba}$, so we get $B_{\ba}=0$, a contradiction. \smallskip  

Similarly, involving the second socle paths of length $3$ in blocks $V_1$ and $V_2$, we can prove that also the 
remaining paths $A_{\eta}$ have length $m_\eta n_\eta -1$. Concluding, for any arrow $\alpha\in Q_1$, we 
have $A_\alpha=\alpha g(\alpha)\cdots g^{m_\alpha n_\alpha -2}(\alpha)$. \medskip 

Finally, observe that for any arrow $\alpha$ of $\hat{Q}$, we have an $f$-orbit $(\alpha \ f(\alpha) \ f^2(\alpha))$ 
and one can easily see from the relations defining $\hat{\La}$ that then 
$$c_{\ba} B_{\ba} =c_{\ba} A_{\ba} f^2(\alpha) = \alpha f(\alpha) f^2(\alpha) = 
c_{g(\alpha)} \alpha A_{g(\alpha)} = c_{\alpha} B_\alpha.$$  

\subsection{Bases of projective $\hat{\La}$-modules} \label{hat:bases}

In the next step, we describe bases of the indecomposable projective $\hat{\La}$-modules. Fix a vertex 
$z\in \hat{Q}_0=Q_0$ and consider the projective module $\hat{P}_z=e_z\hat{\La}$. \medskip 

First assume that $z$ is different from $c_i,d_i$ and $y_i$, i.e. $z$ is $2$-regular in $Q$. Let $\alpha,\ba$ 
be the arrows in $Q$ starting from $z$. Then it follows from Lemma \ref{basis2reg} that the set $\cB_z$ of all 
initial submonomials of $B_\alpha$ and $B_{\ba}$ is a basis of $P_z=e_z\La$. \smallskip 

We claim that $\cB$ is also a basis for $\hat{P}_z$. Because $z$ is a $2$-vertex of $Q$, we infer that any 
path in $\hat{Q}$ starting from $z$ and containing arrows $\xi_i,\mu_i$ or $\rho_i$ must contain a subpath of 
the form $\alpha_i\xi_i=A_{\overline{\alpha_i}}$, $\nu_i\mu_i=A_{\overline{\nu_i}}$ or $\ve_j\rho_j=A_{\overline{\ve_j}}$. 
As a result, every path in $\hat{Q}$ staring from $z$ is a combination of paths in $Q$, and hence, $\cB$ 
generates $\hat{P}_z$. Now, it is sufficient to see that $\cB$ is also independent set in $\hat{\La}$. Indeed, 
if a non-zero combination of paths in $\cB$ is zero in $\hat{\La}$, then it belongs to $\hat{I}\setminus I$ 
(viewed as an element of $K\hat{Q}$), so it is generated by relations in $R_1\cup R_2$. But this is impossible, 
since all paths involved in relations from $R_1\cup R_2$ are not paths in $Q$. \medskip 

Finally, let $z$ be a $1$-regular vertex in $Q$. Let first $z=c_i$ or $d_i$. Since the arguments are dual, we 
will describe the basis of $\hat{P}_z$ only for $z=c_i$. We know from Lemma \ref{basis1reg}, that the module 
$e_z\La$ has a basis $\cB$ formed by all initial submonomials of $B_{\beta}$, where $\beta=\beta_i$, together 
with $\beta_i\nu_i$. As in $2$-regular case, every path in $\hat{Q}$ starting from $z$ and containing subpath 
of the form $\alpha_i\xi_i$, $\nu_i\mu_i$ or $\ve_j\rho_j$ can be written as a combination of paths in $Q$. 
Moreover, in $\hat{\La}$ we have $\beta_i\nu_i=(\xi_i\mu_i)^{m_i-1}\xi_i$, and hence, we conclude that $\hat{P}_z$ 
is generated by the initial submonomials of $B_\beta$ and initial submonomials of $B_{\xi}$, where $\xi=\xi_i$ and 
$B_{\xi}=(\xi_i\mu_i)^{m_i}$. As above, one can see that this set is also independent, due to the shape of 
generators in $R_1\cup R_2$. \medskip 

Applying analogous arguments, we conclude that $\hat{P}_{d_i}$ has a basis consisting of all initial submonomials 
of $B_{\delta_i}$ and $B_{\mu_i}=(\mu_i\xi_i)^{m_i}=c_{\delta_i} B_{\delta_i}$ (except, say $B_{\mu_i}$). \medskip 

Now, it remains to see that the same works for $z=y_j$, namely, the module $\hat{P}_{y_j}$ has a basis consisting 
of all initial submonomials of $B_{\eta_j}$ and $B_{\rho_j}=\rho_j^{m'_j}\equiv c_{\eta_j} B_{\eta_j}$ (except, 
say $B_{\rho_j}$). \medskip  

Hence, we have described bases of the indecomposable projective $\hat{\La}$-modules. Moreover, the socle 
elements in $\La$ induce socles of projective $\hat{\La}$-modules, so in particular, we conclude that 
$\hat{\La}$ is a weakly symmetric algebra.

\subsection{Symmetrizing form for $\hat{\La}$} \label{hat:symmetric} 

Now, we will show that algebra $\hat{\La}$ is also symmetric, by showing a particular symmetrizing form for 
$\La$ induced from a symmetrizing form for $\La$. \medskip 

Indeed, applying results from \cite[2.11]{Lnk}, we deduce that a $K$-linear form $t:A\to K$ is non-dedenerate and 
symmetric if and only if the induced homomorphism $\Phi:A\to A^*$, $\Phi(a)=t\cdot a$, is a monomorphism of 
($A$-$A$)-bimodules. We will work with symmetrizing $K$-linear form, instead of $K$-bilinear form used in the 
definition of symmetric algebra (these two approaches are equivalent, see \cite[Theorem IV.2.2]{SkY}). \smallskip 

Since $\La$ is symmetric, we can fix a non-degenerate and symmetric $K$-linear form $t$ (associated to the 
bilinear form via $t=(1,-)_\La$). \smallskip 

We want to construct a non-degenerate symmetric $K$-linear form $\hat{t}:\hat{\La}\to K$ on $\hat{\La}$. Clearly, 
it should coincide with $t$ when restricted to projectives $\hat{P}_z$ with $z$ a $2$-vertex of $Q$. We can either 
put $\hat{t}(u)=t(u)$ for any element $u$ from the basis of $\hat{P}_z$, $z$ a $2$-vertex of $Q$, and then extend 
$\hat{t}$ to projectives $\hat{P}_z$ at $1$-regular vertices $z$ of $Q$. Instead, we will define $\hat{t}$ globally. \smallskip 

First, let $\cB_z$ denotes the basis of $\hat{P}_z$ constructed in \ref{hat:bases}. In particular, $\cB_z$ coincides 
with a basis of the projective $\La$-module $P_z$, if $z$ is a $2$-vertex of $Q$, and otherwise, $\cB_z$ consists of 
initial submonomials of $B_\eta$ and $B_{\bar{\eta}}$, say except $B_{\bar{\eta}}$, where $\eta$ is the unique arrow 
in $Q_1$ starting from $z$ and $\bar{\eta}=\xi_i,\mu_i$ or $\rho_j$. \smallskip 

Finally, we define $\hat{t}:\hat{\La}\to K$ which assigns to the coset $u+\hat{I}$ of the path $u\in\cB=\bigcup_{z\in Q_0}\cB_z$ 
the following element 
$$\hat{t}(u+I)=\left\{ \begin{array}{cc} 
c_\alpha^{-1}, & \mbox{if } u=B_\alpha \mbox{ for an arrow }\alpha\in \hat{Q}_1, \\ 
0, & \mbox{otherwise,} 
\end{array} \right.$$ 
and extend linearly. \medskip 

It is easy to see that this is well-defined $K$-linear form, because of the equalities 
$c_\alpha B_{\alpha}=c_{\ba}B_{\ba}$, for any arrow $\alpha\in\hat{Q}_1$, and it is symmetric, since the 
function $c_\bullet$ is constant on $g$-orbits in $\hat{Q}$. Moreover, $\hat{t}(B_\alpha)=t(B_{\alpha})$, 
for any $\alpha\in Q_1$ (and the same for initial submonomials). \smallskip 

It remains to see that $\hat{t}$ is non-degenerate, that is, its kernel contains no non-zero right 
ideal. Due to \cite[Lemma 2.11.4]{Lnk}, it is equivalent to say that the induced homomorphism $\hat{\Phi}:\hat{\La}
\to \hat{\La}^*$, $\hat{\Phi}(u)=\hat{t}\cdot u$, is injective. On the other hand, it follows from the shape of bases 
constructed in \ref{hat:bases}, that any projective module $\hat{P}_z$ is injective, even isomorphic to 
$\hat{I}_z=D(\hat{A}e_z)$ (in particular, $\hat{\La}$ is weakly symmetric). Therefore, to show that $\hat{\Phi}$ is 
a monomorphism, it is enough to see that $\hat{\Phi}(u+I)$ is non-zero for every socle element $u\in\soc(\hat{P}_z)$. 
But this is immediate from the definition of $\hat{t}$ and $\hat{\Phi}$, and we are done.

\subsection{Resolutions of simple $\hat{\La}$-modules} 

In this part, we will show that all simple modules over $\hat{\La}$ are periodic of period $4$. Let $z$ be an arbitrary 
vertex of $Q$. By the assumption, $S_z$ is periodic of period $4$ as a $\La$-module, and we have the following exact 
sequence in $\mod \La$ 
$$\xymatrix{0 \ar[r] & S_z \ar[r] & P_z \ar[r]^{d_3} & P_z^- \ar[r]^{d_2} & P_z^+ \ar[r]^{d_1} & P_z \ar[r] & S_z \ar[r] & 0},  
\leqno{(\dagger)}$$  
where $\ima (d_s) \cong \Omega_\La^s(S_z)$; see Section \ref{sec:2}. The exact sequence $(\dagger)$ is determined by the 
`middle map' $d_2$, whose columns are generators of $\Omega^2_{\La}(S_z)$. We shall prove using 
relations in $\hat{I}$ that there is an analogous exact sequence in $\mod\hat{\La}$ providing a $4$-periodic resolution 
of $S_z$. \smallskip 

To begin with, by the construction each $z$ is a $2$-vertex of $\hat{Q}$, so there are two arrows $\alpha:z\to j$ 
and $\ba:z\to k$, and two arrows $\gamma:x\to z$ and $\gamma^*:y\to z$ in $\hat{Q}$. Hence we always have 
part of projective resolution of $S_z$ in $\mod\hat{\La}$ given as 
$$\xymatrix{ \hat{P}_z^+ \ar[r]^{\hat{d}_1} & \hat{P}_z \ar[r] & S_z \ar[r] & 0}$$ 
with $\ker(\hat{d}_1)\cong \Omega_{\hat{\La}}(S_z)$, where $\hat{P}_z^+=\hat{P}_{j}\oplus \hat{P}_{k}$ and 
$\hat{d}_1=[\alpha \ \ba]$. Due to symmetricity of $\hat{\La}$ (Section \ref{hat:symmetric}), we have part of injective 
resolution of the form (see also \cite[Lemma 4.1]{AGQT}): 
$$\xymatrix{0\ar[r] & S_z \ar[r] & \hat{P}_z \ar[r]^{\hat{d}_3} & \hat{P}_z^-}$$ 
with $\ima(\hat{d}_3)\cong \Omega^{-1}_{\hat{\La}}(S_z)$, where $\hat{P}_z^-=\hat{P}_{x}\oplus \hat{P}_{y}$ and 
$\hat{d}_3=\vec{\gamma \\ \gamma^*}$. We will show in a few steps below that there exists an exact sequence in 
$\mod \hat{\La}$ of the form 
$$\xymatrix{0 \ar[r] & S_z \ar[r] & \hat{P}_z \ar[r]^{\hat{d}_3} & \hat{P}_x\oplus\hat{P}_y \ar[r]^{\hat{d}_2} & 
\hat{P}_j\oplus\hat{P}_k \ar[r]^{\hat{d}_1} & \hat{P}_z \ar[r] & S_z \ar[r] & 0 }\leqno{(\hat{\dagger})}$$ 
with $\ima(\hat{d}_s)\simeq\Omega_{\hat{\La}}^s(S_z)$. Since $\hat{d}_1$ and $\hat{d}_3$ are known, we will only 
specify the `middle map' $\hat{d}_2$ in $\mod\hat{\La}$, i.e. a homomorphism $\hat{d}_2:\hat{P}^-_z\to\hat{P}^+_z$ 
in $\mod\hat{\La}$ with $\ker(\hat{d}_2)=\ima(\hat{d}_3)$ and $\ima(\hat{d}_2)=\ker(\hat{d}_1)$. \smallskip 

We divide the rest part into two major cases: $z$ is a $2$-vertex and $z$ is a $1$-vertex (of $Q$). In each 
case, we construct the correct `middle map' $\hat{d}_2$ to get the exact sequence $(\hat{\dagger})$. The 
following table provides a handy summary of the content in both cases. \medskip 

\begin{tabular}{l|l|l|l}
\(\textsc{Case I} \) & \({\it Step 1}\) & \({\it Step 2}\) & \({\it Step 3}\) \\ 
\(\mbox{z is a $2$-vertex}\) & \(\mbox{$z$ is of type A}\) & \(\mbox{$z$ is of type B ($k$ is $1$-regular)} \) & \(\mbox{$z$ is of type C}\) \\ 
 &  & 
\(\mbox{(a) $\hat{d}_2$ for $k$ in block of type $V_2$} \) & 
\(\mbox{(a) $j\in V_2$ and $k\in V_2$} \) \\ 
\(\xymatrix@R=0.3cm@C=0.4cm{x\ar[rd]^{\gamma}&&j\\&z\ar[ru]_{\alpha}\ar[rd]_{\ba}&\\y\ar[ru]^{\gamma^*}&&k}\) & 
\(\xymatrix@R=0.1cm@C=0.2cm{&\ar[rd]&&\\ 
x\ar[rdd]&&j\ar[ru]\ar[rd]&\\ 
&&&\\
&z\ar[ruu]\ar[rdd]&&\\ 
&&& \\ 
y\ar[ruu]&&k\ar[ru]\ar[rd]& \\ &\ar[ru]&&}\) & 

\(\xymatrix@R=0.3cm@C=0.2cm{&&&&&&&&& \\&&& & j\ar[lu]\ar[ru]& && k \ar[rd] && \\
&&& \ar[ru]&&&z\ar[llu]\ar[ru]&  & \circ \ar[ld]\\ &&& &x\ar[rru]&  &&y\ar[lu]&& }\) & 

\(\xymatrix@R=0.3cm@C=0.2cm{& &&&&\\&  &j\ar[ld]&& y\ar[ld]& \\
& \circ\ar[rd]&& z \ar[lu]\ar[rd]&&\circ\ar[lu] \\& & x\ar[ru]&&k\ar[ru]&}\) \\
\cline{3-4}
\( \) & \( \) & \(\mbox{(b) $\hat{d}_2$ for $k$ in block of type $V_1$}\) & \(\mbox{(b) $j\in V_2$ and $k\in V_1$}\) \\ 
\(\) & \(\) & 
\(\xymatrix@R=0.2cm@C=0.3cm{&&  &j&& & \\ &&
&& z \ar[lu]\ar@<+0.4ex>[rr]&&k=y\ar@<+0.4ex>[ll] \\&&
 & x\ar[ru]&&&}\) & 
\(\xymatrix@R=0.3cm@C=0.2cm{&  &j\ar[ld]&& & \\
& \circ\ar[rd]&& z \ar[lu]\ar@<0.4ex>[rr]&&k=y \ar@<0.4ex>[ll]\\& & x\ar[ru]&&&}\) \\ 
\cline{3-4}
\(\) & \(\) & \(\) & \(\mbox{(c) $j\in V_1$ and $k\in V_1$} \) \\ 
\(\) & \(\) & \(\) & \(\xymatrix@C=0.2cm{j=x\ar@<+0.4ex>[rr]&&\ar@<+0.4ex>[ll] z\ar@<+0.4ex>[rr]&&k=y\ar@<+0.4ex>[ll]}\) \\  
\hline 
\(\textsc{Case II} \) & \(\mbox{\textit{Step 1.} $z=c_i$ or $d_i$}\) & \(\mbox{\textit{Step 2.} $z=y_i$ }\) & \(\) \\ 
\(\mbox{$z$ ia a $1$-vertex}\) & 
\( \xymatrix@R=0.2cm@C=0.2cm{ &\ar[ld]z& \\ j\ar[rd]&& x \ar[lu] \\ & \bullet \ar[ru]&} \) & 
\(\xymatrix@C=0.2cm@R=0.3cm{&& &&&\circ \\&& z\ar@<+0.4ex>[rr]&& j=x \ar@<+0.4ex>[ll]\ar[ru]\ar[rd] & \\
&& &&& \circ }\) & 
\(\) 
\end{tabular} \medskip 

\textsc{Case I.} In the first case, we assume that $z$ is a $2$-vertex of $Q$. Then arrows $\gamma,\gamma^*$ ending 
at $z$ and arrows $\alpha,\ba$ starting at $z$ are arrows of both $\hat{Q}$ and $Q$, so in particular, we have 
$P_z^-=P_x\oplus P_y$, $P_z^+=P_j\oplus P_k$, and $d_1=[\alpha \ \ba]$, $d_3=\vec{\gamma \\ \gamma^*}$ 
are given by the same formulas as $\hat{d}_1$ and $\hat{d}_3$, respectively, but viewed as homomorphism of modules 
over different algebras. We proceed in three steps below, according to the type of vertex $z$. \medskip 

{\it Step 1}. First assume $z$ is a vertex of type A, i.e. both $x,y$ are $2$-vertices. It follows from the description 
of the local shape of $Q$ around $1$-vertices, that also $j,k$ are $2$-vertices; see Proposition \ref{prop:4.3}. Then the indecomposable 
summands of modules $\hat{P}_z^{\pm}$ have the same bases corresponding summands of modules $P_z^{\pm}$. It follows 
that images and kernels of maps $d_s$ are the same as images and kernels of induced maps $\hat{d}_s$ between projective 
$\hat{\La}$-modules and therefore, it is sufficient to take $\hat{d}_2=d_2$ (see also arguments presented in Section 
\ref{hat:bases}). Hence we are done in this step. \medskip 

{\it Step 2}. Now, let $i$ be a vertex of type B, say $y$ is a $1$-vertex and $x$ is a $2$-vertex  (the second case 
is analogous). Then $k$ is also a $1$-vertex and vertices $y,k$ belong to a block of type V$_2$ or V$_1$ (in the latter 
case $y=k$). In particular, also $j$ is a $2$-vertex. \smallskip 

(a) Let $k,y$ be vertices in a block of type V$_2$. Without loss of generality, we assume that $\ba=\alpha_i$, for some 
$i\in\{1,\dots,p\}$. Then also $\gamma^*=\delta_i$ and $\alpha_i\beta_i\nu_i$ is involved in the relation 
$\rho_1^{(i)}\in I$. For simplicity, we assume that all parameters $c_\alpha$ are $1$ (it does not affect the 
ranks of matrices used to compute required dimensions). 

There is a relation $\alpha f(\alpha)-\ba('A_{\ba})$ in $\hat{I}\cap\hat{I}$, and a relation in $\hat{I}$ of the 
form $\alpha_i\xi_i-A_\alpha=\alpha(-'A_\alpha)+\ba \xi_i$. Therefore, we conclude that the map 
$\hat{d}_2:\hat{P}_x\oplus\hat{P}_y\to\hat{P}_j\oplus\hat{P}_k$ given by the following matrix 
$$\hat{M}=\vec{f(\alpha) & -'A_\alpha \\ -'A_{\ba} & \xi_i }$$ 
satisfies $\hat{d}_1\hat{d}_2=0$. Using properties desribed in Section \ref{summary}, we also 
get $\hat{d}_2\hat{d}_3=0$. In other words, we have $\ima(\hat{d}_3)\subseteq \ker(\hat{d}_2)$ 
and $\ima(\hat{d}_2)\subseteq \ker(\hat{d}_1)$. We will show that these inculsions are equalities. \smallskip 

Consider arbitrary element $\vec{X \\ Y}$ in $\ker(\hat{d}_2)$, where $X\in \hat{P}_x$ and $Y\in\hat{P}_y$ are 
expressed in terms of bases as follows 
$$X=\sum_{s\geqslant 0}\lambda_s[B_{\gamma}]_s+\sum_{s\geqslant 1}\lambda'_s[B_{\bar{\gamma}}]_s \quad\mbox{and} 
\quad Y=\sum_{s\geqslant 0}\sigma_s[B_{\delta_i}]_s+\sum_{s\geqslant 1}\sigma'_s[B_{\mu_i}]_s$$ 

Rewriting $\hat{M}\cdot \vec{X \\ Y}=0$, we obtain the following two equations 
$$f(\alpha)X=\ 'A_\alpha Y\leqno{(1)}$$ 
$$\xi_iY=\ 'A_{\ba}X\leqno{(2)}$$ 

Now, observe that $f(\alpha)\gamma=A_{g(\alpha)}$, and this element belongs to the second socle of $P_j$. Hence 
$f(\alpha)\gamma g(\gamma)\in\soc(\La)\cap(e_j\La e_k)$ is zero. Similarly, $'A_{\alpha}\mu_i$ has the form 
$\cdots \nu_i\mu_i=\cdots A_{\overline{\nu_i}}\in e_j\hat{\La} e_k$, thus $'A_\alpha[B_{\mu_i}]_s=0$, for 
$s\geqslant 2$, by the zero relations in $\hat{I}$, and $'A_\alpha\mu_i=0$ except $'A_\alpha=\nu_i$, and then 
it is $A_{\overline{\nu_i}}=A_{f(\alpha)}$, and $B_i$ is critical with $\alpha=g^{-1}(\nu_i)$. \smallskip 

Consequently, equation (1) gives   
$$\lambda_0 f(\alpha) +\lambda_1f(\alpha)\gamma + \sum_{s\geqslant 1}\lambda'_s[B_{f(\alpha)}]_{s+1}= 
\sigma_0 A'_{g(\alpha)}+\sigma_1 A_{g(\alpha)}+\sigma_2 B_{g(\alpha)}$$ 
in case when $B_i$ is non-critical (otherwise, the right hand side has an additional summand $\sigma'_1A_{f(\alpha)}$). 
Hence we obtain the following linear system 
$$\left\{ \begin{array}{l} 
\sigma_1=\lambda_1, \ \lambda'_{m_{f(\alpha)}n_{f(\alpha)}-1}c_{g(\alpha)}=\sigma_2c_{f(\alpha)}, \\ 
\sigma_0=\lambda_0=\lambda_1'=\dots=\lambda'_{m_{f(\alpha)}n_{f(\alpha)}-2}=0 
\end{array} \right.\leqno{(*)}$$ 
or 
$$\left\{ \begin{array}{l} 
\sigma_1=\lambda_1, \ \lambda'_{m_{f(\alpha)}n_{f(\alpha)}-1}c_{g(\alpha)}=\sigma_2c_{f(\alpha)}, \\ 
\sigma_0=\lambda_0=\lambda_1'=\dots=\lambda'_{m_{f(\alpha)}n_{f(\alpha)}-3}=0\mbox{ and }
\lambda'_{m_{f(\alpha)}n_{f(\alpha)}-2}=\sigma'_1 \end{array} \right.\leqno{(*)}$$ 
when $B_i$ is critical. 

On the other hand, we have $('A_{\ba})\bar{\gamma}=A_{\beta_i}'\bar{\gamma}$, and this path is ending 
with a path of the form $\theta f(\theta)$, where $f(\theta)=\bar{\gamma}$. We know that this path 
belongs to the second socle, so $A'_{\beta_i}\bar{\gamma}\in\soc(\La)$. It follows that 
$A'_{\beta_i}\bar{\gamma}=0$, since the path is not a cycle. As a result, using relations 
$\xi_i\delta_i=A_{\beta_i}$ and $\xi_i\delta_ig(\delta_i)=0$, we rewrite the equation (2) 
in the form 
$$\sigma_0\xi_i+\sigma_1 A_{\beta_i}+ \sum_{s=1}^{2m_i-1}\sigma'_s[B_{\xi_i}]_{s+1}=
\lambda_0A'_{\beta_i}+\lambda_1 A_{\beta_i}+\lambda_2 B_{\beta_i},$$ 
which gives the following linear system 
$$\left\{ \begin{array}{l} 
\sigma_1=\lambda_1, \ \sigma'_{2m_i-1}c_{\beta_i}=\lambda_2c_{\xi_i}, \\ 
\lambda_0=\sigma_0=\sigma'_1=\dots=\sigma'_{2m_i-1}=0. 
\end{array} \right.\leqno{(**)}$$ \smallskip 

Combining $(**)$ and $(*)$, we obtain (in both cases) a system of linear equations, whose matrix is of rank 
$m_{\bar{\gamma}}n_{\bar{\gamma}}+2m_i+1=\dim_K \ima(\hat{d}_2)$. Therefore, its nullspace has dimension 
$$\dim_K \ker(\hat{d}_2)=m_\gamma n_\gamma + m_{\bar{\gamma}}n_{\bar{\gamma}} + m_{\delta_i}n_{\delta_i}+ 2m_i - 
(m_{\bar{\gamma}}n_{\bar{\gamma}}+2m_i+1)=m_\gamma n_\gamma+m_{\delta_i}n_{\delta_i}-1=$$ 
$$m_{\ba}n_{\ba}+m_\alpha n_\alpha -1=\dim_K\hat{P}_z-1=\dim_K\Omega^{-1}_{\hat{\La}}(S_z)=\dim_K \ima(\hat{d}_3).$$ 
This shows that $\ker(\hat{d}_2)=\ima(\hat{d}_3)$. \medskip 

Finally, we can similarly prove that $\ima(\hat{d}_2)=\ker(\hat{d}_1)$, by comparing dimensions. Indeed, for every 
element $\vec{X\\Y}\in\hat{P}_j\oplus\hat{P}_k$ one can write in terms of bases: 
$$X=\sum_{s\geqslant 0}\lambda_s[B_{f(\alpha)}]_s+\sum_{s\geqslant 1}\lambda'_s[B_{g(\alpha)}]_s \quad\mbox{and} 
\quad Y=\sum_{s\geqslant 0}\sigma_s[B_{\beta_i}]_s+\sum_{s\geqslant 1}\sigma'_s[B_{\xi_i}]_s.$$ 
Then $\hat{d}_1(\vec{X\\Y})=0$ is equivalent to $\alpha X+\ba Y=0$, which gives the following equation 
$$\lambda_0\alpha+\lambda_1A_{\ba}+\sum_{s\geqslant 1}\lambda'_s[B_\alpha]_{s+1}+
\sum_{s\geqslant 0}\sigma_s[B_{\ba}]_{s+1}+\sigma'_1A_\alpha=0.$$ 

Consequently, we get the following system of linear equations  
$$\left\{ \begin{array}{l}  
\lambda'_{m_\alpha n_\alpha-2}+\sigma'_1=0, \ \sigma_{m_{\beta_i}n_{\beta_i}-2}+\lambda_1=0, \ 
\sigma_{m_{\beta_i}n_{\beta_i}-1}c_{\ba}+\lambda'_{m_\alpha n_\alpha-1}c_\alpha=0 \\
\lambda_0=\lambda'_1=\dots=\lambda'_{m_\alpha n_\alpha-3}=0, \ 
\sigma_0=\sigma_1=\dots=\sigma_{m_{\beta_i}n_{\beta_i}-3}=0 
\end{array} \right.$$ 
whose matrix is of rank is $m_\alpha n_\alpha + m_{\beta_i}n_{\beta_i}-1$, and hence, we obtain that 
$$\dim_K \ker(\hat{d}_1)=m_{\bar{\gamma}}n_{\bar{\gamma}}+2m_i+1,$$ 
which is exactly the dimension of $\ima(\hat{d}_2)$. \smallskip 

This provides $\hat{d}_2$ in case (a). \medskip 

In case (b) and all further cases we follow the same strategy, but we omit the technical details of computations. 
We only give appropriate $\hat{d}_2$ such that $\hat{d}_1\hat{d}_2=\hat{d}_2\hat{d}_3=0$ and provide a list 
relations needed to compute dimensions of $\ker(\hat{d}_2)$ and $\ker(\hat{d}_1)$ and confirm the remaining 
inclusions. \medskip 

(b) If $k=y$ is a $1$-vertex in a block of type V$_1$, say $\ba=\ve_i$, for $i\in\{1,\dots,q\}$, then 
the middle map $\hat{d}_2:\hat{P}_x\oplus \hat{P}_{y_i}\oplus \hat{P}_j\to \hat{P}_{y_i}$ is given by the 
matrix 
$$\vec{f(\alpha) & -'A_\alpha \\ -'A_{\ba} & \rho_i}.$$ 
To compute dimensions of $\ker(\hat{d}_2)$ and $\ker(\hat{d_1})$ we need to use the following relations 
in $\hat{\La}$: 
$$f(\alpha)\gamma g(\gamma)=0=\alpha f(\alpha)\gamma, \ \ve_j\rho_i^2, 
'A_\alpha\rho=0, \ \mbox{ and } \quad A'_{\eta_i}\bar{\gamma}=0, $$ 
The first two relations follow, since $\alpha f(\alpha)$ and $f(\alpha)\gamma=f(\alpha)f^2(\alpha)$ are 
contained in the second socle, so we get paths $f(\alpha)\gamma g(\gamma)$ and $\alpha f(\alpha)\gamma$, 
which belong to the socle, but are not cycles. Hence they must be zero. Next, observe that $'A_\alpha=A_{g(\alpha)}'$ 
is of length $\geqslant 3$, so $'A_\alpha\rho$ must be zero, because it contains a subpath 
$\gamma g(\gamma)\rho_i=\eta_i^*\ve_i\rho_i$, which is zero in $\hat{\La}$ (it is a relation in $R_2$). 
Similarily, the path $'A_{\ba}=A_{\eta_i}'$ has length $\geqslant 3$, therefore, we obtain that that 
the path $A_{\eta_i}'\bar{\gamma}$ is zero, because it contains a subpath of the form 
$g^{-1}(\gamma)\bar{\gamma}=g^{-1}(\gamma)f(g^{-1}(\gamma))$ lying in the second socle. \medskip

This finishes the second step in \textsc{Case I}. 

\medskip 

{\it Step 3}. In the last step of this case, we assume that both $x,y$ are $1$-regular in $Q$.  
Here there are three subcases, depending on the type of blocks containing $x,j$ and $y,k$ (up to labelling). 

(a) Suppose first that both $j,x$ and $y,k$ are contained in blocks of type $V_2$, equivalently, $j\neq x$ and 
$k\neq y$. In this case, also $j$ and $k$ are $1$-regular, and without loss of generality, we may assume that 
$\ba=\alpha_i$ and $\alpha=\alpha_j$, for some $i,j\in\{1,\dots,p\}$. \smallskip 

Using relations involving $\alpha_i\xi_i$ and $\alpha_j\xi_j$, we deduce as in Step 2, that the map 
$\hat{d}_2:\hat{P}_x\oplus\hat{P}_y\to \hat{P}_j\oplus\hat{P}_k$ given by the matrix 
$$\hat{M}=\Large{\vec{\xi_j & -'A_\alpha \\ -'A_{\ba} & \xi_i}}$$ 
satisfies $\hat{d}_2\hat{d}_3=\hat{d}_1\hat{d}_2=0$. 

Using the zero relations $\xi_j\delta_j\bar{\alpha}_j=\beta_i^*\nu_i\mu_i=0$ in $\hat{\La}$ and 
the identities $'{A}_{\alpha_j}\mu_i='{A}_{\alpha_i}\mu_j=0$ (see also Step 2), one can compute that the 
dimension of $\ker(\hat{d}_2)$ is $m_\alpha n_\alpha + m_{\ba}n_{\ba}-1$, which coincides with 
$\dim_K \ima(\hat{d}_3)=\dim_K\hat{P}_z-1$. Additionally, we deduce from the computed rank that 
$\dim_K \ima(\hat{d_2})$ is $2m_i+2m_j+1$. Using the zero relations in $\hat{I}$, one can check that 
$$\dim_K \ker(\hat{d}_1)=2m_j+2m_i+1,$$ 
which is exactly the dimension of $\ima(\hat{d}_2)$, so that the sequence $(\hat{\dagger})$ is exact. \medskip 

(b) Let one pair of vertices, say $j,x$, is contained in a block of type V$_2$, say $\alpha=\alpha_j$, 
$j\in\{1,\dots,p\}$, and $k=y$ in a block of type V$_1$, $\ba=\ve_i$, $i\in\{1,\dots,q\}$. Then the correct 
`middle map' $\hat{d}_2:\hat{P}_{d_j}\oplus \hat{P}_{y_i}\to \hat{P}_{c_j}\oplus\hat{P}_{y_i}$ is given by the following matrix $$\hat{M}=\Large{\vec{\beta_j\nu_j & -'A_\alpha \\ -'A_{\ba} & \rho_i}}.$$  
The relations in $\hat{I}$ needed to compute $\dim_K\ker(\hat{d}_2)$ and $\dim_K \ker(\hat{d}_1)$ are  
$$\beta_j\nu_j\delta_j, \ \beta_j^*\nu_j\mu_j, \ \eta_i^*\ve_i\rho_i, \ \rho_i\eta_i\bar{\ve}_i, \ 
\alpha_j\xi_j\mu_j, \ \ve_i\rho_i^2,$$ 
and $\beta_j\nu_j\mu_j=B_{\xi_j}=(\xi_j\mu_j)^{m_j}$ is in the socle.  

\medskip 

(c) If both $j,x$ and $y,k$ belong to blocks of type V$_1$, then $Q$ is the triangle quiver $Q^T$, which is 
excluded in this part of the proof. We omit this case, treated separately in the final part \ref{subsec:6.5}.   

\bigskip 

\textsc{Case II.} Eventually, it remains to consider the case when $z$ is a $1$-vertex of $Q$. Then 
$\hat{P}_z^+$ has dimension $m_\alpha n_\alpha+2m$ or $m_\alpha n_\alpha+ m+1$, where $m=m_i$ or $m'_i$. In this 
case, both $j$ and $x$ are $2$-regular vertices of $Q$, so in particular, we have 
$\dim_K \hat{P}_j=\dim_K P_j$ and $\dim_K\hat{P}_x=\dim_KP_x$ (see Section \ref{hat:bases}). Here, we 
have two possibilities, depending on the type of block containing $z$. \medskip 

{\it Step 1.} Assume first that $z$ belongs to a block $B_i$ of type V$_2$, $z=c_i$; for $z=d_i$ computations are 
analogous. In this case, we have $\hat{d}_1=[\beta_i \ \xi_i]$ and $\hat{d}_3=\vec{\alpha_i \\ \mu_i}$ . \smallskip 

Using the two commutativity relations in $R_1$ starting from $z=c_i$ we can define the map 
$$\hat{d}_2: \hat{P}_{a_i}\oplus\hat{P}_{d_i} \to \hat{P}_{b_i}\oplus \hat{P}_{d_i}$$ 
given by the matrix $\vec{-{'A}_{\beta_i} & \nu_i \\ \delta_i & -(\mu_i\xi_i)^{m_i-1}}$ satisfying 
$\hat{d}_1\hat{d}_2=\hat{d}_2\hat{d}_3=0$. \medskip 

Now, recall that $A_{\bar{\nu_i}}'\bar{\alpha}_i='A_{\beta_i}\bar{\gamma}=0$ as in Step 1a) in \textsc{Case I} 
and all paths of length $3$ around the block (in $Q$) containing $z$ are zero in $\hat{\La}$, by the assumption 
on weights; see also \ref{hat:construction}. These together with the zero relations in $R_2$ imply that 
$$\dim_K \ker(\hat{d}_2)=m_{\beta_i} n_{\beta_i} + 2m_i-1,$$ 
which coincides with $\dim_K\hat{P}_z-1=\dim_K \ima(\hat{d}_3)$, and 
$$\dim_K\ker(\hat{d}_1)=m_{\nu_i} n_{\nu_i} + m_{\ba_i}n_{\ba_i}+1,$$ 
which is exactly $\dim_K \ima(\hat{d}_2)$. \medskip  

{\it Step 2.} Finally, let $z=y_i$, $i\in\{1,\dots,q\}$, belong to a block of type V$_1$. In this case, we have 
$\alpha=\eta_i$, $\gamma=\ve_i$, and $\ba=\gamma^*=\rho_i$.  The map 
$\hat{d}_2:\hat{P}_{y_i}\oplus \hat{P}_{x_i}\to \hat{P}_{x_i}\oplus \hat{P}_{y_i}$ is given by the following 
matrix 
$$\vec{ \ve_i & -'A_{\alpha} \\  -\rho_i^{m'_i-2} & \eta_i }$$  
This is the correct `middle map', because of the relations $\ve_i\rho_i^2=\rho_i^2\eta_i=0$ and 
$A'_{\bar{\ve}_i}\bar{\ve_i}=0$. The latter relation follows as in part b) above, since we have 
a proper subpath $\eta_i^* f(\eta_i^*)$ of $A'_{\bar{\ve}_i}\bar{\ve}_i$, which belongs to the 
second socle of $\hat{\La}$.

\subsection{Tameness}\label{hat:tame} 

In this part, we will show that $\hat{\La}$ is also tame. General strategy is to find a degeneration of 
$\hat{\La}$ to a (symmetric) special biserial algebra $\hat{B}$, which is known to be tame (see \cite{WaWa}), 
and then $\hat{\La}$ is tame, due to results of Geiss \cite{Geiss}. \medskip 

To do so, we will define an algebraic family of algebras 
$\hat{\La}(t)$, $t\in K$, in the varitey $\alg_d(K)$ of $d$-dimensional $K$-algebras, such that for any arrow 
$\alpha\in Q_1$, the following relations hold in $\hat{\La}(t)=K\hat{Q}/\hat{I}(t)$:   
\begin{enumerate}
\item[(1)] $\alpha f(\alpha)=c_{\ba} t^{v(\alpha)} A_{\ba}$, where $v(\alpha)\in\bN_{\geqslant 1}$, 
\item[(2)] $c_{\alpha} B_{\alpha}=c_{\ba}B_{\ba}$, and 
\item[(3)] all zero relations from $I$. 
\end{enumerate} 

We note that the permutation $g$ is defined on $Q$ and it is extended to $\hat{Q}$ by adding orbits of length 
$2$ or $1$ containing arrows $\xi_i,\mu_i$ or $\rho_j$. On the other hand, permutation $f$ is defined only 
partially on $Q$ (not determined for arrows with $1$-regular target), but it can be extended to a permutation 
of arrows in $\hat{Q}$. Then $\hat{Q}$ is a $2$-regular quiver, and by construction, we have two permutations 
$f$ and $g=\bar{f}$ of $\hat{Q}_1$, such that $f^3$ is identity on $\hat{Q}_1$, and for any $\alpha\in\hat{Q}_1$ the path 
$\alpha f(\alpha)$ is involved in a minimal relation of $\hat{I}$ of the form 
$$\alpha f(\alpha)=c_{\ba} A_{\ba}.\leqno{(q)}$$

Suppose for a moment that we have an algebraic family $\hat{\La}(t)$, $t\in K$, as above, with $v:\hat{Q}_1\to\bN$ 
with $v(\alpha)\geqslant 1$, for all $\alpha\in\hat{Q}_1$. It follows that algebra $\hat{\La}(0)$ is a degeneration 
of $\hat{\La}=\hat{\La}(1)$ (in geometric sense), 
so in particular, we have $\dim_K\hat{\La}(0)=\dim_K\hat{\La}$. But then for any $\alpha\in \hat{Q}_1$, there is 
a zero-relation $\alpha f(\alpha)\in \hat{I}(0)$, while $\alpha g(\alpha)\notin \hat{I}(0)$, since otherwise, 
we would get a dimension jump $\dim_K\hat{\La}(0)<\dim_K \hat{\La}$. In this case, we conclude that $\hat{\La}(0)$ 
is a special biserial algebra. \smallskip 

Now, we will construct an algebraic family with $v(\alpha)\geqslant 1$. Take a degree function 
$u:\cP_{\hat{Q}}\to\bN^*$ on $\hat{Q}$, that is, take arbitrary $u(\alpha)\geqslant 1$ for arrows 
$\alpha\in\hat{Q}_1$, and extended additively for paths $\mu=\alpha_1\dots\alpha_r\in\cP_{\hat{Q}}$: 
$u(\mu)=\sum_{i=1}^r u(\alpha_i)$. For a fixed $t\in K^*$, we define 
a homomorphism $f_t:\hat{\La}\to K\hat{Q}$ given for arrows $\alpha\in\hat{Q}_1$ as $f_t(\alpha)=t^{u(\alpha)}\alpha$, 
and extended to products and linear combinations. Then we have an induced isomorphism of algebras 
$\vf_t:\hat{\La}\to\hat{\La}(t)$, 
where $\hat{\La}(t):=K\hat{Q}/\hat{I}_t$, for $\hat{I}_t=f_t(\hat{I})$. In particular, every relation of type 
$(q)$ in $\hat{\La}$ gives the following generator of $\hat{I}_t$ 
$$\alpha f(\alpha)-c_{\ba}t^{v(\alpha)}A_{\ba}$$ 
where $v(\alpha)=u(A_{\ba})-u(\alpha)-u(f(\alpha))$. Therefore, it is sufficient to find a degree function 
$u$ such that the associated function $v$ satisfies $v(\alpha)\geqslant 1$, for all arrows $\alpha\in\hat{Q}_1$. \medskip 

Recall that function $m:\hat{Q}_1\to\bN^*$ is constant on $g$-orbits, so it induces a function 
$m:\cO(g)\to\bN^*$ on the set $\cO(g)$ of $g$-orbits in $\hat{Q}_1$. Following \cite[Section 6]{WSA-GV} we take 
$M=lcm\{m_{\cO}n_{\cO}; \ \cO\in\cO(g)\}$ and for any arrow we put $q(\alpha)=m_\alpha n_\alpha$. Then take 
the degree function $u$ defined by $u(\alpha)=M/q(\alpha)$, which is also constant on $g$-orbits, so the following holds 
$$u(B_{\ba})=m_{\ba}n_{\ba} u(\ba)=M.$$ 
On the other hand, we have $f^2(\alpha)=g^{-1}(\ba)$, so $u(B_{\ba})=u(A_{\ba}f^2(\alpha))=u(A_{\ba})+u(f^2(\alpha))$. 
As a result, with this $u$ we get 
$$v(\alpha)=u(A_{\ba})-u(\alpha)-u(f(\alpha))=M-u(\alpha)-u(f(\alpha))-u(f^2(\alpha))=$$ 
$$=M\left( 1-\frac{1}{q(\alpha)}-\frac{1}{q(f(\alpha))}-\frac{1}{q(f^2(\alpha))} \right).$$ \medskip 

We will show below that $v(\alpha)\geqslant 1$, for any arrow $\alpha\in\hat{Q}_1$. \smallskip 

Indeed, if $\alpha$ is one of the arrows $\xi_i,\mu_i$ or $\rho_i$, then $m_\alpha n_\alpha=2m_i$ or $m_i'$, so 
$m_\alpha n_\alpha\geqslant 3$, by the assumption. If $\alpha$ is different from $\xi_i,\mu_i,\rho_j$, then 
it is an arrow between $2$-regular vertices, and $\alpha$ is an arrow of $Q$. Therefore, either $t(\ba)$ 
is a $2$-vertex and then $\ba f(\ba)=c_\alpha A_\alpha$ with $A_\alpha$ a path along the 
$g$-orbit of $\alpha$ of length $m_\alpha n_\alpha -1$, or $\ba=\xi_i,\mu_i$ or $\rho_j$ and then also 
$\ba f(\ba)=c_\alpha A_\alpha$ (in $\hat{\La}$). In both cases, $A_\alpha$ is of length $\geqslant 2$, 
since otherwise, we would obtain that $\alpha$ belongs to the square of the Jacobson radical of $\La$ (or $\hat{\La}$), 
which is impossible for an arrow $\alpha\in Q_1\subset \hat{Q}_1$. It follows that $q(\alpha)=m_\alpha n_\alpha\geqslant 3$, 
for all arrows $\alpha\in\hat{Q_1}$. \smallskip 

Finally, observe that $v(\alpha)\leqslant 0$ if and only if 
$\frac{1}{q(\alpha)}+\frac{1}{q(f(\alpha))}+\frac{1}{q(f^2(\alpha))} \geqslant 1$. We know that all denominators 
are $\geqslant 3$, so this can happen only when $q(\alpha)=q(f(\alpha))=q(f^2(\alpha))=3$. If one of the 
orbits is of length $1$, say $n_\alpha=1$ (and $m_\alpha=3$), then $n_{f(\alpha)}=n_{f^2(\alpha)}=3$ and 
$\hat{Q}$ has two vertices, two loops and a $2$-cycle, which contradicts the assumptions ($|Q_0|\geqslant 3$). 
Hence, then $m_\alpha=m_{f(\alpha)}=m_{f^2(\alpha)}=1$ and all the arrows $\alpha,f(\alpha),f^2(\alpha)$ belong 
to $g$-cycles of length $3$. \smallskip 

Because $g=\bar{f}$ and the arrows $\alpha,f(\alpha),f^2(\alpha)$ belong to pairwise distinct $g$-cycles of 
length $3$, we infer that each $\eta\in\{\alpha,f(\alpha),f^2(\alpha)\}$ belongs to a triangle 
$(\eta \ g(\eta) \ g^2(\eta))$ with $t(g^2(\eta))=s(\eta)$ and $f(g^2(\eta))=\bar{\eta}$. By definition, 
we have $g(\eta)=\overline{f(\eta)}$ and $f(g(\eta))=\overline{g^2(\eta)}$. In particular, there exists 
an $f$-triangle consisting of arrows $\eta^*=g^2(f(\eta)):s(g^2(f(\eta)))\to t(\alpha)$, 
$f(\eta^*)=\bar{f(\eta)}=g(\eta):t(\eta)\to t(g(\eta))$, and 
$f^2(\eta^*)=f(g(\eta))=\overline{g^2(\eta)}: t(g(\eta))\to s(g^2(f(\eta)))$. It follows that each $\eta=\alpha,f(\alpha)$ 
or $f^2(\alpha)$ gives two `adjacent' $f$-orbits and $g$-orbits (all of length $3$) given as follows 
$$
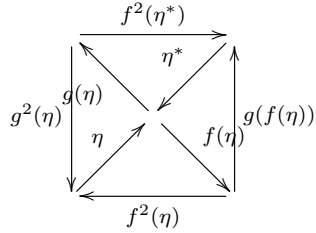
\xymatrix{ \ar[dd]_{g^2(\eta)} \ar[rr]^{f^2(\eta^*)} && \ar[ld]_{\eta^*}\\ & \ar[rd]^{f(\eta)} \ar[lu]^{g(\eta)} & \\ 
\ar[ru]^\eta && \ar[ll]^{f^2(\eta)} \ar[uu]_{g(f(\eta))} }$$ and hence $\hat{Q}$ consists of four $f$-orbits 
As a result, the above arrows exclude all arrows in $\hat{Q}$ (because it is $2$-regular), and then $\hat{Q}$ 
consists of four $g$-cycles of length $3$ (and four $f$-triangles), and consequently, it is the tetrahedral 
quiver mentioned in Section \ref{sec:HSA}. In this case, the quiver $\hat{Q}$ is a $2$-regular quiver with 
$6$ vertices and without $2$-cycles and loops. But then, by construction \ref{hat:construction}, we get that 
$Q=\hat{Q}$ is $2$-regular, a contradiction (we assumed that $Q$ has at least one $1$-vertex).

\subsection{Final conclusions} \label{subsec:6.5}

Summarizing considerations of \ref{hat:construction}-\ref{hat:tame}, it has been proven that for any 
GQT algebra $\La=KQ/I$ with biregular Gabriel quiver $Q$ (having at least three vertices and different from the 
spherical quiver and triangle quiver), there exist algebras $\hat{\La}=\hat{\La}(\hat{m})=K\hat{Q}/\hat{I}$ 
depending on chosen weights $\hat{m}=\{m_1,\dots,m_p,m_1',\dots,m'_q\}$ such that $\hat{\La}({\bf 1})\cong\La$, 
where ${\bf 1}=\{m_i=1,m_i'=2\}$, and for $m_i\geqslant 2,m'_i\geqslant 3$, the algebra $\hat{\La}$ is a tame 
symmetric algebra with simple modules periodic of period $4$. Moreover, in this case the Gabriel quiver of 
$\hat{\La}$ is $\hat{Q}$, and it is $2$-regular. In particular, it follows that $\hat{\La}$ is of infinite 
representation type \cite[see Section 10]{WSA}, and consequently, $\hat{\La}$ is a GQT algebra. \smallskip 

Now, applying the main result of \cite{AGQT}, we deduce that $\hat{\La}=\hat{\Lambda}(\hat{m})$ is a weighted 
surface algebra different from a singular tetrahedral algebra or it is the higher tetrahedral algebra. 
The latter cannot happen, since then $Q=\hat{Q}$ is a $2$-regular quiver, a contradiction. Therefore, 
we conclude that  $\Lambda$, by construction, is a weighted surface algebra, as required. 
Moreover, then $Q$ is obtained from a triangulation quiver $\hat{Q}$ by removing virtual arrows and 
$\hat{m}$ is the restriction of the weight function for $\hat{\Lambda}$ to virtual arrows). \medskip 

This proves the main theorem in case $Q$ is different from $Q^S$ or $Q^T,Q^{T'}$. The case $Q=Q^S$ is covered by 
Theorem \ref{prop:spherical}. Hence it remains to see that every GQT algebra $\La=KQ/I$ with $Q=Q^T$ or $Q^{T'}$ is 
a weighted surface algebra (different from the singular triangle algebra). In this case $Q$ is a glueing 
$$\xymatrix{\ar@{.>}@(lu, ld)[]_{\rho'}1\ar@<.35ex>[r]^{\eta'}&\ar@<.35ex>[l]^{\ve'} 2 \ar@<.35ex>[r]^{\ve} & 
\ar@<.35ex>[l]^{\eta} 3 }$$ 
of two blocks of type V$_1$ (then $Q=Q^T$), or a block of type II with a block of type $V_1$ (then $Q=Q^{T'}$). \medskip 

Observe first that in case $Q=Q^{T'}$, we can apply Lemma \ref{lemB}, and conclude that the paths 
$\ve\eta\ve,\eta\ve\eta$ are involved in minimal relations of the form $\rho^{(5)}_i,\rho^{(6)}_i$. Therefore, 
if $\ve\eta\ve\in J^4$, we may assume that also Lemma \ref{gfV1} holds (the only result using $Q\neq Q^{T'}$), 
and then, one can repeat arguments from  \ref{hat:construction}-\ref{hat:symmetric} and \ref{hat:tame}, and 
construct $\hat{\La}$, given by $\hat{Q}=Q\cup\{\rho:3\to 3\}$, with the same properties. To get periodicity of simple 
$\hat{\La}$-modules in this case, we need to choose three middle maps. Indeed, the simple module $S_1$ 
in $\mod\hat{\La}$ admits an exact sequence of the form $(\hat{\dagger})$ with the middle map 
$\hat{d}_2:\hat{P}_1\oplus\hat{P_2} \to \hat{P}_1\oplus\hat{P_2}$ given by the matrix  
$$\hat{M}=\vec{(\rho')^{m'-2} & \eta'\\ -\ve'& -'A_{\eta'}}.$$  
Analogous exact sequence can be constructed for the simple at $1$-regular vertex $3$. Finally, one can check that 
the simple module $S_2$ in $\mod\hat{\La}$ also admits an exact sequence of the form $(\hat{\dagger})$, where 
the middle map $\hat{d}_2:\hat{P}_3\oplus\hat{P_1} \to \hat{P}_3\oplus\hat{P_1}$ is defined by the matrix 
$$\hat{M}=\vec{\rho & -'A_{\ve} \\-'A_{\ve'} & \rho'}.$$  

As a result, in case $Q=Q^{T'}$ and $\ve\eta\ve\in J^4$, there is a GQT algebra $\hat{\La}=\hat{\La}(\hat{m})$ 
with $2$-regular Gabriel quiver such that $\hat{\La}({bf 1})\cong\La$. As before, we conclude that $\hat{\La}$, 
and hence $\La$, must be a weighted surface algebra given by a triangulation quiver with one virtual loop. \medskip 

Now, observe that arguments from the proofs of Lemma \ref{lem:gen} and \ref{lem:2vertex} can be applied 
also in case, when $\ve\eta\ve\notin J^4$. Indeed, if $Q=Q^{T'}$, but $\ve\eta\ve\notin J^4$, then it is 
sufficient to prove that $g^{-1}(\ve)\ve\eta\ve\in J^5$. It follows from the proof of Lemma \ref{gfV1} that 
$\ve\eta\ve\notin J^4$ holds only when $\ve\eta\ve=A_{\ve'}=\ve'\eta'\ve=\Theta_3(\ve')$ (modulo $J^4$), and then 
$$g^{-1}(\ve)\ve\eta\ve\equiv\eta'\ve'\eta'\ve,$$ 
modulo $J^5$. But $\eta'\ve'=\eta' f(\eta')\in J^3$ or it is $\rho^2$ modulo $^3$, and then 
$\eta'\ve'\eta'=\rho'\rho'\eta'\in J^4$, since $\rho'\eta'=\rho'f(\rho')\in J^3$. This completes the proof 
in case $Q=Q^{T'}$. \medskip 

Finally, assume that $Q=Q^T$. If both $\ve\eta\ve,\ve'\eta'\ve'\in J^4$, then one can repeat the above arguments 
and show $\La$ is a WSA given by the triangulation quiver with two virtual loops \cite[see Example 3.4]{WSA-GV}. 
If $\ve\eta\ve,\ve'\eta'\ve'\prec I$, but one is not in $J^4$, say $\ve\eta\ve$, then we can see that 
$\eta'\ve\eta\ve\in J^5$, by considering the radical quotients of $e_1\La$. 
With this, one can prove that projective $\La$-modules have bases as in Lemmas \ref{basis2reg} and \ref{basis1reg}, 
and then construction of $\hat{\La}$ can be repeated. We skip the details. \medskip 

It remains to consider the case when $\ve\eta\ve\nprec I$ or $\ve'\eta'\ve'\nprec I$. The second case is 
analogous, so we proceed only in the first. So assume $\ve\eta\ve\nprec I$, this is equivalent to $\eta\ve\eta \nprec I$ 
(see the proof of Lemma \ref{lemB}, or use the exact sequence for $S_1$). It follows from the arguments in the 
proof of Lemma \ref{lemB} that then $\eta'\ve\eta \prec I$. Since we have a square $(\ve' \ \eta' \ \ve \ \eta)$, 
we conclude from Lemma \ref{lem:3.7} that also the following rotations  
$$\ve\eta\ve', \ \eta\ve'\eta', \ \mbox{and } \ve'\eta'\ve$$
are involved in minimal relations (one could directly see this using the exact sequences for all simples 
$S_1,S_2$ and $S_3$). \medskip 

With this we can continue as before. Namely, we construct a basis for $e_2\La$ using the radical quotients, 
which is straightforward, by the above relations. On can see that $e_2\La$ has a basis consisting of the 
initial submonomials of $B_{\ve} = (\ve\eta)^{a}$ and $(\ve'\eta')^b= B_{\ve'}$ (excluding maybe $B_{\ve'}$). 
Similarly, one can show that $e_1\La$ has a basis consisting of initial submonomials of $B_{\eta'}=(\eta'\ve')^b$ 
together with $\eta'\ve$, and $e_3\La$ has a basis consisting of initial submonomials of $B_\eta=(\eta\ve)^a$ 
together with $\eta\ve'$. One can directly see from the shape of relations that $\La$ is a WSA given 
by a triangulation quiver $\hat{Q}$ being a gluing of two $f$-triangles $(\eta \ \ve' \ \ve'')$ and 
$(\eta' \ \ve \ \eta'')$ such that $\eta'',\ve''$ form a $2$-cycle of virtual arrows (and we have three 
$g$-cycles of length $2$). In particular, 
then $\La$ is a non-singular triangle algebra \cite[see Example 3.3]{WSA-GV}. Note that this gives precisely 
the algebra $Q(3A)_1$ from \cite[page 304]{E1}. \medskip 

Alternatively, one can treat $Q$ as a single block of type $V_2$, with the $2$-vertices $a$ and $b$ identified. 
Then apply construction of $\hat{\La}$ which then gives $\hat{Q}$ as above, and one may show directly that 
$\hat{\La}$ is a GQT algebra following the same strategy as in \ref{hat:bases}-\ref{hat:tame}, and then 
$\hat{\La}$ (so $\La$) must be the reuired WSA.  

\section*{Acknowledgements} 
This research was initiated during the visit of the first named author at the Faculty of Mathematics 
and Computer Science of the Nicolaus Copernicus University (November 2024) and continued during the 
program ``Research in Pairs" of MFO Oberwolfach (February-March 2025). Both the first and the last named author 
would like to thank the the MFO Oberwolfach for their support. Part of the results has been proven  
during the visit of the last named author at the Mathematical Institute in Oxford (April and November 2025). 
All the authors thank both the Mathematical Institute in Oxford and the Faculty of Mathematics and 
Computer Science in Toru\'n for their hospitality.


\begin{thebibliography}{99}


\bibitem{ASS}
  {I.~Assem, D.~Simson, A.~Skowro\'nski},
  {Elements of the Representation Theory of Associative Algebras 1:
  Techniques of Representation Theory},
  {London Mathematical Society Student Texts, vol. {65}},
  Cambridge University Press, Cambridge, 2006.

\bibitem{WSA-SD}
\textit{J. Bialkowski, K. Erdmann, A. Hajduk, A. Skowro\'{n}ski, K. Yamagata,}
Socle equivalences of weighted surface algebras. J. Pure and Appl. Algebra {\bf 226} (2022) 106886. 

\bibitem{CB}
\textit{W. Crawley-Boevey}, On tame algebras and BOC's. 
Proc. London Math. Soc. {\bf 56} (1988), 451--483.

\bibitem{Dro} {\it Y. A. Drozd}, Tame and wild matrix problems, in: Representation Theory II, in: Lecture Notes in Math., vol.832, Springer-Verlag, Berlin, Heidelberg, 1980, pp. 242--258. 

\bibitem{E1} 
\textit{K. Erdmann}, Blocks of tame representation type and related algebras. 
Springer Lecture Notes in Mathematics {\bf 1428} (1990). 

\bibitem{note} {\it K. Erdmann,} A note on representation-finite symmetric algebras. arXiv:2304.11877v1.

\bibitem{E25} {\it K. Erdmann, A. Skowyrski}, A note on spherical algebras. arXiv:2512.14224v1. 

\bibitem{EHS1} 
{\it K. Erdmann, A. Hajduk, A. Skowyrski}, Tame symmetric algebras of period four. Archiv der Math. {\bf 122} (2024), 249--264. 

\bibitem{EHS2} {\it K. Erdmann, A. Hajduk, A. Skowyrski}, Local structure of tame symmetric algebras of period four. 
Submitted to J. Algebra, arXiv:2411.01235[math.RT]. 

\bibitem{perconj} {\it K. Erdmann, A. Skowro\'nski}, Periodicity conjecture for blocks of group algebras, 
Colloq. Math. 138 (2015), 283--294. 

\bibitem{WSA}
\textit{K. Erdmann, A. Skowro\'{n}ski, } Weighted surface algebras.
J. Algebra {\bf 505} (2018), 490--558.

\bibitem{AGQT} 
\textit{K. Erdmann, A. Skowro\'{n}ski, } Algebras
of generalized quaternion type. Adv. Math. {\bf 349}(2019), 1036--1116,


\bibitem{WSA-GV} 
\textit{K. Erdmann, A. Skowro\'{n}ski, } Weighted surface algebras: general version. OWP-2019-07. 
J. Algebra {\bf 544} (2020), 170--227.


\bibitem{HTA} 
\textit{K. Erdmann, A. Skowro\'{n}ski, }
Higher tetrahedral algebras. Algebras and Representation Theory, {\bf 22}(2019), no. 2, 387--406.


\bibitem{HSA} 
\textit{K. Erdmann, A. Skowro\'{n}ski,}  Higher spherical algebras.
Archiv Math. {\bf 114} (2020), 25--39. 


\bibitem{AGDT}
	\textit{K. Erdmann, A. Skowro\'{n}ski, } Algebras
		of generalized dihedral type. Nagoya Math. J. {\bf 240}(2020), 181--236. 

\bibitem{WSA-corr}
\textit{K. Erdmann, A. Skowro\'{n}ski, } Weighted surface algebras: general version, Corrigendum.
		J. Algebra {\bf 569} (2021), 875--889.

\bibitem{BGA}
\textit{K. Erdmann, A. Skowro\'{n}ski,} From Brauer graph algebras to biserial weighted
		surface algebras. J. Algebraic Combinatorics {\bf 51} (2020), 51--88.


\bibitem{Geiss} {\it C. Geiss}, On degenerations of tame and wild algebras, Arch. Math. (Basel) {\bf 64} (1995), 11–16. 





\bibitem{HSS} 
\textit{T. Holm, A. Skowro\'{n}ski, A. Skowyrski,} Virtual mutations of weighted surface algebras. 
J. Algebra 619 (2023), 822--859.


\bibitem{Lnk} {\it M. Linckelmann}, The Block Theory of Finite Group Algebras vol. I, London Mathematical Society Student Text {\bf 91} (2018). Cambridge University Press.

\bibitem{Rin} \textit{C.M. Ringel,}  Tame Algebras. \  Representation Theory I,  Proceeding, ottawa, Carleton University, 1979, Springer Lecture Notes in Mathematics, vol. {\bf 831}, p. 137-287.

\bibitem{Sk}
\textit{A. Skowro\'{n}ski,}
		Selfinjective algebras: \ finite and tame type.  \ Contemp. Math. {\bf 406} (2006), 169--238.


\bibitem{Sky}
\textit{A. Skowyrski,}
Two tilts of higher spherical algebras. Algebr. Represent. Theory 25 (2021), 237--254.

\bibitem{SS}
	\textit{A. Skowyrski, A. Skowro\'{n}ski}, Generalized weighted surface algebras. arXiv:2106.15218v3.

\bibitem{SkY} 
	\textit{A. Skowro\'{n}ski, K. Yamagata,}
Frobenius algebras. I. Basic representation theory. EMS Textbooks in Mathematics. European Mathematical Society (EMS), Zürich, 2011. 

\bibitem{WaWa} {\it B. Wald, J. Waschb\"usch}, Tame biserial algebras, J. Algebra {\bf 95} (1985), 480--500. 

\end{thebibliography}
\end{document}